\documentclass[11pt,a4paper,oneside]{amsart}
\usepackage[utf8]{inputenc}

\usepackage[foot]{amsaddr}

\usepackage{geometry}
\usepackage{amsfonts}
\usepackage{graphicx}
\usepackage{tcolorbox}
\usepackage[toc,page]{appendix}
\usepackage{amsthm, amsmath}
\usepackage{dsfont}
\usepackage{bm}
\usepackage{mathtools}
\usepackage{enumerate} 
\usepackage{algorithm}
\usepackage{algpseudocode}
\usepackage[colorlinks=true,linkcolor=blue,filecolor=green,citecolor=blue]{hyperref}

\DeclareMathOperator*{\argmax}{arg\,max}

\newtheorem{theorem}{Theorem}
\newtheorem{proposition}[theorem]{Proposition}
\newtheorem{lemma}[theorem]{Lemma}

\newtheorem{remark}[theorem]{Remark}
\theoremstyle{definition}
\newtheorem{definition}[theorem]{Definition}

\newtheorem{corollaryinner}{Corollary}[theorem] 
\newenvironment{corollary}[1][]
{
\if\relax\detokenize{#1}\relax
\else
\ifcsname #1-used\endcsname
\expandafter\xdef\csname #1-used\endcsname{\the\numexpr\csname #1-used\endcsname+1}%
\else
\expandafter\gdef\csname #1-used\endcsname{1}%
\fi
\renewcommand{\thecorollaryinner}{\ref{#1}.\csname #1-used\endcsname}%
\fi
\corollaryinner
}
{\endcorollaryinner}

\newcounter{noteMCctr} 

\newcounter{noteZZctr} 

\def\to{\rightarrow}

\def\sN{{\mathbb{N}}}
\def\sP{\mathbb{P}}

\def\sR{{\mathbb R}}

\newcommand{\lf}{\lfloor}
\newcommand{\rf}{\rfloor}

\newcommand{\lc}
{\mathrel{\raise2pt\hbox{${\mathop<\limits_{\raise1pt\hbox
{\mbox{$\sim$}}}}$}}}

\newcommand{\gc}
{\mathrel{\raise2pt\hbox{${\mathop>\limits_{\raise1pt\hbox{\mbox{$\sim$}}}}$}}}

\newcommand{\ec}
{\mathrel{\raise2pt\hbox{${\mathop=\limits_{\raise1pt\hbox{\mbox{$\sim$}}}}$}}}

\def\bb{\begin{equation}} \def\ee{\end{equation}}
\def\bbn{\begin{equation*}} \def\een{\end{equation*}}

\def\beqn{\begin{eqnarray}}  \def\eqn{\end{eqnarray}}

\def\beqnx{\begin{eqnarray*}} \def\eqnx{\end{eqnarray*}}

\def\bn{\begin{enumerate}} \def\en{\end{enumerate}}

\def\bd{\begin{description}} \def\ed{\end{description}}



\makeatletter
\@namedef{subjclassname@2020}{\textup{2020} Mathematics Subject Classification}
\makeatother

\begin{document}

\title[Log regret in an ergodic market making model]{Logarithmic regret in the ergodic Avellaneda--Stoikov market making model}

\author{Jialun Cao$^1$}
\email{Galen.Cao@ed.ac.uk}

\author{David \v{S}i\v{s}ka$^1$}
\email{D.Siska@ed.ac.uk}

\author{Lukasz Szpruch$^{1,2}$}
\email{L.Szpruch@ed.ac.uk}

\author{Tanut Treetanthiploet$^3$}
\email{tanut.t@infinitaskt.com}

\address{$^1$School of Mathematics, University of Edinburgh, United Kingdom}
\address{$^2$The Alan Turing Institute, London, United Kingdom}
\address{$^3$Infinitas by Krungthai, Bangkok, Thailand}

\date{\today}

\begin{abstract}
We analyse the regret arising from learning the price sensitivity parameter $\kappa$ of liquidity takers in the ergodic version of the Avellaneda--Stoikov market making model.
We show that a learning algorithm based on a maximum-likelihood estimator for the parameter achieves the regret upper bound of order $\ln^2 T$ in expectation.
To obtain the result we need two key ingredients.
The first is the twice differentiability of the ergodic constant under the misspecified parameter in the Hamilton--Jacobi--Bellman (HJB) equation with respect to $\kappa$, which leads to a second--order performance gap.
The second is the learning rate of the regularised maximum-likelihood estimator which is obtained from concentration inequalities for Bernoulli signals. 
Numerical experiments confirm the convergence and the robustness of the proposed algorithm. 
\end{abstract}

\keywords{Regret, Online learning, Adaptive control, Ergodic control, Market making, Maximum likelihood estimation}

\subjclass[2020]{Primary 93E35; Secondary 93C40, 93C41, 93E20, 91G80}


\maketitle

\section{Introduction}

Market makers are market participants who are willing to both buy and sell an asset at any time thus providing liquidity.
They aim to make a profit from the spread, i.e. buying at a lower price (bid) and selling at a higher price (ask) at the cost of carrying inventory risk.
While the principle is simple, executing this consistently profitably is not straightforward due to price volatility, various market micro-structure considerations, information asymmetry and other factors. 

Avellaneda and Stoikov~\cite{avellaneda2008high} have proposed a formulation of the market making task as a stochastic control problem within a parsimonious model.
Since then, the framework has been extensively studied and extended to incorporate various additional features, see \cite{gueant2013dealing, labadie2013high, cartea2015algorithmic, cartea2015risk, cartea2017algorithmic, cartea2023automated} and the references therein. 

In this paper we introduce the ergodic formulation of the model. 
We will establish an upper bound on regret of order $\ln^2 T$ arising from having to learn the key unknown parameter online (while executing a strategy) in the ergodic market making model.
In the remainder of the introduction we will briefly introduce the ergodic market making model, the concept of regret, provide a literature review and highlight the main contributions of this paper. 
In Section~\ref{notations and model} we will state all the assumptions and results in detail.

\subsection*{Ergodic formulation of the Avellaneda--Stoikov model}

The model was originally formulated in a finite-time-horizon setting, where the market maker's objective is to maximise the expected profit over a fixed time period.
In this paper we re-formulate the model in an ergodic setting.
To formulate a learning algorithm and its regret in the finite-time-horizon model we would have considered the episodic setting.
That is, the market maker runs with a fixed $\kappa$ until the time $T$ and then liquidates their inventory, updates their estimate of $\kappa$ and starts again. 
This feels unnatural as liquidating the entire inventory at $T$ with a market order would be costly and not the behaviour one would expect. 
It seems more realistic to assume that the market maker is continuously learning the parameter $\kappa$ and updating their strategy based on the new information while managing their inventory risk according to their risk appetite expressed via a quadratic penalty on the inventory and the inventory bounds.
The ergodic formulation allows us to capture learning and regret in this more natural setting.

The market maker places one buy/sell order at distances $\delta^-$, $\delta^+$ from the mid price denoted $S_t$ and updates these continuously as new information arrives.
These are the controls.
On average $\lambda^\pm$ per unit of time buy / sell market orders (orders from liquidity takers) arrive.
These hit the limit order posted by the market maker with probability of $e^{-\kappa \delta^\pm}$.
The system thus has the controlled dynamics given by
\begin{equation*}
\begin{split}
dS_t & = \sigma dW_t, \quad S_0 = s_0\,, \\
dQ_t^{\delta^{\pm}} & =  dN_t^{\delta, -} - dN_t^{\delta, +}, \quad Q_0 = q_0 \,,\\
dX^{\delta^{\pm}}_t & = (S_{t-} + \delta_t^{+}) dN_t^{\delta, +} - (S_{t-} - \delta_t^{-}) dN_t^{\delta, -}, \quad X_0 = x_0\,,
\end{split}	
\end{equation*}
where $(S_t)_{t\geq 0}$ is the exogenous mid-price process, $(Q_t^{\delta^{\pm}})_{t\geq 0}$ is the market maker's inventory and $(X^{\delta^{\pm}}_t)_{t\geq 0}$ is the market maker's cash balance. 
The inventory and cash processes are driven by $N_t^{\delta, \pm}$, two independent Poisson jump processes with intensities $\lambda^\pm e^{-\kappa \delta^\pm}$.   
The market maker wishes to maximise the long-run average reward which sums the earnings and changes to mark-to-market value of their holdings of the risky asset but is subject to a quadratic inventory penalty expressing their risk aversion: 
\begin{equation*} 
J(q, x,S;\delta^{\pm}) = \lim_{T \rightarrow + \infty}  \frac{1}{T} \mathbb{E}_{q,x,S} \bigg[  
\int_0^T \mathrm{d} (X^{\delta^{\pm}}_t  + S_t Q_t^{\delta^{\pm}}) - \phi \int_0^T (Q^{\delta^{\pm}}_t)^2 \, \mathrm{d} t 
\bigg]\,.
\end{equation*}
If the values of all the parameters are known then the market maker can solve the ergodic Hamilton--Jacobi--Bellman (HJB) equation associated to the problem and obtain the optimal strategy in closed form as we show in Section~\ref{Ergodic Analysis for Market Making Model}.
Under the optimal strategy, the market maker's reward, per of unit time, will be given by the ergodic constant
\begin{equation*}
\gamma(\kappa) = \sup_{\delta^\pm} J(x,S, q;\delta^{\pm}).
\end{equation*}

\subsection*{Online learning and regret}
The model parameters are: liquidity takers orders' arrival rates $\lambda^\pm$, the price sensitivity of the liquidity takers $\kappa$, the mid price volatility $\sigma$ (which actually plays no role as the mid price process is a martingale) and the risk aversion $\phi$. 
The market maker chooses their risk aversion and thus it's not a parameter that they would need to learn. 
The liquidity takers orders' arrival rates $\lambda^\pm$ can be observed and learned offline (without participating in the market) since, in the framework of the model where our market maker is assumed to provide a relatively small fraction of the overall liquidity, it is unlikely that the presence of their volume in the market would impact the rate of liquidity taking. 
This leaves $\kappa$ and this is the key parameter. 
Although exchanges may provide market participants with visibility of the order book and message-level  trades execution data, allowing them to estimate $\kappa$ without direct participation, this is insufficient for an accurate estimate for $\kappa$. 
A key challenge is that other market makers will react to the presence of the additional volume placed by our market maker at the distance $\delta^\pm$ thus potentially rendering any offline estimate of $\kappa$ inaccurate.
Indeed, some liquidity providers may choose to place their volume at better price (they want to trade) than the spread given by our market maker while others may wish to place the volume at worse price (they may think our market maker knows something about the price they don't).
The offline estimate of $\kappa$ can of course be used as the initial value in the learning algorithm.

The key parameter to learn online (i.e. while participating in the market) is thus $\kappa$. 
At each time $t\geq 0$ the market maker will have their estimate of the parameter denoted $\kappa_t$ while the true, unknown, value is $\kappa^\ast$. 
They can solve the ergodic control problem and obtain the strategy which would be optimal if $\kappa_t$ would be the true parameter. 
Let us denote this strategy by $\psi^{\kappa_t,\pm}$.

Our aim is to gain asymptotic understanding of the {\em regret} given by
\begin{equation} 
\label{original regret intro}
\mathcal{R}(T) = 
\gamma(\kappa^\ast) T - \mathbb{E}_{q, x, S} \bigg[ \int_0^T \mathrm{d} \big(X^{\psi^{\kappa_t,\pm}}_t  + S_t Q^{\psi^{\kappa_t,\pm}}_t\big) - \phi \int_0^T (Q^{\psi^{\kappa_t,\pm}}_t)^2 \, \mathrm{d} t \bigg]\,.
\end{equation}
This is the difference between the optimal, inaccessible, reward up to time $T$ and the reward the agent gains by following their chosen method of learning.

If the market maker would use a fixed $\kappa \neq \kappa^\ast$ then their expected regret would be roughly $(\gamma(\kappa^\ast)-\gamma(\kappa;\kappa^\ast))T$, i.e. linear. 
Any algorithm which achieves sub-linear regret is learning.
We construct a regularised maximum-likelihood estimator, see~\eqref{derivative of the log-likelihood function} and Algorithm~\ref{alg:cap}, to achieve the expected regret upper bound of order $\ln^2 T$.
See Theorem~\ref{regret theorem}. 

\subsection*{Existing literature}
Before we proceed to discussing online learning let us mention the ``offline'' learning approach in Cartea~\cite{cartea2017algorithmic}. 
There, parameter uncertainty for the finite-time-horizon market making model is accepted and robust controls which take model ambiguity into account are derived. 

Online learning and regret analysis in stochastic control has been studied in the context of adaptive control and reinforcement learning. 
Broadly, there are three relatively distinct areas. 

The first area is discrete and finite space and time Markov decision problems (either discounted or ergodic).
Here regret of order $\sqrt{T}$ is expected in the general setting and with additional structural assumptions regret of order $\ln T$ is achievable, see Auer and Ortner~\cite{auer2006logarithmic}, Auer et al.~\cite{auer2008near} and references therein. 

The second area is still discrete time with linear dynamics and convex cost / concave rewards.  
This makes the setting tractable even in the case of more general state spaces and action spaces.
This is the setting most explored in the literature over the years:
Kumar~\cite{kumar1983optimal},
Campi and Kumar~\cite{campi1998adaptive}, 
Abbasi-Yadkori~\cite{abbasi2011regret},
Abeille and Lazaric~\cite{abeille2018improved}, 
Agarwal et al.~\cite{agarwal2019logarithmic}
Dean et al.~\cite{dean2018regret},
Cohen et al.~\cite{cohen2019learning},
Cassel et al.~\cite{cassel2020logarithmic},
Faradonbeh et al.~\cite{faradonbeh2020adaptive},
Lale et al.~\cite{lale2020explore},
Simchowitz and Foster~\cite{simchowitz2020naive},
Hambly et al.~\cite{hambly2021policy} and undoubtedly some others. 
The theme is again that order $\sqrt{T}$ is achievable and if more can be assumed (e.g. ``identifiability conditions'' which imply ``self-exploration'') then regret upper bound of order $\ln T$ holds. 

Finally, the third which is the continuous-time, in linear-convex framework setting is the least explored.
Guo et al.~\cite{guo2023reinforcement} considers finite-time-horizon linear-convex episodic learning and propose algorithm which achieves order $\sqrt{N \ln N}$ regret (with $N$ being the episode number) under an identifiability assumption.
In Basei et al.~\cite{basei2022logarithmic} where, the episodic learning LQR is studied, regret bound of order $(\ln N)(\ln (\ln N))$ is obtained, again under an identifiability assumption.
Szpruch et al.~\cite{szpruch2024optimal} show that without the identifiability assumption it is possible to balance exploration and exploitation (by adding an entropic regularizer) to achieve order $\sqrt N$ regret. 
In Szpruch et al.~\cite{szpruch2021exploration} this is improved to $\ln^2 N$ by means of establishing stronger (2nd order) regularity result for the dependence of the problem value function on the unknown system parameters. 

Having reviewed existing results, we note that the study of regret in continuous-time ergodic control has been limited. 
Fruit and Lazaric~\cite{fruit2017exploration} derive regret bounds in semi-Markov decision processes (SMDP) within the ergodic setting and show that the regret of order $\sqrt{T}$ is achievable under certain assumptions (e.g. lump sum reward). 
In Gao and Zhou~\cite{gao2024logarithmicregretboundscontinuoustime}, the order of regret is improved to $\ln T$ by focusing on continuous-time Markov decision processes, a more specific case than SMDP. 
This represents a significant step forward, showing that logarithmic regret is achievable in continuous-time ergodic frameworks. 
Nevertheless jump diffusion dynamics and non-linear running rewards required in the Avellaneda--Stoikov model do not fit into the framework of any of the existing papers.
From other results in the literature we see that our result showing $\ln^2 T$ regret is nearly as good as it gets but a question remains whether this is optimal i.e. what is the regret lower bound in this setting. 
The numerical experiment shows regret of order $\ln^2 T$ is a good fit for what we observe, see Figure~\ref{regret}.

\subsection*{Our contributions} 
To the best of authors' knowledge this is the first paper on regret analysis for ergodic control of jump diffusions. 
The control problem we focus on is the ergodic version of the Avellaneda--Stoikov market making model and we show that the expected regret has an upper bound of order of $\ln^2 T$. 

There are three main ingredients which allow us to obtain this result. 
First, we prove existence of and convergence to an invariant measure in the ergodic Avellaneda--Stoikov market making model.
While the well-posedness of the ergodic problem follows mostly from the analysis carried out in Gu{\'e}ant and Manziuk~\cite{gueant2020optimal} the result on existence of and convergence to the invariant measure is new and relies on newly established explicit solution to the ergodic HJB corresponding to our problem. 

Second, we obtain bounds on the second--order derivative of the average earnings per unit time (i.e. the ergodic constant under a misspecified $\kappa$) with respect to the parameter $\kappa$ which has to be learned. This leads to a second-order performance gap in the regret analysis, which is crucial.

Finally, using concentration inequalities for Bernoulli random variables we show that a regularised maximum likelihood estimator yields a high probability bound of order $N^{-1/2}$ on the distance between the true value $\kappa^\ast$ and the estimate $\kappa_N$ obtained after $N$ market orders have arrived.

\section{Main results} 
\label{notations and model}
In this section, we will introduce the ergodic Avellaneda--Stoikov market making model and state the main results of the paper. 

The market maker (agent) proposes bid$(-)$ and ask$(+)$ depths (control) $(\delta^{\pm}_t)_{t \geq 0}$ pegged to the mid-price of an asset and wish to make profit from the spread. The space $(\Omega^{W}, \mathcal{F}^{W}, \mathbb{P}^{W})$ supports a Brownian motion $(W_t)_{t \geq 0}$ that describes the asset mid-price process $(S_t)_{t \geq 0}$, following the dynamics 
\begin{equation} \label{mid-price}
dS_t = \sigma dW_t, \quad S_0 = s_0 \, . 
\end{equation}
Apart from the market maker, there are liquidity takers sending market orders (MOs) at random times. The model assumes that the arrivals of buy$(+)$ and sell$(-)$ MOs, $(M_t^{\pm})_{t\geq 0}$, follow two independent Poisson processes with intensities $\lambda^{+}$ and $\lambda^{-}$ defined on $(\Omega^{M}, \mathcal{F}^{M}, \mathbb{P}^M)$. 
Given two independent IID sequences $U^\pm_i \sim U(0,1)$ defined on $(\Omega^U, \mathcal F^U, \mathbb P^U)$ 
an incoming market buy/sell order trades with the sell/buy volume posted by the market maker when $U^\pm_{M^\pm_t} \geq e^{-\kappa^\pm \delta_t^\pm}$.
The probability space for the model is thus
\begin{equation} \label{probability space}
\begin{split}
(\Omega, \mathcal{F}, \mathbb{P}) = \big( 
\Omega^{W} \times \Omega^{M} \times \Omega^{U},
\mathcal{F}^{W} \otimes \mathcal{F}^{M} \otimes \mathcal{F}^{U},  \mathbb{P}^{W} \otimes \mathbb{P}^{M} \otimes  \mathbb{P}^{U}
\big) \, .
\end{split}
\end{equation}
The filtration is $\mathbb F := (\mathcal F_t)_{t\geq 0}$, where $\mathcal F_t = \sigma(W_r,r\leq t) \vee \sigma(M_r^\pm,r\leq t) \vee \sigma(U^+_{M^+_r},r\leq t) \vee \sigma(U^-_{M^-_r},r\leq t)$.
Let $\underline q \in \mathbb Z^{-}$ and $\bar q \in \mathbb Z^{+}$ denote the market maker's inventory limits. 
The market maker will stop posting buy/sell orders when their inventory is at $\bar q$ and at $\underline q$ respectively.
At other times their strategy is to post at a distance $\delta^{\pm} \in \mathbb R$ from the mid-price $S_t$.
The reason for imposing inventory boundaries is that they reduce an infinite state-space control problem into a finite one, making it computationally tractable by leading to a matrix representation for an explicit solution.
Clearly, the strategy $\delta^\pm$ must be adapted to the filtration $\mathbb F$.
Let $(N_t^{\delta, \pm})_{t\geq0}$ be the controlled counting processes for the agent's filled buy/sell orders, i.e. 
\[
N_t^{\delta,\pm} = N_{t-}^{\delta,\pm} + (M^\pm_t - M^{\pm}_{t-})\mathds{1}_{\left\{U^\pm_{M^\pm_t} \geq e^{-\kappa^\pm \delta_{t-}^\pm}\right\}}\,.
\]
Hence the inventory process $(Q_t)_{t \geq 0}$ of the market maker is
\begin{equation} \label{inventory}
dQ_t^{\delta^{\pm}} =  dN_t^{\delta, -} - dN_t^{\delta, +}, \quad Q_0 = q_0 \, \text{ and } \, \underline{q} \leq Q_t \leq \bar{q} \,.
\end{equation}
Let us write $\Omega^{Q} = [\underline q, \bar q] \cap \mathbb Z$,
so that $Q_t$ takes values in $\Omega^{Q}$ for $t \geq 0$. 
Let $(X_t)_{t \geq 0}$ denote the market maker's cash balance, satisfying 
\begin{equation} \label{cash}
dX^{\delta^{\pm}}_t = (S_{t-} + \delta_t^{+}) dN_t^{\delta, +} - (S_{t-} - \delta_t^{-}) dN_t^{\delta, -}, \quad X_0 = x_0 \, . 
\end{equation}

Let us define the class of admissible policies as
\begin{equation} \label{admissible policies}
\begin{split}
\mathcal{A} = \big\{ (\delta^{\pm}_t)_{t\geq 0} & : \text{$\overline{\mathbb R}$-valued, bounded from below, progressively measurable }\\
& \text{w.r.t. $\mathbb F$ and s.t. for any $T>0$ we have } \mathbb E\int_0^T |\delta^\pm_t|^2\, \mathrm{d}t<\infty  \big\}\,.
\end{split}
\end{equation}

\subsection{The ergodic market making model} 
\label{Ergodic Analysis for Market Making Model}

In this section we will formulate the ergodic control problem, state key results connecting the control formulation with the ergodic HJB equation, provide explicit solution for the ergodic HJB and formulae for the Markovian ergodic optimal controls.

\subsubsection*{Control problem formulation} 

The market maker aims to maximise a long-run average reward of the accumulated PnL with a running inventory penalty. 
The quadratic penalty on running inventory plays a crucial role by providing a continuous incentive to steadily drive the inventory level toward zero. 
This is important in any volatile market, where the market maker seeks to minimise directional exposure to adverse price movements.

Let $J(x, S, q; \delta^{\pm})$ be the ergodic reward functional given by 
\begin{equation} \label{ergodic reward functional}
J(q,x,S;\delta^{\pm}) = \lim_{T \rightarrow + \infty}  \frac{1}{T} \mathbb{E}_{q, x, S} \bigg[  
\int_0^T \mathrm{d} (X^{\delta^{\pm}}_t  + S_t Q_t^{\delta^{\pm}}) - \phi \int_0^T (Q^{\delta^{\pm}}_t)^2 \, \mathrm{d} t 
\bigg]\, ,
\end{equation}
where the notation $\mathbb{E}_{q, x, S} [\cdot]$ represents expectation conditional on $Q_0 = q, X_0 = x, S_0 = S$ and $\phi \geq 0$ is the running inventory penalty parameter. 
For the optimal ergodic control problem, the purpose is to give a characterisation of the optimal long-run average reward, also known as the ergodic constant
\begin{equation*}
\gamma = \sup \Big\{ J(q,x,S;\delta^{\pm}): \delta^{\pm} \in \mathcal{A} \Big\}\,,
\end{equation*}
and to construct an optimal feedback (Markov) control $\psi^\pm$. We will later see that $\gamma$ is indeed independent of $q, x, S$ and thus calling it the ergodic constant is justified. Of course it still depends on all the model parameters, in particular on $\kappa$.

Let us define a running reward function $f: \Omega^{Q} \times \overline{\mathbb R}^2 \to \mathbb{R}$ as 
\begin{equation} \label{f function}
f(q;\delta^{\pm}) = \delta^{+} \lambda^{+} e^{- \kappa^{+} \delta^{+}} + \delta^{-} \lambda^{-} e^{- \kappa^{-} \delta^{-}} - \phi q^2 \, .    
\end{equation}
By~\eqref{mid-price}, \eqref{inventory} and~\eqref{cash}, we have 
\begin{equation*}
\begin{split}
& d(X^{\delta^{\pm}}_t + S_t Q^{\delta^{\pm}}_t) = dX^{\delta^{\pm}}_t + S_t dQ^{\delta^{\pm}}_t + Q^{\delta^{\pm}}_t dS_t + dS_t dQ^{\delta^{\pm}}_t \\
& = (\delta^{+} \lambda^{+} e^{- \kappa^{+} \delta^{+}} + \delta^{-} \lambda^{-} e^{- \kappa^{-} \delta^{-}}) dt 
+ \delta^{+} d\Tilde{N}_t^{\delta, +} + \delta^{-} d\Tilde{N}_t^{\delta, -} + \sigma Q^{\delta^{\pm}}_t dW_t \, , 
\end{split}
\end{equation*}
where $\Tilde{N}_t^{\delta, \pm}$ are independent compensated Poisson processes. As the intensities of $N_t^{\delta^{+}}$ and $N_t^{\delta^{-}}$ would be $0$ whenever $\delta^{+} = + \, \infty$ and $\delta^{-} = + \, \infty$ and otherwise $\delta^{\pm} \in \mathcal{A}$ is clearly square integrable, therefore $\mathbb{E} \big[ \int_0^T \delta^{+} d\Tilde{N}_t^{\delta, +} \big] = 0$ and $\mathbb{E} \big[ \int_0^T \delta^{-} d\Tilde{N}_t^{\delta, -} \big] = 0$. Moreover, $(Q_t^{\delta^{\pm}})_{t\geq0} \in \Omega^Q$ is $\mathcal{F}_t-$adapted and bounded and so $\mathbb{E} \big[ \int_0^T \sigma Q^{\delta^{\pm}}_t dW_t \big] = 0$.
Hence the ergodic market making control problem can be reduced from dimension of 3 to 1 by Fubini's theorem
\begin{equation} \label{reduced ergodic reward functional}
J(q; \delta^\pm) = \lim_{T \rightarrow + \infty}  \frac{1}{T} \mathbb{E}_{q} \Big[  
\int_0^T f(Q^{\delta^{\pm}}_t; \delta^{\pm}) \, \mathrm{d} t 
\Big] \, , \quad  \gamma = \sup \Big\{ J(q;\delta^{\pm}): \delta^{\pm} \in \mathcal{A} \Big\}\, . 
\end{equation}

To analyse the ergodic control problem~\eqref{reduced ergodic reward functional}, some preliminaries are required. 
We start with the existence and uniqueness analysis for the classical market making problem in the discounted finite and infinite-time-horizon settings.  

\subsubsection*{Key results for the discounted finite-time and infinite-time problems} \label{existence and uniqueness chapter}
We first define the unoptimised Hamiltonian function $H: \Omega^{Q} \times \mathbb{R}^2 \times \mathbb{R}^2 \to \mathbb{R}$ and the optimised Hamiltonian function $H: \Omega^{Q} \times \mathbb{R}^2 \to \mathbb{R}$ for the market making model as 
\begin{equation} \label{hamiltonian function}
\begin{split}
H(q, \delta^\pm, \boldsymbol{p}) & = \lambda^{+} e^{-\kappa^{+} \delta^{+}} (p_1 + \delta^{+}) \mathds{1}_{q > \underline{q}} + \lambda^{-} e^{-\kappa^{-} \delta^{-}} (p_2 + \delta^{-}) \mathds{1}_{q < \bar{q}} - \phi q^2 \,,\\
H(q, \boldsymbol{p}) & = \sup_{\delta^{\pm} \in \mathbb R^2} H(q, \delta^\pm, \boldsymbol{p})\,.
\end{split}
\end{equation}

Now we consider the optimal market making problem in the discounted finite-time-horizon setting. We assume that the market maker has a penalty $G(q)$ for any inventory $q \in \Omega^{Q}$ held at the terminal time $T > 0$
\begin{equation} \label{terminal condition}
G(q) =  - \alpha q^2,
\end{equation}
with $\alpha \geq 0$ the terminal inventory penalty parameter. 
Let $v_r(t,q; T)$ be the value function given by 
\begin{equation} \label{reduced finite-time value function}
\begin{split}
v_r(t, q; T) & = \sup_{\delta_u^{\pm} \in \mathcal{A}} \mathbb{E}_{t, q} \bigg[ \int_t^{T} e^{-r(u - t)} f(Q_u;\delta_u^{\pm}) \, \mathrm{d} u + e^{-r(T - t)} G(Q_T^{\delta^{\pm}}) \bigg] \, , 
\end{split}
\end{equation}
where $r \geq 0$ is the discounted factor, the running reward function $f$ is given by~\eqref{f function} and $\mathcal{A}$ denotes the class of admissible policies defined by~\eqref{admissible policies}.
The associated Hamilton–Jacobi–Bellman (HJB) equation to the value function~\eqref{reduced finite-time value function} is 
\begin{equation} \label{finite-time hjb}
0 = \partial_t u(t,q) - r u(t,q) + H \Big(q, (u(t, q') - u(t , q))_{q' \in \{q-1, q+1\}}\Big), \quad \forall q \in \Omega^{Q},
\end{equation}
subject to the terminal condition (\ref{terminal condition}). 

Theorem~\ref{exist and unique for finite-time} provides the existence and uniqueness for the optimal market making problem in the discounted finite-time-horizon setting. The proof is provided in Appendix \ref{proof theorem exist and unique finite}. We also recommend Gu\'{e}ant \textit{et al.} \cite{gueant2020optimal} for the proof of a more general stochastic control problem with a discrete state space. 

\begin{theorem}[Existence and uniqueness for discounted finite-time HJB]
\label{exist and unique for finite-time}
There exists a unique solution $u$ to the HJB equation~\eqref{finite-time hjb} on $t \in (-\infty, T]$ with the terminal condition~\eqref{terminal condition} such that for any $t' > 0$ we have $u \in C^1([-t', T]; \Omega^Q)$. Moreover, $u = v_r$. 
\end{theorem}

It is well known~\cite{cartea2015algorithmic,gueant2016financial} that there is an explicit solution to $v_r$ satisfying~\eqref{reduced finite-time value function} in the case of $r=0$, denoted by $v_0$, given by the following theorem.
\begin{theorem} [Explicit solution of finite-time-horizon model] \label{solution of finite-time-horizon model}
Assume $\kappa^{\pm} = \kappa$ and $r = 0$. Let $\boldsymbol v_0(t; T) = [v_0(t, \bar{q} ;T), v_0(t, \bar{q}-1 ;T), ... , v_0(t, \underline{q} ;T)]^T$ be a $(\bar{q} - \underline{q} + 1)$-dim vector of the  solution to HJB equation (\ref{finite-time hjb}) with terminal condition (\ref{terminal condition}). Let $\boldsymbol z$ be the $(\bar{q} - \underline{q} + 1)$-dim vector with components $z_j = e^{-\alpha \kappa j^2}$ and $\boldsymbol A$ be the $(\bar{q} - \underline{q} + 1)-$square matrix
\scriptsize
\begin{align*}
\boldsymbol A = \begin{bmatrix}
- \phi \kappa \bar{q}^2 & \lambda^{+} e^{-1} & 0 &  ... & &  \\ 
\lambda^{-} e^{-1} & - \phi \kappa (\bar{q}-1)^2 & \lambda^{+} e^{-1 } & ...& & & \\ 
& &  & ... &  & \\
& ... &  \lambda^{-} e^{-1} & - \phi \kappa (\bar{q}-i)^2 & \lambda^{+} e^{-1} & ... & \\
& & & ...  & \\
& & & ... &  \lambda^{-} e^{-1} & - \phi \kappa (\underline{q}+1)^2 & \lambda^{+} e^{-1}\\ 
& & & ... & 0 & \lambda^{-} e^{-1} & - \phi \kappa \underline{q}^2
\end{bmatrix}
\end{align*}
\normalsize
Then the explicit solution is uniquely given by 
\begin{equation*}
\boldsymbol{v}_0(t; T) =  \frac{1}{\kappa} \ln ( e^{(T-t) \boldsymbol{A}} \cdot \boldsymbol{z} ) \, . 
\end{equation*}
\end{theorem}

Let us move to discussing the infinite-time-horizon problem.
The value function, in the discounted infinite-time-horizon setting, is
\begin{equation} \label{reduced infinite-time-horizon model}
v_r(q) = \sup_{\delta^{\pm} \in \mathcal{A}} \mathbb{E}_{q} \Big[
\int_0^{+\infty}e^{-rt} f(Q_t;\delta_t^{\pm}) \, \mathrm{d} t  \Big] \, ,
\end{equation}
where, in this case, the discount factor is strictly positive  $r > 0$. The associated HJB equation for the control problem~\eqref{reduced infinite-time-horizon model} is 
\begin{equation} \label{hjb for discounted infinite-time-horizon model}
0 =  - r u(q) + H \Big(q, (u(q') - u(q))_{q' \in \{q-1, q+1\}}\Big), \quad \forall q \in \Omega^{Q}\,.
\end{equation}

Theorem~\ref{existence of discounted infinite} gives the existence of the solution to the discounted infinite-time-horizon problem, the proof is provided in Appendix~\ref{proof theorem exist infinite}. 
\begin{theorem} [Existence for discounted infinite-time HJB]
\label{existence of discounted infinite}
Let $v_r(\cdot, \cdot; T)$ be the unique solution to the HJB equation (\ref{finite-time hjb}) with the terminal condition (\ref{terminal condition}) and $r > 0$. 
Then for $v_r: \Omega^Q \to \mathbb{R}$ given by~\eqref{reduced infinite-time-horizon model} we have $\forall q \in \Omega^{Q}$ and $\forall t \in \mathbb{R}^{+}$ that 
\begin{equation*}
v_r(q) = \lim_{T \rightarrow + \infty} v_r(t, q; T) \, . 
\end{equation*}
Moreover, $v_r$ is the unique solution to~\eqref{hjb for discounted infinite-time-horizon model}. 
\end{theorem}

\subsubsection*{The ergodic HJB and its connection to the ergodic control problem}
\label{ergodic analysis}
In this section, we analyse the ergodic control problem~\eqref{reduced ergodic reward functional} by considering the asymptotic behaviour of $T \rightarrow + \, \infty$ in the finite-time-horizon model~\eqref{reduced finite-time value function} with $r= 0$ and $r \rightarrow 0$ in the discounted infinite-time-horizon model~\eqref{reduced infinite-time-horizon model}. We prove that $\lim_{r \rightarrow 0} r v_r(q)$ is equal to the ergodic constant $\gamma$ in~\eqref{reduced ergodic reward functional}. Then explicit solutions to the ergodic control problem are derived.

We start with Theorem~\ref{ergodic constant} that analyses the asymptotic behaviour of $r \rightarrow 0$ in the discounted infinite-time-horizon model~\eqref{reduced infinite-time-horizon model}, the proof of which is provided in Appendix~\ref{proof theorem ergodic control problem}. 
\begin{theorem} \label{ergodic constant}
For the value function $v_r$ given by~\eqref{reduced infinite-time-horizon model} there exists a constant $\hat \gamma \in \mathbb{R}$ such that 
\begin{equation*}
\lim_{r \rightarrow 0} r v_r(q) = \hat \gamma, \quad  \forall q \in \Omega^{Q} \, . 
\end{equation*}
Moreover, $\hat v(q) = \lim_{r \rightarrow 0} \Big( v_r(q) - v_r(0) \Big)$ is well defined for $\forall q \in \Omega^{Q}$. 
Finally, $\hat \gamma$ and $\hat v$ solve the ergodic HJB equation
\begin{equation} \label{ergodic hjb equation}
0 = - \hat \gamma + H \Big(q, (\hat v(q') - \hat v(q))_{q' \in \{q-1, q+1\}}\Big), \quad \forall q \in \Omega^{Q} \, , 
\end{equation}
where the Hamiltonian function $H$ is given by (\ref{hamiltonian function}). 
\end{theorem}

The next theorem, Theorem~\ref{solve gamma}, states that the constant $\hat \gamma$ from Theorem~\ref{ergodic constant} is equivalent to the optimal long-run average reward $\gamma$ in the ergodic control problem~\eqref{reduced ergodic reward functional}, which associates the equation~\eqref{ergodic hjb equation} with the ergodic control problem. Hence we call~\eqref{ergodic hjb equation} the ergodic HJB equation, which contains an unknown pair of the ergodic constant $\gamma$ and ergodic value function  $\hat v$. Theorem~\ref{solve gamma}, which will be proved in Appendix~\ref{proof theorem solve gamma}, is a first step towards obtaining an explicit solution to the equation~\eqref{ergodic hjb equation}. 

\begin{theorem} \label{solve gamma}
Let $\hat \gamma$ be the constant proposed in Theorem~\ref{ergodic constant}. Let $v_0(t, q; T)$ be the unique solution to the HJB equation~\eqref{finite-time hjb} with $r = 0$. Then
\begin{equation}
\lim_{T \rightarrow + \infty} \frac{1}{T} v_0(0, q; T) = \hat \gamma = \gamma, \quad \forall q \in \Omega^{Q} \, , 
\end{equation}
where $\gamma$ is the ergodic constant defined in~\eqref{reduced ergodic reward functional}. 
\end{theorem}

So far we've established the connection between the ergodic constant $\gamma$ and the solution to the ergodic HJB equation~\eqref{ergodic hjb equation}.
Next, we are interested in how this constant depends on the model parameter $\kappa$.
This is best seen from an explicit formulation for $\gamma = \gamma(\kappa)$ given in the following theorem. 

\begin{theorem} \label{solve ergodic constant}
Assume $\kappa^{\pm} = \kappa > 0$. Let $\lambda_{max}(\kappa)$ be the largest eigenvalue of the matrix $\boldsymbol A$ given in Theorem \ref{solution of finite-time-horizon model}. Then the ergodic constant $\gamma$ in~\eqref{reduced ergodic reward functional} is given by     
\begin{equation} \label{explicit solution of gamma}
\gamma = \gamma(\kappa) = \frac{\lambda_{max}(\kappa)}{\kappa} \, . 
\end{equation}
\end{theorem}

The proof is given in Appendix~\ref{proof theorem solve ergodic constant} and is based on establishing the asymptotic behaviour as $T \rightarrow + \, \infty$ in $v_0(t, q; T)$.

The ergodic HJB equation (\ref{ergodic hjb equation}) can be solved once we obtain $\gamma$. Proposition \ref{uniqueness of ergodic control}, which will be proved in  Appendix \ref{proof prop uniquess of ergodic control}, analyses the uniqueness (defined up to a constant) for the solution $\hat v$ to the ergodic HJB equation~\eqref{ergodic hjb equation}. Then we can obtain the existence and uniqueness for the optimal control by Proposition~\ref{Existence and Uniqueness for Ergodic Optimal Control}. 
\begin{proposition} \label{uniqueness of ergodic control}
Let $v$ and $w$ be two solutions to the ergodic HJB equation (\ref{ergodic hjb equation}) with the same $\gamma$.
Then there exists a constant $\eta \in \mathbb{R}$ such that 
\begin{equation*}
v(q) = w(q) + \eta, \quad \forall q \in \Omega^Q \, . 
\end{equation*}
That is, the solution to equation (\ref{ergodic hjb equation}) is unique up to a constant. 
\end{proposition}

\begin{proposition}[Existence and uniqueness for ergodic optimal control] \label{Existence and Uniqueness for Ergodic Optimal Control}
The optimal feedback (Markov) control for the ergodic control problem $\psi = (\psi^+, \psi^-)$ is uniquely given by
\begin{equation} \label{feedback control}
\psi^+(q) = \begin{cases} \frac{1}{\kappa} + \hat v(q) - \hat v(q-1),  & q \neq \underline{q}, \\
\, + \infty,  & q = \underline{q}, \\
\end{cases}, \,
\psi^-(q) = \begin{cases} \frac{1}{\kappa} + \hat v(q) - \hat v(q+1), & q \neq \bar q, \\
\, + \infty,  & q = \bar{q}, \\
\end{cases} 
\end{equation}
where $\hat v$ is the solution to the ergodic HJB equation~\eqref{ergodic hjb equation}.
\end{proposition}
Obviously, $\psi$ given by~\eqref{feedback control} depends on the model parameters, in particular on $\kappa$.
We will denote the optimal feedback control for the ergodic problem with the parameter $\kappa$ as $\psi^{\kappa}$. 

Finally, we come to Theorem \ref{solve ergodic control problem} proved in Appendix~\ref{proof theorem solve ergodic control problem} that provides an explicit solution to the ergodic HJB equation (\ref{ergodic hjb equation}) . 
\begin{theorem} \label{solve ergodic control problem}
Assume that $\kappa^{\pm} = \kappa$. Let $\boldsymbol{\hat{v}} = [\hat{v}(\bar{q}), \hat{v}(\bar{q}-1), ... , \hat{v}(\underline{q})]^{\top}$ be a $(\bar{q} - \underline{q} + 1)$-dim vector of a solution to the ergodic HJB equation (\ref{ergodic hjb equation}) and $\boldsymbol{\hat{v}} = \frac{1}{\kappa} \ln \boldsymbol{\hat{\omega}}$. Let $\gamma$ be the ergodic constant from Theorem \ref{solve ergodic constant} and $\boldsymbol C$ be the $(\bar{q} - \underline{q} + 1)-$square matrix
\scriptsize
\begin{align*}
\boldsymbol C = \begin{bmatrix}
- \kappa (\phi \bar{q}^2 + \gamma) & \lambda^{+} e^{-1} & 0 &  &  & ... & \\ 
\lambda^{-} e^{-1} & - \kappa \big(\phi (\bar{q}-1)^2 + \gamma \big)& \lambda^{+} e^{-1 } & & & ... & \\ 
& & &  ... & \\
& ... & &  & \lambda^{-} e^{-1} & - \kappa \big(\phi (\underline{q}+1)^2 +\gamma \big)& \lambda^{+} e^{-1}\\ 
& ... & & & 0 & \lambda^{-} e^{-1} & - \kappa (\phi \underline{q}^2 + \gamma)
\end{bmatrix} \, .
\end{align*}
\normalsize
Then it holds that
\begin{equation} \label{solve hat omega}
\boldsymbol{C} \boldsymbol{\hat{\omega}} = 0 \, ,
\end{equation}
i.e. $\boldsymbol{\hat{\omega}}$ is the non-trivial solution to the homogeneous equation with coefficient $\boldsymbol{C}$. 
Moreover, $\boldsymbol{\hat{\omega}}$ can be chosen to be positive, and it is unique up to a scalar multiple. 
\end{theorem}
Notice that once we've obtained $\boldsymbol{\hat \omega}$ by solving~\eqref{solve hat omega} we have an explicit formula for the optimal ergodic control $\psi$ uniquely given by~\eqref{feedback control}.

\subsection{Learning and regret}
In this section, we consider the parameter learning problem of the market making model in the ergodic setting, where the price sensitivity of the liquidity takers is unknown to the market maker. 
We assume that the parameter is equal on the bid/ask side $\kappa^{\ast} = \kappa^{\ast, \pm} \in \mathbb{R^+}$. 
The market maker does not observe $\kappa^\ast$, but works with the prior assumption that $\kappa^\ast$ must be in $[\underline K, \bar K]$ with $0 < \underline K < \bar K$.
At each time $t > 0$, the market maker generates the estimate of the parameter denoted $\kappa_t$ from the regularised maximum--likelihood estimator, see Algorithm~\ref{alg:cap} for more details.
Using $\kappa_t$ they can solve the ergodic control problem and obtain the policy $\psi^{\kappa_t}$ given by~\eqref{feedback control}.

\begin{remark}
In a general RL problem the agent aims to learn from data, e.g. states, actions and rewards, a policy that optimises the reward, see \cite{basei2022logarithmic, szpruch2021exploration, gao2024logarithmicregretboundscontinuoustime}.  
In this learning problem, we have derived the global optimal policy $\psi$ (see Section~\ref{ergodic analysis}).
The global optimal policy is attainable if the true $\kappa^\ast$ is known. Therefore, it is sufficient to define a learning algorithm to generate the parameter $\kappa$. 
\end{remark}

In view of this it is natural to define the learning algorithm as the function that generates $\kappa_t$ from all available information up to time $t > 0$.
\begin{definition} \label{admissible learning algorithm}
Let $(\Omega^\ast, \mathcal{F}^\ast, \mathbb{P}^\ast)$ be defined as
\begin{equation*} 
(\Omega^\ast, \mathcal{F}^\ast, \mathbb{P}^\ast) = \big( 
\Omega^{M} \times \Omega^{U},
\mathcal{F}^{M} \otimes \mathcal{F}^{U},  \mathbb{P}^{M} \otimes  \mathbb{P}^{U}
\big) \, , 
\end{equation*}
see details in~\eqref{probability space}, $\mathcal{N}$ be the $\sigma$-algebra generated by $\mathbb{P}^\ast-$null sets, and the continuous-time learning algorithm $\Psi = (\Psi_t)_{t}$ be some function $\Psi: \Omega^\ast \times \mathbb{R}^{+} \to [\underline K, \bar K]$. 
We say that $\Psi = (\Psi_t)_{t}$ is an admissible learning algorithm if $\Psi$ is $\big( \mathcal{G}^{\Psi}_{t^{-}} \otimes \mathcal{B}(\mathbb{R}^{+}) \big) /  \mathcal{B}([\underline K, \bar K])$ measurable  
with the $\sigma$-algebra $\mathcal{G}^{\Psi} = (\mathcal{G}^{\Psi}_t)_{t }$ defined as $\mathcal{G}^{\Psi}_t := \sigma\big\{
M_{s}^{\Psi; \kappa^\ast, \pm} \big | 0 < s \leq t \big \} \vee \sigma\big\{
U_{M^{\Psi; \kappa^\ast, \pm}_s}^{\Psi; \kappa^\ast, \pm} \big | 0 < s \leq t \big \} \vee \mathcal{N}$. 
\end{definition}

\begin{remark}
$\mathcal{G}_t^{\Psi}$ in Definition~\ref{admissible learning algorithm} describes the available and useful information for the agent to estimate $\kappa$ up to time $t$. 
Moreover, it is not hard to see \cite{jennrich1969asymptotic, wagner2006survey} that the learning algorithm $\kappa_t$ generated by a maximum likelihood estimator is $\mathcal{G}^{\Psi}_{t^{-}}$ measurable. 
\end{remark}

To measure the performance of a learning algorithm in the ergodic setting, we utilise the notion of regret proposed by \cite{auer2008near}.
\begin{definition} \label{regret def}
Given a learning algorithm $\Psi$ that generates $\kappa_t$ in $t \in [0,T]$, its expected regret up to time $T$ is defined as 
\begin{equation} \label{original regret}
\mathcal{R}^{\Psi}(T) = \gamma(\kappa^\ast) T - \mathbb{E}_q \Big[ \int_0^T f(Q_t^{\psi^{\kappa_t}; \kappa^\ast}, \psi^{\kappa_t}; \kappa^\ast) \, \mathrm{d} t \Big] \, ,
\end{equation}
where $\gamma(\kappa^\ast)$ is the optimal long-run average reward under the parameter $\kappa^\ast$, $f$ is the running reward function given by 
\begin{equation} \label{reward function}
f(q, \delta^{\pm}; \kappa^\ast) = \lambda^{+} \delta^{+} e^{- \kappa^\ast \delta^{+}} + \lambda^{-} \delta^{-} e^{- \kappa^\ast \delta^{-}} - \phi q^2 \, , 
\end{equation}
and $Q_t^{\psi^{\kappa_t}; \kappa^\ast}$ is the inventory process governed by  $\kappa^\ast$ but with the control $\psi^{\kappa_t}$, i.e. 
\begin{equation} \label{inventory under kappastar}
\begin{split}
dQ_t^{\psi^{\kappa}; \kappa^\ast} & =  dN_t^{\psi^{\kappa};\kappa^\ast, -} - dN_t^{\psi^{\kappa};\kappa^\ast, +} \\
& = \big( \lambda^{+} e^{-\kappa^\ast \psi^{\kappa, -}} - \lambda^{-} e^{-\kappa^\ast \psi^{\kappa, + }} \big)dt  + d\tilde{N}_{t}^{\psi^{\kappa};\kappa^\ast, -} - d\tilde{N}_{t}^{\psi^{\kappa};\kappa^\ast, +} \, , 
\end{split}
\end{equation}
with $N_t^{\psi^{\kappa};\kappa^\ast, \pm}$ the controlled counting processes for the market maker's filled buy/sell orders and $\tilde{N}_{t}^{\psi^{\kappa};\kappa^\ast, \pm}$ the corresponding compensated Poisson processes. 
\end{definition}

An alternative definition of the expected regret which is commonly seen in the finite-time-horizon RL problems,e.g. \cite{basei2022logarithmic, szpruch2021exploration}, is
\begin{equation} \label{another regret}
\begin{split}
\widehat{\mathcal{R}}^{\Psi}(T) & = J(\psi^{\kappa^\ast}; \kappa^\ast) - J(\psi^{\kappa_t}; \kappa^\ast) \\
& = \mathbb{E}_q \Big[ \int_0^T f(Q_t^{\psi^{\kappa^\ast}; \kappa^\ast}, \psi^{\kappa^\ast}; \kappa^\ast) \, \mathrm{d} t \Big] - \mathbb{E}_q \Big[ \int_0^T f(Q_t^{\psi^{\kappa_t}; \kappa^\ast}, \psi^{\kappa_t}; \kappa^\ast) \, \mathrm{d} t\Big] \, .
\end{split}
\end{equation}
The following Lemma will be proved Appendix~\ref{equivalent regret}.
\begin{lemma}
\label{connection to alternative regret}
There exists a constant $C$ independent of $T, q$ such that 
\begin{equation} \label{regret gap bounded}
\left| \gamma(\kappa^\ast) T - \mathbb{E}_q \Big[ \int_0^T f(Q_t^{\psi^{\kappa^\ast}; \kappa^\ast}, \psi^{\kappa^\ast}; \kappa^\ast) \, \mathrm{d} t \Big] \right| \leq C \, , 
\end{equation}

\end{lemma}
Therefore $\mathcal{R}^{\Psi}(T)$ and $\widehat{\mathcal{R}}^{\Psi}(T)$ shares the same asymptotic growth rate, which means that the definitions of regret~\eqref{original regret} and~\eqref{another regret} are asymptotically equivalent.

\subsubsection*{The learning algorithm}
Whenever a MO arrives, the instantaneous fill probability of the market maker's limit order depends only on the depth (offset) relative to the midprice. 
The further the market maker's posted order is from the midprice, the less likely it is to be filled. 
When a buy or sell MO arrives, let $(Y_n)_{n=1}^N \in \{0, 1\}$ denote whether the market maker's order, posted at depth $(\delta_n)_{n=1}^N$, is filled ($Y_n = 1$) or not ($Y_n=0$). The conditional distribution of $Y_n$ given $\delta_n$ is modelled as $\mathcal{L} (Y_n| \delta_n) = \boldsymbol{B}(1, e^{-\kappa^{\ast} \delta_n})$, where $\boldsymbol{B}(1, p)$ denotes the Bernoulli distribution and $p= e^{-\kappa \delta_n}$ represents the instantaneous fill probability of the market maker's limit order given a MO arrives.

To learn $\kappa^\ast$ from the Bernoulli signals in an online manner, we can simply consider a maximum likelihood estimator~\cite[Example 7.2.7]{casella2024statistical}.
The log-likelihood of $\kappa$ given $(Y_n)_{n=1}^N$ and $(\delta_n)_{n=1}^N$ is
\begin{equation*}
\ell_N(\kappa) = \sum_{n=1}^N \left( - \kappa \delta_n Y_n + (1-Y_n) \log (1 - e^{- \kappa \delta_n}) \right) \,.
\end{equation*}
Clearly
\begin{equation} \label{derivative of the log-likelihood function without regularise}
\frac{\mathrm{d}}{\mathrm{d} \kappa}  \ell_N(\kappa) = \sum_{n=1}^N \left( -\delta_n Y_n + (1-Y_n)  \delta_n \frac{e^{- \kappa \delta_n}}{1 - e^{- \kappa \delta_n}} \right) \, ,
\end{equation}
and
\begin{equation*}
\frac{\mathrm{d}^2}{\mathrm{d} \kappa^2} \ell_N(\kappa) = -\sum_{n=1}^N \delta_n^2  (1-Y_n) \left( \frac{e^{- \kappa \delta_n}}{(1 - e^{- \kappa \delta_n})^2} \right) \, .
\end{equation*}

However, one may observe that solutions to $\frac{\mathrm{d}}{\mathrm{d} \kappa} \ell_N(\kappa_N) = 0$, given by~\eqref{derivative of the log-likelihood function without regularise}, do not necessarily exist. 
Indeed e.g. if all $Y_n=1$ for $n=1$ to $N$ then there is no solution.
Moreover, even when a solution $\kappa_N$ exists, it may be arbitrarily large making the second derivative $\frac{\mathrm{d}^2}{\mathrm{d} \kappa^2} \ell_N(\kappa_N)$ arbitrarily small. This is undesirable when quantifying the tail behaviours of the estimator, see also the discussion in Remark~\ref{rm2}. To address these issues, we define the regularised log-likelihood function for estimating $\kappa$ by
\begin{equation}
\label{regularised log-likelihood function}
\begin{split}
& \tilde \ell_N(\kappa)  = \big(\ell_N(\kappa)  + R(\kappa)\big) \mathds{1}_{\kappa \leq \bar K}  + \bigg( \ell_N(\bar K) + R(\bar K)  \\
& \quad  + (\kappa - \bar K)\Big( \tfrac{\mathrm{d}}{\mathrm{d} \kappa} \ell_N(\bar K) + \tfrac{\mathrm{d}}{\mathrm{d} \kappa}   R(\bar K) \Big)  + \tfrac{1}{2} (\kappa - \bar K)^2 \Big( \tfrac{\mathrm{d}^2}{\mathrm{d} \kappa^2} \ell_N(\bar K) + \tfrac{\mathrm{d}^2}{\mathrm{d} \kappa^2}   R(\bar K)  \Big)\bigg) \mathds{1}_{\kappa > \bar K} \, .
\end{split}
\end{equation}
Recall the assumption that $\bar K > \kappa^\ast$.
The regularisation term $R(\kappa)$ is defined as
\begin{equation*}
R(\kappa) = -\kappa \delta_0 + \log (1 - e^{- \kappa \delta_0}) \,, 
\end{equation*}
where $\delta_0 > 0$ is the regularisation parameter. 
Observe that as $\bar K \rightarrow + \infty$, the regularised log-likelihood  $\tilde \ell_N$ converges to $\ell_N + R(\kappa)$, i.e. the standard log-likelihood function plus strictly concave regularisation term for any $\delta_0 > 0$.
By~\eqref{regularised log-likelihood function}, we have
\small
\begin{equation} \label{derivative of the log-likelihood function}
\begin{split}
\frac{\mathrm{d}}{\mathrm{d} \kappa} \tilde \ell_N(\kappa) & = \left(\frac{\mathrm{d}}{\mathrm{d} \kappa} \ell_N(\kappa) + \frac{\mathrm{d}}{\mathrm{d} \kappa}   R(\kappa)\right) \mathds{1}_{\kappa \leq \bar K} \\
& \hspace{5pt} + \left( \Big(\frac{\mathrm{d}}{\mathrm{d} \kappa} \ell_N(\bar K) + \frac{\mathrm{d}}{\mathrm{d} \kappa}   R(\bar K) \Big) + (\kappa - \bar K) \Big( \frac{\mathrm{d}^2}{\mathrm{d} \kappa^2} \ell_N(\bar K) + \frac{\mathrm{d}^2}{\mathrm{d} \kappa^2}   R(\bar K)  \Big)\right) \mathds{1}_{\kappa > \bar K} \,,
\end{split}
\end{equation}
\normalsize
and
\small
\begin{equation} \label{second derivative of the log-likelihood function}
\frac{\mathrm{d}^2}{\mathrm{d} \kappa^2} \tilde \ell_N(\kappa) = \left(\frac{\mathrm{d}^2}{\mathrm{d} \kappa^2} \ell_N(\kappa) + \frac{\mathrm{d}^2}{\mathrm{d} \kappa^2}   R(\kappa)\right) \mathds{1}_{\kappa \leq \bar K} + \left( \frac{\mathrm{d}^2}{\mathrm{d} \kappa^2} \ell_N(\bar K) + \frac{\mathrm{d}^2}{\mathrm{d} \kappa^2}   R(\bar K)  \right) \mathds{1}_{\kappa > \bar K} \, .
\end{equation}
\normalsize

\begin{remark} \label{rm2}
\begin{enumerate}[(1)]
\item By considering the regularised likelihood function $\tilde \ell_N(\kappa)$~\eqref{regularised log-likelihood function}, we can show that the equation $\frac{\mathrm{d}}{\mathrm{d} \kappa} \tilde \ell_N(\kappa) =0$ always admits a unique solution $\kappa_N > 0$ for all $N  \in \mathbb{N}_+$, as stated in Proposition~\ref{prop: existence of estimator}. 
Moreover, we show that any solution $\kappa_N > 0$ to this equation has the property that $-\frac{\mathrm{d}^2}{\mathrm{d} \kappa^2} \tilde \ell_N(\kappa_N)$ is bounded from below, as stated in Proposition~\ref{prop: lower bound Fisher}.
The standard maximum-likelihood estimator does not possess these properties.

\item Note that the depth $\delta_n$ posted by the market maker from the ergodic optimal control~\eqref{feedback control} can take a value of $+ \infty$ when the inventory hits the boundary. In such cases, the market maker's order is filled with probability $0$, i.e. $Y_n = 0$ a.s. For the log-likelihood function, we adopt the convention that $0 \cdot \infty = 0$. 

\item Although the regularised estimator guarantees existence, uniqueness and a well-behaved second derivative, the solution to the equation $\frac{\mathrm{d}}{\mathrm{d} \kappa} \tilde \ell_N(\kappa_N) =0$ can still be extreme for some $N$. In such cases, the agent's posted depth $\delta$, determined by the ergodic optimal control~\eqref{feedback control} as a function of the current inventory, may take values outside of a predefined set $[\underline{\delta}, \bar{\delta}] \cup \{+\infty\}$ for some constants $\underline{\delta}, \bar{\delta} \in \mathbb{R}^+$. This boundedness is crucial for establishing the concentration inequality. 
Furthermore, the second derivative of $\kappa \mapsto \gamma(\kappa; \kappa)$ is not uniformly bounded when $\kappa$ becomes arbitrarily small or large. This property, see Lemma~\ref{gamma bounded by difference of kappa}, is essential for the second--order performance gap, which leads to a logarithmic regret.
Therefore, we impose a constraint on $\kappa_N$ in Algorithm~\ref{alg:cap} to ensure it remains within a compact set.
Corollary~\ref{cor2} then implies that, with high probability, $\kappa_N$ eventually stays within the compact set for all sufficiently large $N$. 
\end{enumerate}
\end{remark}

The learning algorithm is presented in Algorithm~\ref{alg:cap}. 
\begin{algorithm}
\caption{A regularised learning algorithm for ergodic market making $\hat \Psi$}
\label{alg:cap}
\begin{algorithmic}
\Require Choose a small regularisation parameter $\delta_0 > 0$, an initial guess $\kappa_0 >0$, 
a truncation function $\varrho(\kappa) = \kappa \mathds{1}_{[\underline K, \bar K]} (\kappa) + \underline K \mathds{1}_{(0, \underline K]} (\kappa)  + \bar K \mathds{1}_{[\bar K, +\infty)} (\kappa)$ with $\underline K < \kappa^\ast$ and $\bar K >  \kappa^\ast$, the total number $N$ of coming MOs up to time $T$ with the coming times of the MOs $(t_n)_{n=1}^N$ and the signals of filled LOs from the market maker $(Y_n)_{n=1}^N$, the market maker's inventory $(Q_t)_{t\in [0,T]}$ 

\If{$t = 0$} 
\State $\hat \kappa_0 = \varrho(\kappa_0)$
\State Choose the offset $\delta_1 = \psi^{\hat \kappa_0} (Q_0)$ using~\eqref{feedback control}. 
\EndIf

\For{$t = t_i$ with $i =1, 2, \dots N$}

\State Obtain $\kappa_i$ by numerically solving $\frac{\mathrm{d}}{\mathrm{d} \kappa} \tilde \ell_i (\kappa_i) = 0$ with $\frac{\mathrm{d}}{\mathrm{d} \kappa} \tilde \ell_i (\kappa_i)$ given by~\eqref{derivative of the log-likelihood function}.

\State $\hat \kappa_i = \varrho(\kappa_i)$

\State Update $\delta_{i+1} = \psi^{\hat \kappa_i} (Q_{t_i})$ using~\eqref{feedback control}. 

\EndFor

\end{algorithmic}
\end{algorithm}

\begin{remark}
In our setting, there is no trade-off between the exploration and exploitation and so Algorithm~\ref{alg:cap} does not require any exploration phase.
This is referred to as the self-exploration property. 
Even though the agent is exploiting the ``optimal'' control based on the current estimate of $\kappa$, learning still occurs: whenever a market order arrives, the agent can infer information based on whether its own quote was filled or not, since the agent always quotes on at least one of the buy side or the sell side, i.e. for any $\delta^\pm$ take values from the ergodic optimal control~\eqref{feedback control}, we have 
$\mathbb{ P} \left( \{\delta^+ = +\infty\} \cap \{\delta^- = + \infty\}\right) = 0 $,
ensuring that the agent receives informative feedback over time, which supports the convergence of the estimated parameter to $\kappa^\ast$.
Of course, for this the assumption that the fill probability parameter $\kappa^\ast$ is the same for both buy and sell sides is crucial.

Removing this assumption (that $\kappa^\ast$ is the same for both buy and sell sides) would be challenging for two reasons. 
First, the model would lack explicit solutions.
Second, the agent would not be able to keep finite inventory limits while learning.
\end{remark}

\subsubsection*{Regret upper bound}
We now state the main result of this section, which shows the logarithmic regret upper bound of Algorithm~\ref{alg:cap}. 

\begin{theorem} \label{regret theorem}
For the regret upper bound of Algorithm~\ref{alg:cap} $\hat{\Psi}$, there exist constants $C_1, C_2 > 0$ such that $\forall T > 0$, 
\begin{equation}
\mathcal{R}^{\hat{\Psi}}(T) \leq C_1 \ln^2 T + C_2 \, . 
\end{equation}
\end{theorem}

It requires some effort to prove Theorem~\ref{regret theorem}, 
therefore we collect some key results needed for the proof. 

\subsubsection*{Step 1: Analysis of the Performance Gap}
In this section, we analyse the performance gap of the expected regret $\mathcal{R}^{\Psi}(T)$ defined in~\eqref{regret def}. 

We start with the ergodic analysis for the market making model with misspecified $\kappa$
due to the existence of the term, $\mathbb{E}_q \Big[ \int_0^T f(Q_t^{\psi^{\kappa_t}; \kappa^\ast}, \psi^{\kappa_t}; \kappa^\ast) \, \mathrm{d} t \Big]$, in regret. 

Let us define, for $q\in \Omega^Q$, $\delta^\pm \in \mathbb{\overline R}^2$, $\boldsymbol p \in \mathbb R^2$ and $\kappa^\ast \in [\underline K, \bar K]$, the Hamiltonian function
\begin{equation} \label{L function}
H(q, \delta^{\pm}, \boldsymbol{p};\kappa^\ast) =  \lambda^{+} e^{-\kappa^\ast \delta^{+}} (p_1 + \delta^{+}) \mathds{1}_{q > \underline{q}} + \lambda^{-} e^{-\kappa^\ast \delta^{-}} (p_2 + \delta^{-}) \mathds{1}_{q < \bar{q}} - \phi q^2 \, ,
\end{equation}
and the expected reward in the discounted finite-time-horizon setting under model misspecification, $v_r^{\psi^{\kappa}}(t, q; T; \kappa^\ast)$, as 
\begin{equation} \label{discounted finite-time-horizon misspcefication model}
v_r^{\psi^{\kappa}}(t, q; T;  \kappa^\ast) = \mathbb{E}_{q} \bigg[ \int_t^T e^{-r(u-t)} f(Q_u^{\psi^{\kappa}; \kappa^\ast}, \psi^{\kappa} ; \kappa^\ast) \, \mathrm{d} u + e^{-r(T-t)} G(Q_T^{\psi^{\kappa}; \kappa^\ast}) \bigg] \, ,
\end{equation}
where $f$ is given by \eqref{reward function} and $G$ is the terminal condition~\eqref{terminal condition}.
Then $v_r^{\psi^{\kappa}}(\cdot, \cdot; T; \kappa^\ast)$ satisfies the following linear ODE. See proof in Appendix~\ref{proof lemma: linear ode for discounted model}. 
\begin{lemma} \label{lemma: linear ode for discounted model}
The function $v_r^{\psi^{\kappa}}(\cdot, \cdot; T;  \kappa^\ast)$ given by~\eqref{discounted finite-time-horizon misspcefication model} satisfies the linear ODE
\begin{equation} \label{linear ode for discounted model}
0 = \partial_t v_r^{\psi^{\kappa}} - r v_r^{\psi^{\kappa}} + H \Big(q, \psi^{\kappa},  (v_r^{\psi^{\kappa}}(t, q';T; \kappa^\ast) - v_r^{\psi^{\kappa}}(t, q;T; \kappa^\ast))_{q' \in \{q-1, q+1\}} ; \kappa^\ast \Big)\,, 
\end{equation}
for all $q \in \Omega^{Q}$ subject to the terminal condition (\ref{terminal condition}).
\end{lemma}

Next we focus on the long-term average reward of $v_0^{\psi^{\kappa}}$ given by~\eqref{discounted finite-time-horizon misspcefication model} with $r=0$. Proposition~\ref{linear PDE for gamma} provides the existence of $\gamma(\kappa; \kappa^\ast)$, i.e. the average reward per unit time with a misspecified $\kappa$. Moreover, $\gamma(\kappa; \kappa^\ast)$ with the ergodic value function under the model misspecification, $\hat{v}^{\psi^\kappa}(\cdot; \kappa^\ast): \Omega^Q \to \mathbb R$, solves the linear system~\eqref{linear system for gamma(kappa;kappastar)} below. 
Rigorous definition of $\hat{v}^{\psi^\kappa}(\cdot; \kappa^\ast)$ and proof of Proposition~\ref{linear PDE for gamma} are provided in Appendix~\ref{proof Proposition linear PDE for gamma}. 

\begin{proposition} \label{linear PDE for gamma}
There exists $\gamma(\kappa; \kappa^\ast) \in \mathbb{R}$ such that
\begin{equation} \label{definition gamma(kappa;kappastar)}
\gamma(\kappa; \kappa^\ast) = \lim_{T \rightarrow + \infty}\frac{1}{T} v_{0}^{\psi^{\kappa}}(0, q; T; \kappa^\ast) \, , 
\end{equation}
where $v_0^{\psi^{\kappa}}(\cdot, \cdot; T;  \kappa^\ast)$ is given by~\eqref{discounted finite-time-horizon misspcefication model} with $r =0$. 
Moreover, there exist $\gamma(\kappa; \kappa^\ast)$ and $\hat{v}^{\psi^\kappa}: \Omega^Q \to \mathbb{R}$ that solve the linear system 
\begin{equation} \label{linear system for gamma(kappa;kappastar)}
0 = - \gamma(\kappa; \kappa^\ast) + H \Big(q, \psi^{\kappa}, (\hat{v}^{\psi^\kappa}(q';  \kappa^\ast) - \hat{v}^{\psi^\kappa}( q;  \kappa^\ast))_{q' \in \{q-1, q+1\}} ; \kappa^\ast \Big) , \, \forall q \in \Omega^{Q}.
\end{equation}
\end{proposition}

The fact that $v_{0}^{\psi^{\kappa}}(\cdot, \cdot; T; \kappa^\ast)$ satisfies the linear ODE~\eqref{linear ode for discounted model} with $r=0, \, \forall q \in \Omega^{Q}$ allows us to solve it in a matrix form. 
Moreover, by analysing the case of $T \rightarrow + \infty$ in~\eqref{definition gamma(kappa;kappastar)},
we can obtain a closed-form expression for $\gamma(\kappa, \kappa^\ast)$ as shown in Proposition~\ref{analytical solution for gamma(kappastar)}. See proof in Appendix~\ref{proof Proposition analytical solution for gamma(kappastar)}.

\begin{proposition}
\label{analytical solution for gamma(kappastar)}
Let $\boldsymbol{\tilde{A}}_0$ be a $(\bar q - \underline q+1)-$square tridiagonal matrix whose rows are labelled from $\bar q$ to $\underline q$ and entries are given by
\begin{align*}
\boldsymbol{\tilde{A}}_{0}(i,q) = \begin{cases}
- (\lambda^{+} e^{- \kappa^\ast \psi^{\kappa, +}(q)} \mathds{1}_{q > \underline q} + \lambda^{-} e^{- \kappa^\ast \psi^{\kappa, -}(q)}\mathds{1}_{q < \bar q} ), & \text{if $i = q$}, \\
\lambda^{+} e^{- \kappa^\ast \psi^{\kappa, +}(q)}, & \text{if $i = q+1$}, \\
\lambda^{-} e^{- \kappa^\ast \psi^{\kappa, -}(q)}, & \text{if $i = q-1$}, \\
0, & \text{otherwise, }
\end{cases}
\end{align*}
Let $\boldsymbol{U}$ be the matrix whose columns are the eigenvectors of $\boldsymbol{\tilde{A}}_{0}$. Let $\boldsymbol{\tilde b}$ be a $(\bar q - \underline q +1)-$dim vector with each component given by 
\begin{equation*}
b_i = \lambda^{+} \psi^{\kappa, +}(i) e^{- \kappa^\ast \psi^{\kappa, +}(i)} \mathds{1}_{i > \underline q} + \lambda^{-} \psi^{\kappa, -}(i) e^{- \kappa^\ast \psi^{\kappa, -}(i)}\mathds{1}_{i < \bar q} - \phi i^2,
\end{equation*}
for $i = [\bar q, \bar q -1 , ... ,\underline q]$. 
Let $\boldsymbol{W}$ is the $(\bar q - \underline q +1)-$square matrix with the first diagonal element equal to $1$ and all other elements equal to $0$. 
Then $\gamma(\kappa; \kappa^\ast)$ given by~\eqref{definition gamma(kappa;kappastar)} satisfies
\begin{equation} \label{solution of gamma(kappa; kappaast)}
\gamma(\kappa; \kappa^\ast) \mathds{1}  = \boldsymbol{U} \boldsymbol{W} \boldsymbol{U}^{-1} \boldsymbol{\tilde b}, 
\end{equation}
where $\mathds{1}$ is a $(\bar q - \underline q +1)-$dim vector with entries $1$.
\end{proposition}

\begin{remark}
Note that $\gamma(\kappa; \kappa^\ast)$, given by~\eqref{solution of gamma(kappa; kappaast)}, represents the long-term average reward under the optimal ergodic control with parameter $\kappa$, while the true market environment is $\kappa^\ast$.
This is different with $\gamma(\kappa)$, which is given by~\eqref{explicit solution of gamma}.
Clearly, $\gamma(\kappa; \kappa) = \gamma(\kappa)$ for any $\kappa \in [\underline K, \bar K]$, meaning that if agent uses the same $\kappa$ as the ``true'' market parameter, they achieve the optimal long-term average reward.
The key challenge is to prove that $[\underline K, \bar K] \ni \kappa \mapsto \gamma(\kappa; \kappa^\ast) \in \mathbb R$ is twice continuously differentiable. 
\end{remark}

Although Proposition~\ref{analytical solution for gamma(kappastar)} gives an expression for $\gamma(\kappa; \kappa^\ast)$, it is not trivial to prove the regularity of $\gamma(\kappa; \kappa^\ast)$ as $\boldsymbol{\tilde{A}}_0$ is not a self-adjoint or normal operator. We begin with the following lemma that establishes the regularity of $\gamma=\gamma(\kappa)$ given in Theorem~\ref{solve ergodic constant}. This result serves as a preliminary step toward proving Lemma~\ref{second derivative of gamma is bounded}. 
The proof is provided in Appendix \ref{continuity ergodic constant}.
\begin{lemma} \label{gamma kappa Lipschitz differentiable}
The ergodic constant $\gamma:  [\underline K, \bar K] \ni \kappa \mapsto \gamma(\kappa) \in \mathbb{R}$ given in Theorem~\ref{solve ergodic constant} is in $C^2([\underline K, \bar K])$. 
\end{lemma}

We next analyse the regularity of $\gamma(\kappa; \kappa^\ast)$. 
\begin{lemma} \label{second derivative of gamma is bounded}
$\kappa \mapsto \gamma(\kappa; \kappa^\ast)$ is in $C^2([\underline K, \bar K])$. 
\end{lemma}
The key approach in the proof of Lemma~\ref{second derivative of gamma is bounded} (see Appendix \ref{proof lemma second derivative of gamma is bounded}) is to construct a self-adjoint operator similar to $\boldsymbol{\tilde{A}}_0$ and express $\gamma(\kappa;\kappa^\ast)$ in terms of the eigenvector of this self-adjoint operator, which is differentiable.

By Lemma~\ref{second derivative of gamma is bounded}, it is trivial to obtain the following lemma by the fact that $\kappa \mapsto \gamma(\kappa; \kappa^\ast)$ attains the maximum at $\kappa = \kappa^\ast$, because $\psi^{\kappa^\ast}$ is the optimal control for~\eqref{definition gamma(kappa;kappastar)}.

\begin{lemma} \label{gamma bounded by difference of kappa}
There exist a constant $C > 0$ depends on $\lambda^\pm, \underline K$ and $\bar K$ such that 
\begin{equation*}
0 \leq \gamma(\kappa^\ast; \kappa^\ast) - \gamma(\kappa; \kappa^\ast) \leq C \left| \kappa - \kappa^\ast \right|^2, \, \forall \kappa \in [\underline K, \bar K] \, . 
\end{equation*}
\end{lemma}

\begin{remark} \label{remark regret bound dependency on model parameters}
The constant $C$ in Lemma~\ref{gamma bounded by difference of kappa} implicitly depends on the model parameters, such as  $\lambda^\pm, \underline K$ and $\bar K$.
Although we prove that $\kappa \mapsto \gamma(\kappa; \kappa^\ast)$ is twice continuously differentiable, we do not have an analytic expression as it depends on derivatives of the eigenvalues and eigenvectors of the matrix whose entries are functions of the model parameters. 
While this is no obstacle to asymptotic regret analysis it  may be interesting to quantify the dependence of $C$ on the model parameters. 
This has been done numerically, see Figure~\ref{dependency of C1}.
\end{remark}

So far we have performed the ergodic analysis for the market making model with the parameter $\kappa$ misspecified. 
Another key step towards quantifying the performance gap, see Theorem~\ref{performance gap}, is to analyse the ergodicity under the model misspecification, i.e. how fast the state process $\big(Q_t^{\psi^{\kappa}; \kappa^\ast}\big)_{t\geq0}$ following the dynamics~\eqref{inventory under kappastar} converges to the equilibrium distribution.

\begin{definition}[Equilibrium]
\label{equilibrium}
The distribution $\pi \in \mathcal P(\Omega^Q)$ is said to be an equilibrium distribution for the Markov control $\delta^{\pm}$ if, for any $t \geq 0$, it holds that  
$
\pi = \mathcal{L} (Q_t^{\pi, \delta^{\pm}}) \, , 
$
where $\mathcal{L}$ denotes the law and $Q_t^{\pi, \delta^{\pm}}$ is given by (\ref{inventory}) under control $\delta^{\pm}$ with $Q_0 \sim \pi$. 
\end{definition}

The following lemma is proved in Appendix~\ref{proof lemma ergodic control stability}. 
\begin{lemma} \label{ergodic control stability}
For any $\kappa \in [\underline K, \bar K]$, the controlled process 
$\big( Q_t^{\psi^{\kappa} ; \kappa^\ast} \big)_{t \geq 0}$, following the dynamics~\eqref{inventory under kappastar} under the control $\psi^{\kappa}$, admits a unique  equilibrium distribution, denoted by $\pi^{\psi^\kappa; \kappa^\ast}$.
\end{lemma}

As we show in Appendix~\ref{proof lemma ergodic control stability}, $(Q_t)_{t\geq0}$--with superscripts omitted for brevity--can be equivalently represented as a continuous-time Markov chain (CTMC) with the transition rate matrix $\boldsymbol{Q}$ given by~\eqref{transition matrix P}. 
Since the transition rate matrix $\boldsymbol{Q}$ is tridiagonal, the CTMC is irreducible and recurrent. 
Therefore, the convergence of the distribution of $Q_t$ to the equilibrium distribution follows the Convergence Theorem~\cite[Theorem 3.6]{hairer2010convergence}.

\begin{lemma} [Convergence Theorem] \label{convergence theorem}
Let $\pi_t^{\psi^{\kappa}; \kappa^\ast}$ be the probability distribution of the random variables $Q_t$ that follows the controlled dynamics~\eqref{inventory under kappastar} with an initial state $Q_0 \sim \pi_0$ and the control $\psi^{\kappa}$. Let $\pi^{\psi^{\kappa}; \kappa^\ast}$ be the equilibrium distribution established by Lemma~\ref{ergodic control stability}. 
Then, there exists constants $C > 0$ and $0 < \alpha <1$, depending on $\kappa$, such that
\begin{equation*}
\big\| \pi_t^{\psi^{\kappa}; \kappa^\ast} - \pi^{\psi^{\kappa}; \kappa^\ast} \big\|_{TV} \leq C \alpha^{t}, \quad \forall t \geq 0 \, .
\end{equation*}
\end{lemma}

We next state the following proposition, proved in Appendix~\ref{proof of Proposition expectation under equi is 0}, which establishes a key property of the equilibrium distribution $\pi^{\psi^\kappa; \kappa^\ast}$.
 
\begin{proposition} \label{expectation under equi is 0}
Let $\pi^{\psi^\kappa; \kappa^\ast}$ be the equilibrium distribution for the inventory process $Q_t$ following the controlled SDE~\eqref{inventory under kappastar}
\begin{equation*}
dQ_t^{\psi^\kappa; \kappa^\ast} = \big( \lambda^{+} e^{-\kappa^\ast \psi^{\kappa, -}} - \lambda^{-} e^{-\kappa^\ast \psi^{\kappa, +}} \big)dt  + d\tilde{N}_{t}^{\psi^{\kappa, -}} - d\tilde{N}_{t}^{\psi^{\kappa, +}}, \quad t \in [0, T], Q_0 \sim \pi^{\psi^\kappa; \kappa^\ast},
\end{equation*}
with $\psi^{\kappa}$ given by~\eqref{feedback control}. Then it holds that 
\begin{equation*}
\begin{split}
& \mathbb{E} \bigg[ \int_0^T  \lambda^{+} e^{-\kappa^\ast \psi^{\kappa, +}(Q_t^{\psi^{\kappa}; \kappa^\ast})} \Big( \hat{v}^{\psi^{\kappa}}(Q_t^{\psi^{\kappa}; \kappa^\ast}-1; \kappa^\ast) - \hat{v}^{\psi^{\kappa}}(Q_t^{\psi^{\kappa}; \kappa^\ast}; \kappa^\ast) \Big)  \mathds{1}_{Q_t^{\psi^{\kappa}; \kappa^\ast} > \underline q} \\
& + \lambda^{-} e^{-\kappa^\ast \psi^{\kappa, -}(Q_t^{\psi^{\kappa}; \kappa^\ast})} \Big( \hat{v}^{\psi^{\kappa}}(Q_t^{\psi^{\kappa}; \kappa^\ast}+1; \kappa^\ast) - \hat{v}^{\psi^{\kappa}}(Q_t^{\psi^{\kappa}; \kappa^\ast}; \kappa^\ast) \Big) \mathds{1}_{ Q_t^{\psi^{\kappa}; \kappa^\ast} < \bar q} \, \mathrm{d} t \bigg] = 0 \,,
\end{split}
\end{equation*}
where $\hat{v}^{\psi^{\kappa}}(q;\kappa^\ast)$ is defined in Proposition~\ref{linear PDE for gamma}. 
\end{proposition}

With Proposition~\ref{linear PDE for gamma},~\ref{expectation under equi is 0} and Lemma~\ref{gamma bounded by difference of kappa} at hand, we finally obtain Theorem~\ref{performance gap}. 
\begin{theorem} \label{performance gap}
Given a continuous-time learning algorithm $\Psi$ that generates $\kappa_t$ up to time $T > 0$, let $\mathcal{R}^{\Psi}(T)$ be the regret given by \eqref{original regret}, then it holds that 
\begin{equation*}
\mathcal{R}^{\Psi}(T) \leq C_1 \mathbb E \Big[ \int_0^T \left| \kappa_t - \kappa^\ast \right|^2 \, \mathrm{d} t \Big] + \frac{C_2}{\ln (\alpha^{-1})}\,,
\end{equation*}    
with constants $C_1, C_2 > 0, 0 < \alpha <1$  independent of $T$.
\end{theorem}
The proof is provided in Appendix~\ref{proof theorem performance gap}.

\subsubsection*{Step 2: Concentration Inequality}
The next step towards Theorem~\ref{regret theorem} is to quantify the precise tail behaviour, also known as concentration inequality,  of the regularised maximum-likelihood estimator in Algorithm~\ref{alg:cap}. Recall that $(\Omega^\ast, \mathcal{F}^\ast, \mathbb{P}^\ast)$ is given in Definition~\ref{admissible learning algorithm}. 

We start with several significant propositions to the estimator. The proofs of the following propositions are provided in Appendix~\ref{proof of concentration inequality}. 
\begin{proposition} \label{prop: concentration rescaling sub-Gaussian}
Let $(\delta_n)_{n=1}^N$ be a collection of non-negative random variables taking values in $[\underline \delta, \bar \delta] \cup \{+ \infty\}$. Then for any $\varepsilon \geq 0$ and bounded function $f: [\underline \delta, \bar \delta] \cup \{+ \infty\} \rightarrow \mathbb R$, it holds that,
\begin{equation*}
\mathbb P^{\ast} \left( \left| \sum_{n=1}^{N} f(\delta_n) Y_n - \sum_{n=1}^{N}f(\delta_n)e^{-\kappa^{\ast} \delta_n} \right| \leq \left\| f\right\|_{\infty} \sqrt{2N\ln(\tfrac{2}{\varepsilon})} \right) \geq 1 - \varepsilon. 
\end{equation*}
\end{proposition}

\begin{proposition} \label{prop: lower bound Fisher}
There exist constants $c,C >0$ depending on $\underline K, \bar{K}, \underline{\delta}$ and $\bar{\delta}$ such that for any policy $(\delta_n)_{n=1}^\infty$ taking values in $[\underline{\delta}, \bar{\delta}] \cup \{+ \infty\}$, it holds that for any $\varepsilon > 0$,
\begin{equation*}
\mathbb P^\ast \left(    \inf_{\kappa > 0}  \Big( -\frac{\mathrm{d}^2}{\mathrm{d} \kappa^2} \tilde{\ell}_N(\kappa) \Big) \geq c N -C\sqrt{N \ln \big(\tfrac{2}{\varepsilon} \big) }\right) \geq 1-\varepsilon.
\end{equation*}
\end{proposition}

\begin{proposition} \label{prop: existence of estimator}
There exists a unique $\kappa_N > 0$ such that $\frac{\mathrm{d}}{\mathrm{d} \kappa} \Tilde{\ell}_N (\kappa_N) = 0$, where $\frac{\mathrm{d}}{\mathrm{d} \kappa} \Tilde{\ell}_N (\kappa)$ is given by \eqref{derivative of the log-likelihood function}.
\end{proposition}

\begin{proposition} \label{prop: upper bound score}
There exists constants $C, c \geq 0$ depending on $\kappa^{\ast}$, $\delta_0$, and $\bar \delta$ such that for any policy $(\delta_n)_{n=1}^{N}$ taking values in $[\underline \delta, \bar \delta] \cup \{+ \infty\}$, it holds that for any $\varepsilon > 0$, 
\begin{equation*}
\mathbb P^\ast \left(  \Big| \frac{\mathrm{d}}{\mathrm{d} \kappa} \tilde{\ell}_N(\kappa^\ast) \Big| \leq C  \sqrt{N \ln \big( \tfrac{2}{\varepsilon}  \big)} + c \right) \geq 1-\varepsilon.
\end{equation*}
\end{proposition} 

With the above propositions, we obtain Theorem~\ref{thm: concentration} with the proof provided in Appendix~\ref{proof of concentration inequality}, which quantifies the concentration inequality of the regularised maximum likelihood estimator in Algorithm~\ref{alg:cap}. 

\begin{theorem} \label{thm: concentration}
Let $\kappa_N > 0$ be the unique solution to $\frac{\mathrm{d}}{\mathrm{d} \kappa} \tilde{\ell}_N(\kappa_N)= 0$. There exists constants $C, c, N_0\geq 0$ such that for any $\varepsilon \geq 0$, if $N \geq N_0 \ln \big( \tfrac{2}{\varepsilon} \big)$, then
$$\mathbb P^\ast \left( |  \kappa_N - \kappa^\ast | \leq C N^{-1/2} \sqrt{ \ln \big( \tfrac{2}{\varepsilon} \big)} + cN^{-1}\right) \geq 1- 2 \varepsilon.$$
\end{theorem}

We then introduce a corollary to Theorem~\ref{thm: concentration}. See proof in Appendix~\ref{proof corollary}. 
\begin{corollary}
\label{cor}
Let $\kappa_N > 0$ be the unique solution to $\frac{\mathrm{d}}{\mathrm{d} \kappa} \tilde{\ell}_N(\kappa_N)= 0$. Then there exists constants $C, c, N_0\geq 0$ such that for any $\varepsilon \geq 0$, 
$$\mathbb P^\ast \left( |  \kappa_N - \kappa^\ast | \leq C N^{-1/2} \sqrt{ \ln \big( \tfrac{2N}{\varepsilon} \big)} + cN^{-1} \quad \text{for all} \quad N \geq N_0 \ln \big( \tfrac{2}{\varepsilon} \big) \right) \geq 1- \varepsilon\, .$$
\end{corollary}

Another corollary to the above results, which implies that for sufficiently large $N$ the estimator will eventually remain within the compact set $[\underline{K}, \bar{K}]$, is stated below. See Appendix~\ref{proof of cor 2} for the proof.
\begin{corollary} \label{cor2}
Let $\kappa_N > 0$ be the unique solution to $\frac{\mathrm{d}}{\mathrm{d} \kappa} \tilde{\ell}_N(\kappa_N)= 0$. Then there exists constants $N_0, N_0' \geq 0 $  such that  for any $\varepsilon \geq 0$
\begin{equation*}
\mathbb P^\ast \left( \kappa_N \in [\underline K, \bar K] \quad \text{for all} \quad N \geq \max \left( N_0 \ln \big( \tfrac{2}{\varepsilon} \big), N_0'/ \ln \big( \tfrac{2}{\varepsilon} \big)\right)\right) \geq 1- \varepsilon\, .
\end{equation*}
\end{corollary}

\subsubsection{Step 3: Proof of Theorem~\ref{regret theorem}}

With Theorem~\ref{performance gap}, Theorem~\ref{thm: concentration} and Corollary~\ref{cor} at hand, we proceed to prove Theorem~\ref{regret theorem}

Let $\tau_n$ be the time when the $n-$th market order arrives. 
By the fact that the summation of two independent Poisson processes is a Poisson process, 
we have
$(\tau_{n+1} - \tau_n) \sim_{IID} \text{exponential} \left( \lambda^+ + \lambda^- \right)$
with the convention that $\tau_0 = 0$. 
Besides, let us define
$\kappa_t = \kappa_{N_t}$, 
where $N_t$ is the number of signals up to time $t$. 
By using the notation above, we have 
\begin{equation*}
\int_0^T |\kappa_t - \kappa^{\ast}|^2 \, \mathrm{d} t \leq \sum_{n=0}^{N_T} (\tau_{n+1} - \tau_n) |\kappa_n - \kappa^{\ast}|^2 =: X_{N_T} \, .
\end{equation*}
Clearly $X_{N_T}$ is a non-negative random variable. 

The following proposition, Proposition~\ref{X_NT bound in high prob}, which is proved in Appendix~~\ref{proof prop X_NT bound in high prob}, states that, given any $N_T$, i.e. the number of signals up to time $T$,  the random variable $X_{N_T}$ is bounded by $\mathcal{O}(\ln^2 N_T )$ with high probability. 
\begin{proposition} \label{X_NT bound in high prob}
There exist constants $C_1, C_2, C_3, C_4> 0$ such that for any  $\varepsilon > 0$, 
\begin{equation*}
\mathbb{P}^\ast \left( X_{N_T} \leq C_1 \ln^2 N_T + C_2 \ln N_T \ln (\tfrac{2}{\varepsilon}) + C_3 \ln^2 (\tfrac{2}{\varepsilon}) + C_4 \right) \geq 1 - 2 \varepsilon.  
\end{equation*}
\end{proposition}

Let $r(N_T) :=  C_1 \ln^2 N_T + C_2 \ln N_T \ln (\tfrac{2}{\varepsilon}) + C_3 \ln^2 (\tfrac{2}{\varepsilon}) + C_4 $. Then, by using Proposition~\ref{X_NT bound in high prob}, we have 
\begin{equation*}
\begin{split}
\mathbb E [X_{N_T}] & = \mathbb E \Big[ \mathbb E \big[X_{N_T} \big| N_T \big] \Big] \\
& = \mathbb E \Big[ \mathbb E \big[X_{N_T}  \mathds{1}_{X_{N_T} \leq r(N_T) } \big| N_T \big] \Big] + \mathbb E \Big[ \mathbb E \big[X_{N_T}  \mathds{1}_{X_{N_T} > r(N_T) } \big| N_T \big] \Big] \\
& \leq \mathbb E \Big[ \mathbb E \big[r(N_T) \big| N_T \big] \mathbb{P} \big(X_{N_T} \leq r(N_T) \big| N_T \big) \Big]   \\
& \qquad + \mathbb E \Big[ \mathbb E \big[ X_{N_T} \big| N_T , X_{N_T} > r(N_T) \big]  \mathbb{P} \big(X_{N_T} > r(N_T) \big| N_T \big) \Big] \\
& \leq \mathbb E \big[   C_1 \ln^2 N_T + C_2 \ln N_T \ln (\tfrac{2}{\varepsilon}) + C_3 \ln^2 (\tfrac{2}{\varepsilon}) + C_4  \big]  \\ 
& \qquad +  (\bar K - \underline K)^2  \mathbb{E} \Big[ \sum_{n=0}^{N_T} (\tau_{n+1} - \tau_n) (2 \varepsilon) \Big] \, . 
\end{split}
\end{equation*}
Let us set $\varepsilon = \frac{2}{T}$ and we can take $\varepsilon$ out of the expectation. Besides, we know that $x \mapsto \ln x$ is concave and, for large $x$, i.e. $x \geq 3$, $x \mapsto \ln^2 x$ is concave, hence by Jensen's inequality, we have
\small
\begin{equation*}
\begin{split}
\mathbb{E} [X_{N_T}] & \leq C_1  \ln^2 (\mathbb E[N_T]) + C_2 \ln (\mathbb E[N_T]) \ln T + C_3 \ln^2 T + C_4 + 4 (\bar K - \underline K)^2  \frac{\mathbb{E} [ \tau_{N_T + 1}]}{T} \\
& \leq C_1  \ln^2 \big(T(\lambda^+ + \lambda^-)\big)  + C_2 \ln \big( T(\lambda^+ + \lambda^-) \big) \ln T + C_3 \ln^2 T + C_4 \\
& \qquad + 4 (\bar K - \underline K)^2 \left(1 + \frac{1}{T(\lambda^+ + \lambda^-)}\right) \\
& \leq C_1 \big(\ln^2 T   + \ln^2 (\lambda^+ + \lambda^-) \big)+ C_2 \ln^2 T + C_2 \ln^2 T \ln (\lambda^+ + \lambda^-) \\ 
& \qquad + C_3 \ln^2 T +  C_4 + 4 (\bar K - \underline K)^2 \left(1 + \frac{1}{T(\lambda^+ + \lambda^-)} \right)  \\
& \leq \left( C_1 + C_2 (1 + \ln (\lambda^+ + \lambda^-) + C_3 )\right) \ln^2 T  + \left( C_1 \ln^2 (\lambda^+ + \lambda^-) + C_4 + 4 (\bar K - \underline K)^2 \right)
\end{split}
\end{equation*}
\normalsize
where we use the fact that $\ln T \leq \ln^2 T$ for large $T$ and we ignore the term of order $\mathcal{O}(T^{-1})$. 
By using Theorem~\ref{performance gap}, we then have
\begin{equation*}
\begin{split}
\mathcal{R}^{\hat{\Psi}}(T) & \leq C_1' \mathbb E \Big[ \int_0^T \left| \kappa_t - \kappa^\ast \right|^2 \, \mathrm{d} t \Big] + C_2' \frac{e^{T \ln \alpha} -1 }{\ln \alpha} \\
& \leq C_1' \mathbb E \Big[ X_{N_T} \Big] + \frac{C_2'}{\ln (\alpha^{-1})} (1 - \alpha^T) \\
& \leq C_1' \left( C_1 + C_2 (1 + \ln (\lambda^+ + \lambda^-) + C_3 )\right) \ln^2 T  \\
& \qquad + C_1' \left( C_1 \ln^2 (\lambda^+ + \lambda^-) + C_4 + 4 (\bar K - \underline K)^2 \right) + \frac{C_2'}{\ln (\alpha^{-1})} \, ,
\end{split}
\end{equation*}
where $C_1', C_2' > 0$ and $0 < \alpha < 1$ are constants independent of $T$ in Theorem ~\ref{performance gap} and $C_1, C_2, C_3, C_4$ are constants independent of $\varepsilon$ and $T$ in Proposition~\ref{X_NT bound in high prob}, hence the result of Theorem~\ref{regret theorem}.

\section{Numerical experiments} \label{experimental analysis}

\subsection{Ergodic control and regret of Algorithm~\ref{alg:cap}}
In this section, we numerically simulate the ergodic market making model~\eqref{ergodic reward functional} and the achieved regret of Algorithm~\ref{alg:cap}. 
Code used to produce results in this section is available at~\texttt{\url{https://github.com/Galen-Cao/MM_parmater_learning}}.

Let us consider the following parameters in the simulation: $\lambda^{\pm}=1/s$, $\kappa^{\pm}=10\$^{-1}$, $\sigma=1.0 s^{-1/2}\$$, $S_0 = \$10 $, $\bar{q} = 30$, $\underline{q} = -30$ and $\phi = \$ 1 \times 10^{-5}$. 

By Theorem \ref{solution of finite-time-horizon model}, we can determine the square matrix $\boldsymbol{A}$ and the largest eigenvalue of $\boldsymbol{A}$ is $\lambda_{max} = 0.7297$. Then we have $\gamma = 0.07297$ by using Theorem \ref{solve ergodic constant}. 
Figure \ref{value of v/T} (left panel) plots the asymptotic behaviours of $v(0,q; T) / T$ as $T$ increases, 
which $v(t,q;T)$ is the value function in the discounted finite-time-horizon setting with the discount factor $r =0$ given by~\eqref{reduced finite-time value function}. We can see that, for any initial $q \in \Omega^{Q}$, the value $v(0,q; T) / T \rightarrow \gamma = \lambda_{max} / \kappa$ as $T \rightarrow + \infty$ as stated in Theorem \ref{solve gamma}. 
\begin{figure}
\centering
\includegraphics[width=1\textwidth]{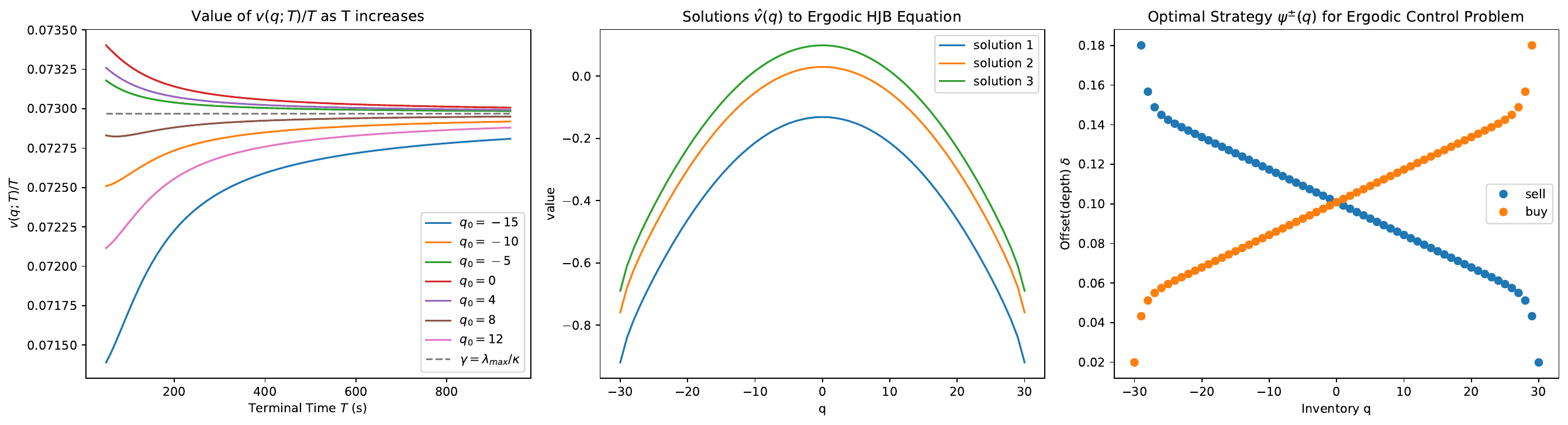}
\caption{Left: The asymptotic behaviours of $v(0,q; T) / T$ as $T$ increases. Middle: Several possible solutions $\hat{v}(q)$ to the ergodic HJB equation (\ref{ergodic hjb equation}). Right: The unique optimal control $\delta^{\ast}$.}
\label{value of v/T}
\end{figure}

We then solve the ergodic HJB equation (\ref{ergodic hjb equation}) and find the optimal feedback control~$\psi$. By Theorem \ref{solve ergodic control problem}, there exists the null space of the $n-$square matrix $\boldsymbol{C}$ with non-trivial solutions satisfying $\boldsymbol{C} \boldsymbol{\hat{\omega}} = 0$ with $\text{rank}(C)=n-1$. We consider the positive solutions in the null space, hence the solutions $\boldsymbol{\hat{v}} = \ln \boldsymbol{\hat{\omega}} / \kappa$ to the ergodic HJB equation can be well-defined. 
Figure \ref{value of v/T} (middle panel) represents several solutions $q \mapsto \hat{v}(q)$ to the ergodic HJB equation (\ref{ergodic hjb equation}). It can be seen that the solution $\hat{v}(q)$ is unique up to a constant.  For all possible solutions $\hat{v}(q)$, the optimal control $\psi^\pm(q)$ for the ergodic control problem is unique, as shown in Figure \ref{value of v/T} (right panel). 

\begin{figure}
\centering
\includegraphics[width=1\textwidth]{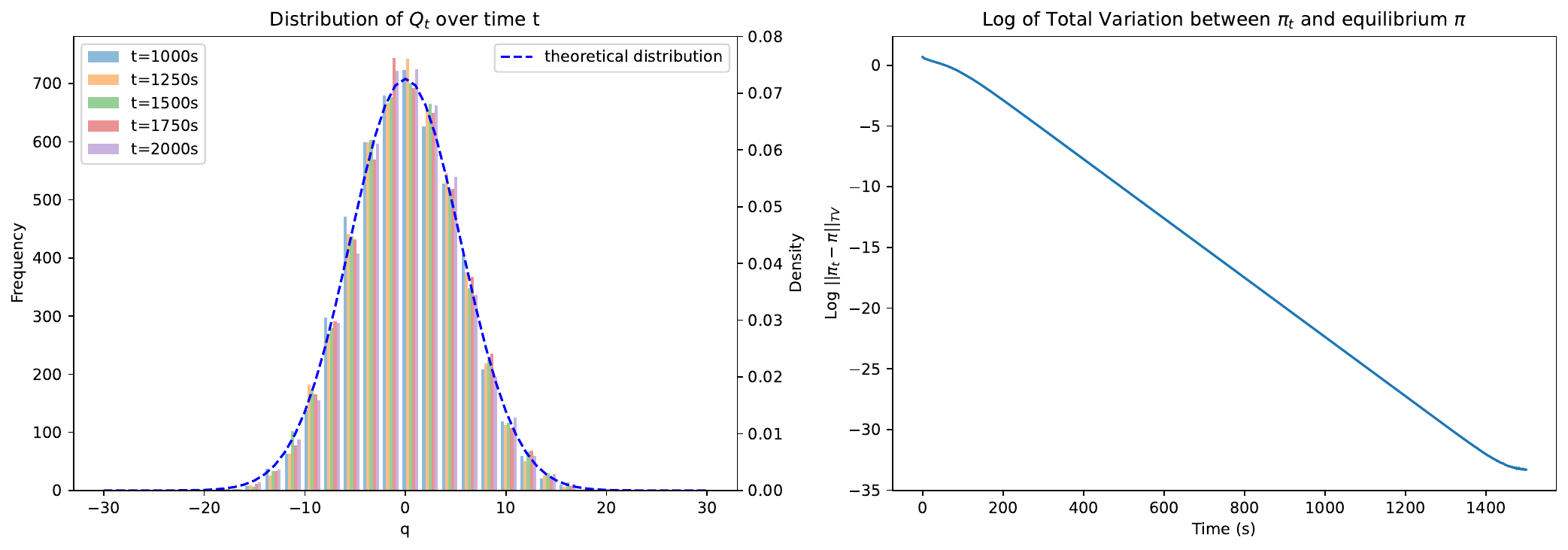}
\caption{Left: Histogram of $Q_t$ at $t = 1000s$, $1250s$, $1500s$, $1750s$ and $2000s$ under the ergodic optimal control $\psi^\pm$. Moreover, the dotted blue line plots the theoretical equilibrium distribution $\pi$ under the ergodic control $\psi^\pm$ as derived in Appendix \ref{proof lemma ergodic control stability}. Right: The log of total variation between the inventory distribution $\pi_t$ and the equilibrium $\pi$. }
\label{distribution}
\end{figure}

We continue to analyse the ergodicity of the market making system. Figure \ref{distribution} (left panel) plots the inventory distribution $\pi_t$ at time $t = 1000s$, $1250s$, $1500s$, $1750s$ and $2000s$ under the ergodic optimal control $\psi^\pm$. We can see that the distribution $\pi_t$ tends to converge to the dotted blue line over time, which is the theoretical equilibrium distribution $\pi$ of the inventory as derived in Appendix \ref{proof lemma ergodic control stability}. The right panel in Figure \ref{distribution} plots the log of the total variation between the distribution $\pi_t$ and the equilibrium $\pi$. We terminates the simulation at $t = 1500s$ as it reaches the machine precision. It shows that 
the convergence rate to the equilibrium is exponential. 

\begin{figure}
\centering
\includegraphics[width=1\textwidth]{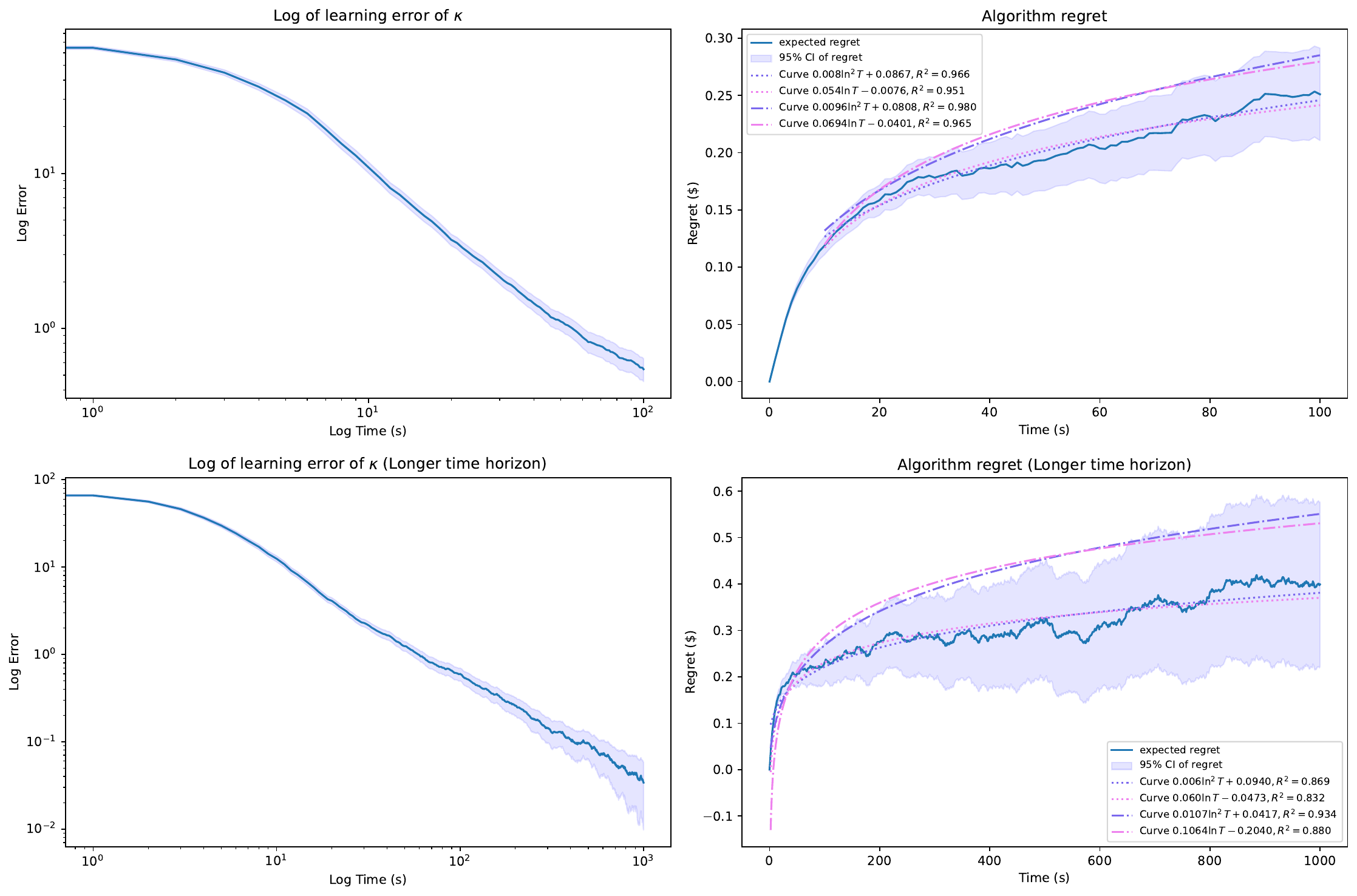}
\caption{Top left: Log-log plot of the learning error $\left|\kappa_t - \kappa^\ast\right|$ over time under Algorithm~\ref{alg:cap}. Top right: The regret of  Algorithm~\ref{alg:cap} over time. Bottom: The learning error and regret over a larger time horizon.}
\label{regret}
\end{figure}

Now we proceed to simulate learning and the regret of Algorithm~\ref{alg:cap}.
The results are in Figure~\ref{regret}. 
The experimental parameters are set as $\lambda^{\pm}=0.4/s$, $\kappa^{\pm\ast}=10\$^{-1}$, $\sigma=0.01 s^{-1/2}\$$, $\bar{q} = 30$, $\underline{q} = -30$, $\phi = \$ 1 \times 10^{-6}$, $\underline K = 1\$^{-1}$ and $\bar K = 100\$^{-1}$. 
We used $1000$ simulation scenarios, with time horizon $T=1000$ seconds. 
We include two plots, one with $T=100$ seconds and one with $T=1000$ seconds.
The left panels show the learning error \( \left| \kappa_t - \kappa^\ast \right| \) in the log-log scale.
Initially, the error decays slowly due to the limited number of Bernoulli signals.
However, as time increases, the estimate \( \kappa_t \) rapidly converges to the true value \( \kappa^\ast \), demonstrating the algorithm's consistency.
The right panels illustrate the Monte Carlo simulation of the regret achieved by Algorithm~\ref{alg:cap}.
Under both time horizons, the regret shows sublinear growth and is bounded by order $\mathcal{O}(\ln^2 T)$. 
Curve fitting analysis further confirms that,  especially at the longer time scale, the curve of order $\mathcal{O}(\ln^2 T)$ provides a better fit than that of order $\mathcal{O}(\ln T)$, which supports our theoretical regret analysis.

Furthermore, we compare Algorithm~\ref{alg:cap} with a myopic benchmark strategy that always posts at $\frac{1}{\kappa_t}$, where $\kappa_t$ is the current estimate of $\kappa^\ast$ at time $t$. 
As shown in Figure~\ref{fig:myopic_comparison}, while the myopic strategy is still able to learn the true parameter over time, the corresponding regret grows linearly with time. In contrast, Algorithm~\ref{alg:cap} achieves sublinear regret, implying the advantages of employing the ``optimal'' (computed from the estimate $\kappa_t$) policies in reducing regret. 
Note that this is the regret of a risk-averse market maker with risk aversion parameter $\phi=\$10^{-6}$.

\begin{figure}
    \centering
    \includegraphics[width=1\linewidth]{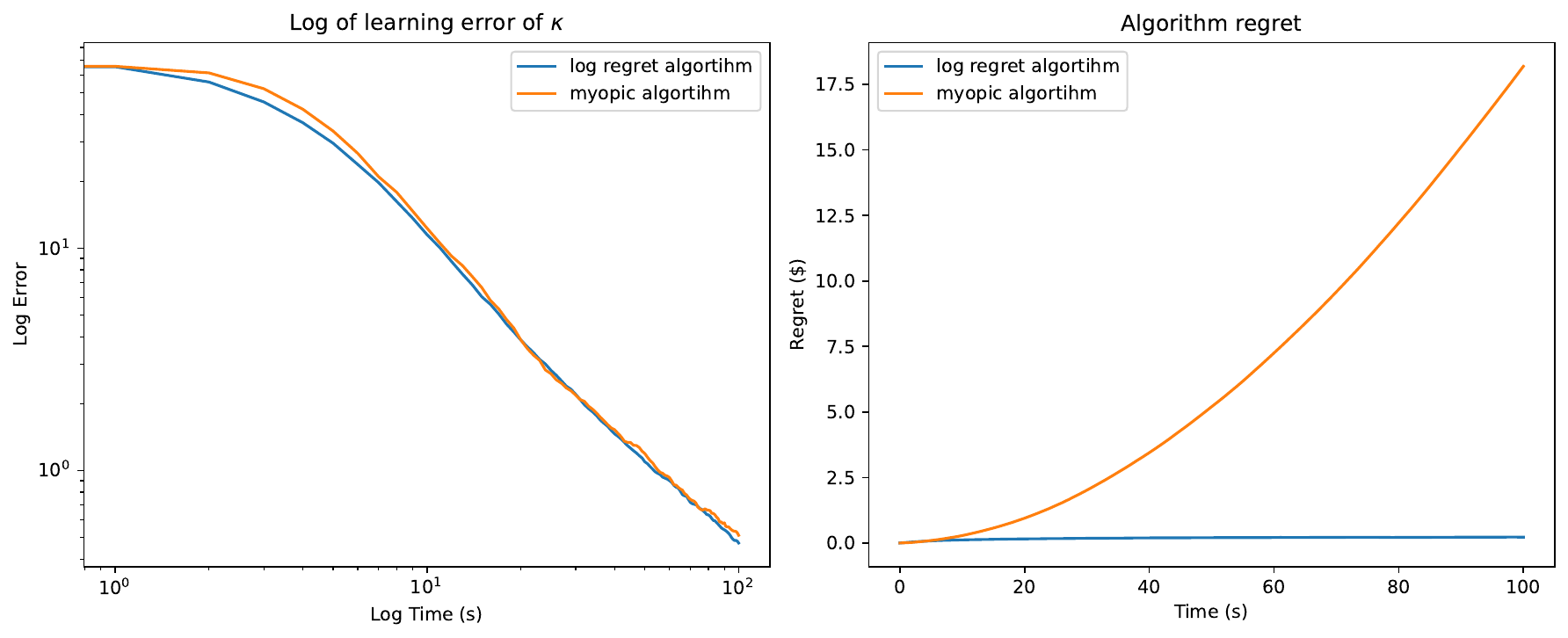}
    \caption{Performance comparison between Algorithm~\ref{alg:cap} and a myopic benchmark strategy that posts at \( 1 / \kappa_i \) using the current estimate. Left: Log-log plot of the estimation error \( |\kappa_t - \kappa^\ast| \). Right: Regret over time.}
    \label{fig:myopic_comparison}
\end{figure}

\subsection{Non-stationary market} 
Financial markets are typically non-stationary.
To handle the non-stationarity of $\kappa$, we incorporate two classical techniques in our learning algorithm: a sliding-window (SW) approach and an exponential-weighted-moving-average (EWMA) approach. 

The sliding window (SW) method is as follows.
The index set of recent data is defined as $\mathcal{I}_i := \{ j \leq i \mid t_i - t_j \leq w \}$. 
We then obtain $\kappa_i^{(w)}$ by numerically solving
$\frac{\mathrm{d}}{\mathrm{d} \kappa} \tilde \ell_i^{(w)} (\kappa) = 0$, where the expression for the derivative  $\frac{\mathrm{d}}{\mathrm{d} \kappa} \tilde \ell_i^{(w)}$ is given by~\eqref{derivative of the log-likelihood function}, but computed using only the data points indexed by  $\mathcal{I}_i$, i.e. from $j = \inf \mathcal{I}_i$ to $j = i$.
Otherwise, the algorithm is the same as Algorithm~\ref{alg:cap}.

The exponential-weighted-moving-average (EWMA) method uses the following log-likelihood 
\begin{equation*}
\ell_N^{EWMA}(\kappa) = \sum_{n=1}^N e^{- \alpha (t_N - t_n)}\left( - \kappa \delta_n Y_n + (1-Y_n) \log (1 - e^{- \kappa \delta_n}) \right) \, ,
\end{equation*}
where $\alpha$ is the weighting parameter.
This is then regularised as in~\eqref{regularised log-likelihood function} where $\ell_N$ is replaced by $\ell_N^{EWMA}$.
The algorithm incorporating the EWMA method simply replaces the log-likelihood function in Algorithm~\ref{alg:cap} with regularisation of $\ell_N^{EWMA}$.

Figure~\ref{non-stationary} illustrates the performance of the learning algorithms in a non-stationary market environment.
We use the same parameters as those used for Figure~\ref{regret}, except in this non-stationary setting, the true value of $\kappa$ changes every $50$ seconds, following the sequence $[20, 30, 10, 40, 25]$. 
The SW algorithm employs a sliding window of $30$ seconds, while the EWMA algorithm sets $\alpha=0.1$. 
The left panel shows how the estimated $\kappa$ (green and orange curves) tracks the true, piecewise constant $\kappa$ (blue dashed line) over time using each method. We observe that after each shift in the true value, the estimate gradually converge to the new value, with a short delay in both methods. 
The right panel presents the regrets of the two methods over time. As expected, the regret grows approximately at an order of $\ln^2 T$ in each regime where $\kappa$ is fixed. However, each change in $\kappa$ introduces a noticeable increase in the regret due to the lag in adaptation. Once the estimates converge to the new value, the growth of regret slows down again. 
Even though with our choice of window size and $\alpha$ the SW method achieves lower regret than the EWMA method this does not imply that the SW method is better; we expect there will be a value of $\alpha$ where the EWMA method achieves the same regret.

\begin{figure}
\centering
\includegraphics[width=1\textwidth]{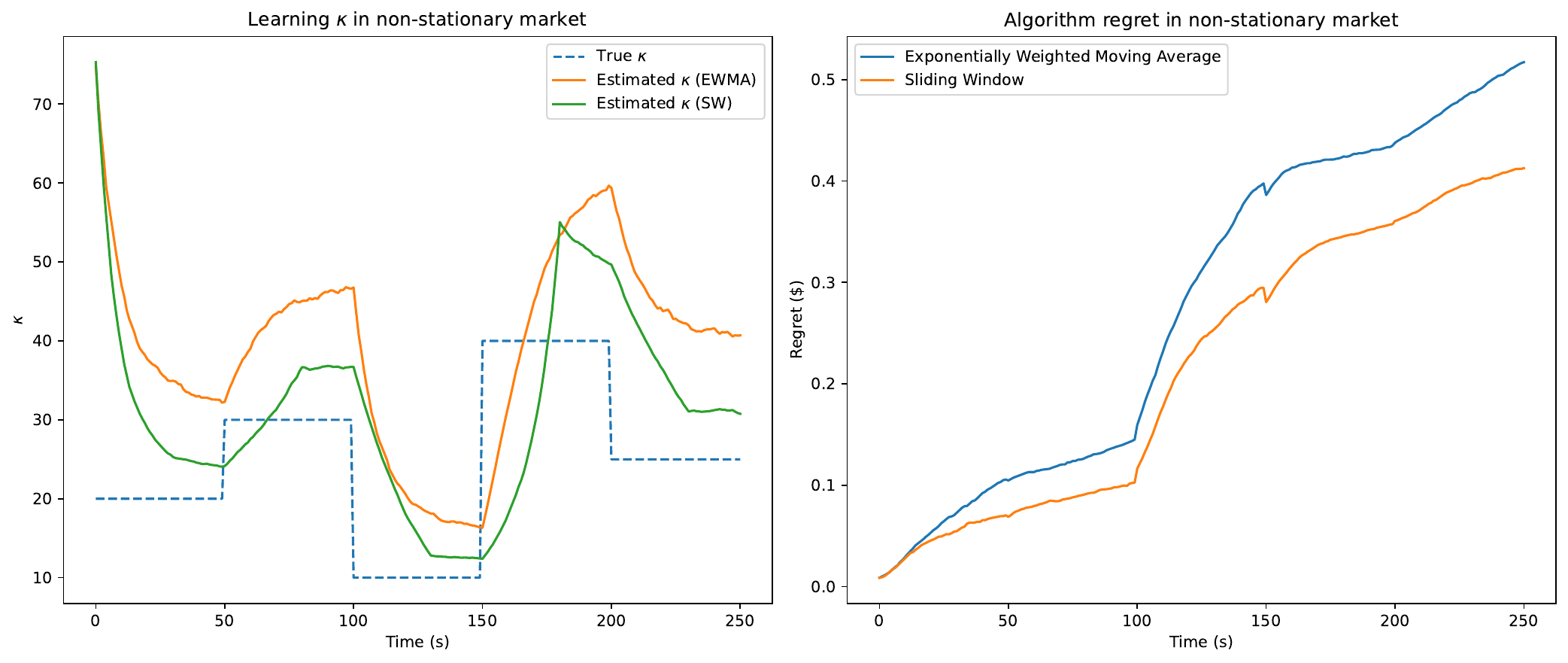}
\caption{Learning $\kappa$ in the non-stationary market}
\label{non-stationary}
\end{figure}

\subsection{The dependence of the regret bound on model parameters}
In this section, we implement numerical experiments to quantify the dependence of the regret constant $C_1$ given in Theorem~\ref{regret theorem} on the key model parameters, $\phi, \lambda^\pm, \bar K, \underline K$ and $\kappa_0 - \kappa^\ast$.
Figure~\ref{dependency of C1} (left panel) presents how $C_1$ varies with $\phi$ and $\lambda$, where we set $\lambda^\pm = \lambda$ in the simulation.
We used $500$ scenarios and $T=100$ seconds with the random seed fixed for each parameter combination.
We observe that $C_1$ increases as both $\phi$ and $\lambda$ increase.
This is expected as larger $\phi$ implies a higher penalty for the holding inventory, while a larger $\lambda$ corresponds to a higher frequency of incoming market orders.
The right panel shows the dependence of $C_1$ on $\bar{K}$, $\underline{K}:= \frac{1}{\bar{K}}$ and on $\kappa_0 - \kappa^\ast$.
As shown in the figure, a higher $\bar{K}$ or higher $\kappa_0 - \kappa^\ast$ leads to a larger constant $C_1$ in the asymptotic expression for regret.

\begin{figure}
\centering
\includegraphics[width=1\textwidth]{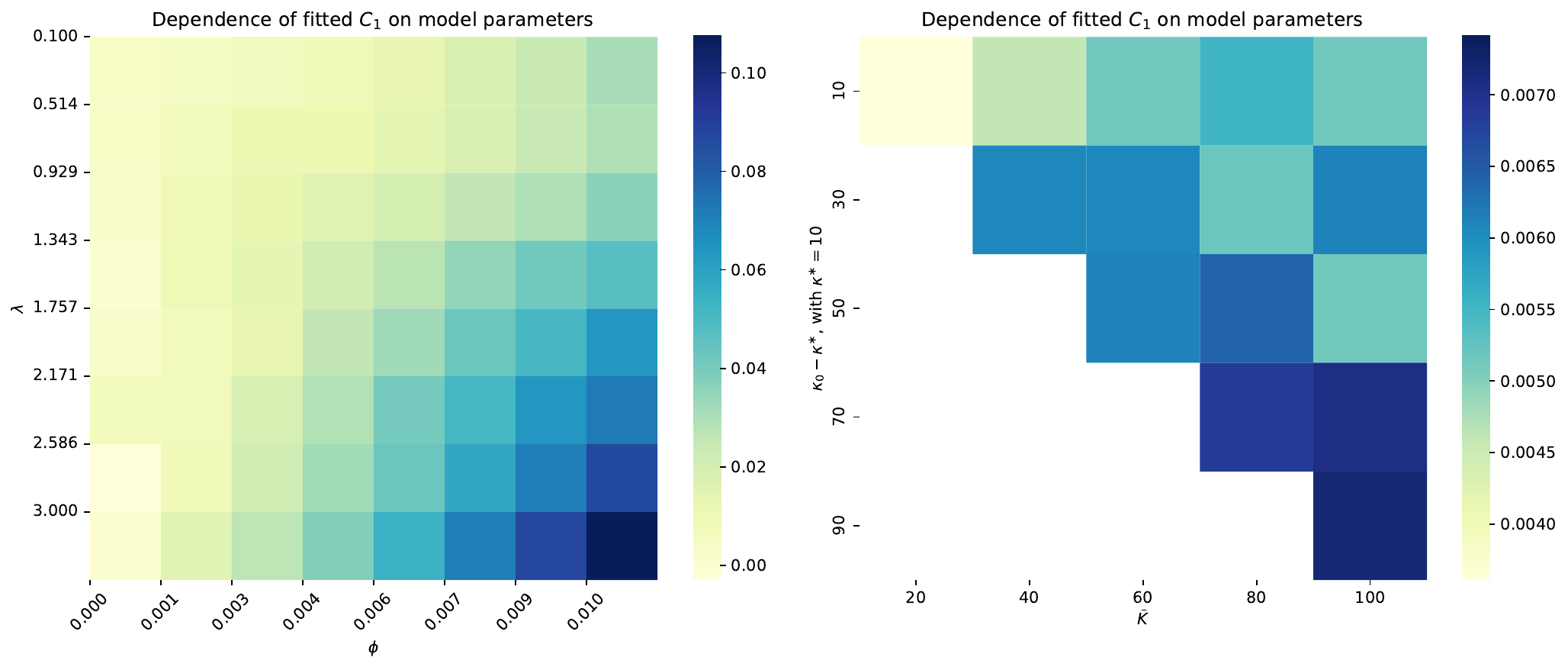}
\caption{Dependence of regret bound constant $C_1$ on model parameters}
\label{dependency of C1}
\end{figure}

\section{Conclusion}
In this paper, we introduced and analysed the ergodic formulation of the Avellaneda--Stoikov market making model.
We established explicit solutions to the ergodic Hamilton--Jacobi--Bellman (HJB) equation and thus derived the optimal ergodic Markov controls.
We've further shown that under the ergodic optimal control there is a unique invariant distribution for the market maker's inventory and that any initial distribution converges exponentially fast to the equilibrium one. 
This allowed us to establish the regret upper bound of $\mathcal{O}(\ln^2 T)$ for learning the unknown price sensitivity of liquidity takers $\kappa^\ast$.
Our work extends the known results on the market making model by providing a rigorous analysis of the ergodic setting and offering a robust solution for parameter learning.
The numerical experiments further validate the theoretical results, confirming the robustness of the proposed algorithm. 

A number of interesting questions have not been addressed in this paper and are left for future work. 
In particular, a key extension of the market making framework presented here accounts for adverse selection. 
Learning the parameters modelling adverse selection and establishing a regret bound would be interesting. 
Further, it would be interesting to compare this approach to a more classical RL algorithms where the optimal policy is learned directly. 
One would conjecture that as long as the market is behaving as the model postulates (up to the unknown parameter) using the optimal control derived from the maximum likelihood performs better. 
However, should the environment deviate from the model it's possible that the pure RL approach will outperform the method proposed here.

\appendix 
\section{Proofs} \label{appendix}

\subsection{Proof of Theorem \ref{exist and unique for finite-time}} \label{proof theorem exist and unique finite}

We first prove the properties of the Hamiltonian function $H$ given by~\eqref{hamiltonian function}.
\begin{lemma} [Hamiltonian function] \label{property of hamiltonian}
For the Hamiltonian function $H$, we have 
\begin{enumerate}[(i)]
\item \label{property of hamiltonian item 1} $\forall (q, \boldsymbol{p}) \in \Omega^{Q} \times \mathbb{R}^2$, $H(q, \boldsymbol{p})$ is finite. 

\item \label{property of hamiltonian item 2} $\forall \boldsymbol{p} \in \mathbb{R}^2$, $\exists \, \delta^{\pm, *} \in \mathbb R^2$ such that 
$$H(q, \boldsymbol{p}) = H(q, \delta^{\pm, *}, \boldsymbol{p}) \, .
$$

\item \label{property of hamiltonian item 3} $p_1 \mapsto H(q, (p_1,p_2)))$ is strictly increasing for any $q \in \Omega^{Q}$ and $p_2 \in \mathbb{R}$; and $p_2 \mapsto H(q, (p_1,p_2)))$ is strictly increasing for any $q \in \Omega^{Q}$ and $p_1 \in \mathbb{R}$.  

\item \label{property of hamiltonian item 4} $\boldsymbol{p} \mapsto H(q, \boldsymbol{p})$ is locally Lipschitz for any $q \in \Omega^{Q}$. 
\end{enumerate}    
\end{lemma}
\begin{proof}
\eqref{property of hamiltonian item 1} \eqref{property of hamiltonian item 2} Given $\forall (q, \boldsymbol{p}) \in  \Omega^{Q} \times \mathbb{R}^2$, we have 
\begin{equation*}
\begin{split}
H(q, \boldsymbol{p}) & = \sup_{\delta^{+} \in \mathbb R} \Big\{ \lambda^{+} e^{-\kappa^{+} \delta^{+}} (p_1 + \delta^{+}) \Big\} + \sup_{\delta^{-} \in \mathbb R} \Big\{ \lambda^{-} e^{-\kappa^{-} \delta^{-}} (p_2 + \delta^{-}) \Big\} - \phi q^2 \, . 
\end{split}
\end{equation*}

Consider the function $c(\delta): \delta \mapsto c(\delta) \in \mathbb{R}$ as $c(\delta) = \lambda e^{- \kappa \delta} (p + \delta)$,
where $\lambda, \kappa$ and $p$ are given. 
By letting the first derivative of $c(\delta)$ be $0$ and checking that the second derivative is less than $0$, we know that $c(\delta)$ attains its supremum at $\delta^{*} = \frac{1}{\kappa} - p$. 
Therefore, with the fact that $\phi q^2 \geq 0$, 
\begin{equation*}
H(q, \boldsymbol{p}) \leq  \lambda^{+} e^{-\kappa^{+} \delta^{+,*}} (p_1 + \delta^{+,*}) + \lambda^{-} e^{-\kappa^{-} \delta^{-, *}} (p_2 + \delta^{-, *}) < + \infty. 
\end{equation*}
Moreover, the supremum in the right hand side can be attained at $\delta^{\pm, *}$ given by the above expression, hence the results. 

\eqref{property of hamiltonian item 3} Given $q \in \Omega^{Q}$ and $p_2 \in \mathbb{R}$, consider $p_1$ and $p'_1$ such that $p_1 > p'_1$. Since $\lambda^{+}e^{-\kappa^{+} \delta^{*}} > 0$, we have 
\begin{equation*}
\begin{split}
\lambda^{+} e^{-\kappa^{+} \delta^{+}} (p_1 + \delta^{+}) \mathds{1}_{q > \underline{q}} + & \lambda^{-} e^{-\kappa^{-} \delta^{-}} (p_2 + \delta^{-}) \mathds{1}_{q < \bar{q}} - \phi q^2 \, > \, \\ & \lambda^{+} e^{-\kappa^{+} \delta^{+}} (p'_1 + \delta^{+}) \mathds{1}_{q > \underline{q}} + \lambda^{-} e^{-\kappa^{-} \delta^{-}} (p_2 + \delta^{-}) \mathds{1}_{q < \bar{q}} - \phi q^2 
\end{split}
\end{equation*}
By taking the supremum on both sides, we have $H(q, (p_1, p_2)) > H(q, (p'_1,p_2))$. Similarly, we have $H(q, (p_1, p_2)) > H(q, (p_1,p'_2))$ if $p_2 > p'_2$ given $q \in \Omega^{Q}$ and $p_1 \in \mathbb{R}$. 

\eqref{property of hamiltonian item 4} We consider $\boldsymbol{p} = (p_1, p_2)$ and $\boldsymbol{p'} = (p'_1, p'_2)$, then
\begin{equation*}
\begin{split}
\left| H(q, \boldsymbol{p}) - H(q, \boldsymbol{p'}) \right| & =  \bigg| \sup_{\delta^{+} \in \mathbb R}  \Big\{ \lambda^{+} e^{-\kappa^{+} \delta^{+}} (p_1 + \delta^{+}) \Big\} + \sup_{\delta^{-} \in \mathbb R} \Big\{ \lambda^{-} e^{-\kappa^{-} \delta^{-}} (p_2 + \delta^{-}) \Big\}  \\ 
& \quad - \sup_{\delta^{+} \in \mathbb R} \Big\{ \lambda^{+} e^{-\kappa^{+} \delta^{+}} (p'_1 + \delta^{+}) \Big\} - \sup_{\delta^{-} \in \mathbb R} \Big\{ \lambda^{-} e^{-\kappa^{-} \delta^{-}} (p'_2 + \delta^{-}) \Big\} \bigg| \\
& = \left| \frac{\lambda^{+} e^{-1}}{\kappa^{+}}  \big( e^{-\kappa^{+} p_1} - e^{-\kappa^{+} p'_1}  \big) + \frac{\lambda^{-} e^{-1}}{\kappa^{-}}  \big( e^{-\kappa^{-} p_2} - e^{-\kappa^{-} p'_2} \big) \right| \, .
\end{split}
\end{equation*}
With the fact that $x \mapsto e^{x}$ is locally Lipschitz, we conclude the result. 
\end{proof}

To prove the existence of the unique solution $u$ to the HJB equation~\eqref{finite-time hjb} on $(-\infty, T]$, we use Lemma~\ref{property of hamiltonian} (4) to prove the locally Lipschitz of the ODE and apply~\cite[Theorem 3.3]{gueant2020optimal}.  

Moreover, by a standard verification argument, we know that $u = v_r$, which $v_r$ is the value function of the discounted finite-time-horizon control problem~\eqref{reduced finite-time value function}, hence the result of Theorem~\ref{exist and unique for finite-time}.

\subsection{Proof of Theorem \ref{existence of discounted infinite}} \label{proof theorem exist infinite}
\begin{proof}
We first prove that $f(Q_t; \delta_t)$ for any $(Q_t)_{t \geq 0}$ taking values in $\Omega^Q$ and $\delta_t^\pm \in \mathcal{A}$ is bounded. 
We know that, by using Lemma~\ref{property of hamiltonian} (1), there exists a constant $\bar C \in \mathbb{R}$ such that, 
\begin{equation} \label{upper bound of f}
\begin{split}
f(Q_t; \tilde \delta_t^{\pm}) & = \tilde \delta_t^{+} \lambda^{+} e^{- \kappa^{+} \tilde \delta_t^{+}} + \tilde \delta_t^{-} \lambda^{-} e^{- \kappa^{-} \tilde \delta_t^{-}} - \phi (Q^{\tilde \delta^{\pm}}_t)^2 \\ 
& \leq \sup_{\delta^{\pm} \in \mathbb{R}^2}  \Big( \delta^{+} \lambda^{+} e^{- \kappa^{+} \delta^{+}} + \delta^{-} \lambda^{-} e^{- \kappa^{-} \delta^{-}} \Big) \\  
& \leq \bar C \, , 
\end{split}
\end{equation}
and by the boundedness from below of the admissible control set $\mathcal{A}$, there exists $\underline C \in \mathbb{R}$ such that 
\begin{equation} \label{lower bound of f}
\begin{split}
f(Q_t; \tilde \delta_t^{\pm}) & \geq \inf_{\delta^{\pm} \in \mathcal{A}}  \Big( \delta_t^{+} \lambda^{+} e^{- \kappa^{+} \delta_t^{+}} + \delta_t^{-} \lambda^{-} e^{- \kappa^{-} \delta_t^{-}} \Big)  - \phi (\max(\bar q, \underline q))^2 \\  
& \geq \underline C \, .
\end{split}
\end{equation}

To see that, $\forall q \in \Omega^{Q}$ and $t \in \mathbb{R}^{+}$, $v_r(q) = \lim_{T \rightarrow + \infty} v_r(t,q;T)$, we apply~\cite[Proposition 4.1]{gueant2020optimal} by using the running reward function $f$ is bounded. 
Moreover, by a standard verification argument, we know that $v_r$ is the solution to the HJB equation~\eqref{hjb for discounted infinite-time-horizon model}, hence the result of Theorem~\ref{existence of discounted infinite}.
\end{proof}

\subsection{Proof of Theorem \ref{ergodic constant}} \label{proof theorem ergodic control problem}
To prove Theorem~\ref{ergodic constant}, we first prove the following lemma. 
\begin{lemma} \label{boundness of rv_r(q)}
Let $v_r(q)$ be given by~\eqref{reduced infinite-time-horizon model}, we have
\begin{enumerate}[(i)]
\item \label{boundness of rv_r(q) item 1} $\exists \, C_1 \in \mathbb{R}^{+}$ such that $|r v_r(q)| \leq C_1$ for any $q \in \Omega^{Q}$ and $r \in \mathbb{R}^{+}$.

\item \label{boundness of rv_r(q) item 2} $\exists \, C_2 \in \mathbb{R}^{+}$ such that $| v_r(\hat{q}) - v_r(q) | \leq C_2 |\hat q - q |$ for any $q, \hat{q} \in \Omega^{Q}$ and $r \in \mathbb{R}^{+}$.
\end{enumerate}
\end{lemma}

\begin{proof}
By using the fact that the running reward function $f$ is bounded~\eqref{upper bound of f} and~\eqref{lower bound of f}, we can apply~\cite[Lemma 4.3(1)]{gueant2020optimal} to get statement~(\ref{boundness of rv_r(q) item 1}) of the lemma. 

To prove statement~\eqref{boundness of rv_r(q) item 2}, we first define the stopping time $\tau(q, \hat q)$ for the process $Q_t$ under a control $\delta^{\pm} \in \mathcal{A}$ with initial condition $Q_0 = q$ as 
\begin{equation*}
\tau := \inf\{ t \left. \right\vert Q_t^{\delta^{\pm},q} = \hat q\} \, . 
\end{equation*}
Since the dynamics of $Q_t^{\delta^{\pm},q}$ can be equivalently represented by a continuous-time Markov chain that is irreducible and recurrent (see the detailed discussion in Section~\ref{proof lemma ergodic control stability}, with $\psi$ replaced by $\delta^\pm$), it follows that $\mathbb{E}[\tau] < +\infty$. 

Let us consider $\delta^{\pm, \varepsilon} \in \mathcal{A}[0, \tau]$ with $\varepsilon > 0$ such that 
\begin{equation*}
v_r(q) - \varepsilon \leq \mathbb{E}_q \Big[
\int_0^{\tau} e^{-rt} f(Q_t;\delta^{\pm, \varepsilon}) \, \mathrm{d} t  + e^{-r \tau} v_r( \hat q) \Big] \, .
\end{equation*}

By~\eqref{upper bound of f} and~\eqref{lower bound of f}, there exists $\bar C, \underline C$ such that $0 \leq f(q; \delta^{\pm}) - \underline C \leq \bar C - \underline C$ for any $q \in \Omega^{Q}, \delta^{\pm} \in \mathcal{A}$. Therefore, with the fact that $e^{-rt} \leq 1$ for $t \in [0, \tau]$
\begin{equation*}
\begin{split}
v_r(q) - \varepsilon - \frac{\underline C}{r} & \leq  \mathbb{E}_q \Big[
\int_0^{\tau} e^{-rt} \big( f(Q_t;\delta^{\pm})  - \underline C \big)\, \mathrm{d} t  + e^{-r \tau} \big( v_r(\hat q) - \frac{\underline C}{r} \big)\Big] \\
& \leq \mathbb{E} \Big[ \int_0^{\tau} e^{-rt}  \big( f(Q_t;\delta^{\pm})  - \underline C \big) \, \mathrm{d} t \Big] +  \mathbb{E} \big[ e^{-r \tau } \big] \big( v_r(\hat q) - \frac{\underline C}{r} \big) \\
& \leq (\bar C  - \underline C ) \mathbb E[\tau] + v_r(\hat q) - \frac{\underline C}{r} \, .
\end{split}
\end{equation*}
Therefore, 
\begin{equation*}
\begin{split}
\frac{v_r(q) - v_r(\hat q)}{|q - \hat q|} & \leq  \frac{1}{|q - \hat q|}\Big( (\bar C  - \underline C ) \mathbb E[\tau] + \varepsilon \Big) \leq (\bar C  - \underline C ) \mathbb E[\tau] + \varepsilon, \quad q, \hat{q} \in \Omega^{Q}, q \neq \hat q \, . 
\end{split}
\end{equation*}
Since $\mathbb{E} [\tau] < + \infty$ and $q \in \Omega^{Q}$, by letting $\varepsilon \rightarrow 0$, we conclude that $v_r(q) - v_r(\hat q)$ is bounded from above. By simply changing the order of $q$ and $\hat q$, we conclude the lower boundedness. Hence, we can find $C_2 \in \mathbb R$ such that 
$$| v_r(\hat{q}) - v_r(q) | \leq C_2 |\hat q - q|, \quad \forall q, \hat{q} \in \Omega^{Q} , r \in \mathbb{R}^{+} \, .$$
\end{proof}

Now we are ready to prove Theorem~\ref{ergodic constant}. 

\begin{proof}
In the proof we follow the ideas from~\cite[Proposition 4.6, 4.7]{gueant2020optimal}.
As $\left|r v_r(q)\right| \leq C_1$ and $\left| v_r(q) - v_r(0) \right| \leq C_2 \bar q = C'_2, \, \forall q \in \Omega^{Q}$ by Lemma~\ref{boundness of rv_r(q)}, we can consider a sequence $(r_n)_{n \in \mathbb N}$ converging towards $0$ such that the sequences $\big(r_n v_{r_n}(q) \big)_{n \in \mathbb N}$ and $\big(v_{r_n}(q) - v_{r_n}(0) \big)_{n \in \mathbb N}$ are convergent for $q \in \Omega^{Q}$. Let $\hat \gamma(q)$ denote the limit of the sequence $\big(r_n v_{r_n}(q) \big)_{n \in \mathbb N}$, we have
\begin{equation*}
0 = \lim_{n \rightarrow +\infty} r_n \big( v_{r_n}(q) - v_{r_n}(0) \big) = \lim_{n \rightarrow +\infty} r_n v_{r_n}(q) - \lim_{n \rightarrow +\infty} r_n v_{r_n}(0) = \hat \gamma(q) - \hat \gamma(0) \, . 
\end{equation*}
Let $\hat \gamma \in \mathbb R$ be a constant, then $\hat \gamma(q) = \hat \gamma(0) = \hat \gamma$ for any $q \in \Omega^{Q}$.  

Next, we prove that $\hat \gamma$ is independent of the sequence $(r_n)_{n \in \mathbb N}$. From the HJB equation (\ref{hjb for discounted infinite-time-horizon model}), we have, for the sequence  $v_{r_n}(q)$,
\begin{equation*}
0 =  - r_n v_{r_n}(q) + H \Big(q, (v_{r_n}(q') - v_{r_n}(q))_{q' \in \{q-1, q+1\}}\Big), \quad \forall q \in \Omega^{Q} \, .
\end{equation*}
Let $\hat{v}(q) = \lim_{n \rightarrow +\infty} \big(v_{r_n}(q) - v_{r_n}(0) \big)$. As the sequence $\big(v_{r_n}(q) - v_{r_n}(0) \big)_{n \in \mathbb N}$ is convergent, we know that $\hat{v}(q)$ is well defined. Take $n \rightarrow + \infty$ on both sides, we have 
\begin{equation} \label{ergodic hjb equation proof}
0 = - \hat \gamma + H \Big(q, (\hat{v}(q') - \hat{v}(q))_{q' \in \{q-1, q+1\}}\Big), \quad \forall q \in \Omega^{Q} \, .
\end{equation}
We then consider another sequence $(r'_n)_{n \in \mathbb N}$ converging towards $0$ that leads to another limit $\eta \in \mathbb R$ for the sequence $\big(r'_n v_{r'_n}(q) \big)_{n \in \mathbb N}$, i.e. $\lim_{n \rightarrow + \infty} r'_n v_{r'_n}(q) = \eta, \, \forall q \in \Omega^{Q}$. Let $\hat w(q) = \lim_{n \rightarrow + \infty} \big(v_{r'_n}(q) - v_{r'_n}(0)\big)$, then we have 
\begin{equation*}
0 = - \eta + H \Big(q, (\hat{w}(q') - \hat{w}(q))_{q' \in \{q-1, q+1\}}\Big), \quad \forall q \in \Omega^{Q}. 
\end{equation*}
Let $z(q) = \hat w(q) - \hat v(q)$. Since the domain for $z(q)$ is bounded, we know that the supremum and infimum exist. Let us denote $\bar z = \sup_{q \in \Omega^{Q}} z(q)$, $\underline z = \inf_{q \in \Omega^{Q}} z(q)$ and
$\varepsilon = \frac{\hat \gamma - \eta}{ \bar z - \underline z + 1}$. 

Let us first assume $\hat \gamma > \eta$ and prove that $\hat \gamma \leq \eta$ by contradiction. By the definition of $\varepsilon$, we have, for $\forall q \in \Omega^{Q}$, 
\begin{equation*}
\begin{split}
0 \leq \varepsilon &  (\bar z - z(q) + 1) \leq \hat \gamma - \eta  \\
& = H \Big(q, (\hat{v}(q') - \hat{v}(q))_{q' \in \{q-1, q+1\}}\Big) - H \Big(q, (\hat{w}(q') - \hat{w}(q))_{q' \in \{q-1, q+1\}}\Big) \, .
\end{split}
\end{equation*}
Therefore, 
\begin{equation*}
\begin{split}
- \varepsilon \hat w(q) + H \Big(q, (\hat{w}(q') - & \hat{w}(q))_{q' \in \{q-1, q+1\}} \Big) \leq \\ & -\varepsilon(\hat v(q) + \bar z + 1) + H \Big(q, (\hat{v}(q') - \hat{v}(q))_{q' \in \{q-1, q+1\}}\Big) \, . 
\end{split}
\end{equation*}
By using the comparison principle, see~\cite[Lemma 4.4]{gueant2020optimal}, we know that $\hat v(q) + \bar z + 1 \leq \hat w(q)$, for $\forall q \in \Omega^{Q}$, which indicates a contradiction with the definition of $\bar z$. Hence, we have $\hat \gamma \leq \eta$. By simply changing the order of $\hat \gamma$ and $\eta$, we can obtain that $\hat \gamma \geq \eta$. Therefore, we conclude that $\hat \gamma = \eta$, i.e. $\hat \gamma$ is independent of the sequence $(r_n)_{n \in \mathbb N}$.
\end{proof}

\subsection{Proof of Theorem \ref{solve gamma}} \label{proof theorem solve gamma}
\begin{proof}
Let us define $\mu (t, q) = v_0(T-t, q; T)$, then $\mu (t, q)$ satisfies the following equation 
\begin{equation} \label{reverse time hjb equation}
- \partial_t \mu(t, q) + H \Big(q, (\mu(t, q') - \mu(t, q))_{q' \in \{q-1, q+1\}}\Big) = 0, \quad \forall (t, q) \in [0, +\infty) \times \Omega^{Q} \, ,
\end{equation}
subject to the initial condition $\mu(0, q) = G(q)$ with $G$ given by~\eqref{terminal condition}. We consider $U(t,q) = \mu(t,q) - \hat \gamma t$ for $(t, q) \in [0, +\infty) \times \Omega^{Q}$ with $\hat \gamma$ given in Theorem~\ref{ergodic constant}. We proceed to prove that $U(t,q)$ is bounded. 

Let us consider $\varphi^c (t, q) = \hat \gamma t + \hat v(q) + c, \, \forall (t,q) \in [0, +\infty) \times \Omega^{Q}$ with the constant $c \in \mathbb R$ and
$\hat \gamma$ given in Theorem \ref{ergodic constant} and $\hat v(q) = \lim_{n \rightarrow + \infty} \big(v_{r_n}(q) - v_{r_n}(0)\big)$, where $(r_n)_{n \in \mathbb N}$ is a sequence converging towards $0$ such that $\big(v_{r_n}(q) - v_{r_n}(0) \big)_{n \in \mathbb N}$ is convergent. 
From the ergodic HJB equation (\ref{ergodic hjb equation}), we have,
\begin{equation*}
\begin{split}
- \partial_t \varphi^c (t, q) + & H \Big(q, (\varphi^c(t, q') - \varphi^c(t, q))_{q' \in \{q-1, q+1\}} \Big) = \\ & - \hat \gamma + H \Big(q, (\hat v(q') - \hat v(q))_{q' \in \{q-1, q+1\}}\Big) = 0, \quad \forall (t,q) \in [0, +\infty) \times \Omega^{Q}  \, . 
\end{split}
\end{equation*}
Let us denote $c_1 = \inf_{q \in \Omega^{Q}} \big(G(q) - \hat v(q)\big)$, where $G(q)$ is the initial condition for the equation (\ref{reverse time hjb equation}). Then we have, $\forall q \in \Omega^{Q}$, 
\begin{equation*}
\varphi^{c_1} (0, q) = \hat v(q) + \inf_{q \in \Omega^{Q}} \big(G(q) - \hat v(q)\big) \leq G(q) = v(T, q;T) = \mu(0,q) \, . 
\end{equation*}
By using the comparison principle (see \cite[Proposition 3.2]{gueant2020optimal}), we know that 
$\varphi^{c_1} (t, q) \leq \mu(t,q)$ for $(t,q) \in [0, +\infty) \times \Omega^{Q}$. 
We then consider $c_2 = \sup_{q \in \Omega^{Q}} \big(G(q) - \hat v(q)\big)$, and clearly $\varphi^{c_2} (0, q) \geq \mu(0,q)$. By using the comparison principle again, we have $\varphi^{c_2} (t, q) \geq \mu(t,q)$. Therefore,
\begin{equation*}
\varphi^{c_1} (t, q) \leq  \mu(t,q) \leq \varphi^{c_2} (t, q), \quad \forall (t,q) \in [0, +\infty) \times \Omega^{Q} \, .
\end{equation*}
By the expression of $\varphi^{c_1}$ and $\varphi^{c_2}$, we have
\begin{equation*}
\hat v(q) + c_1 \leq \mu(t,q) - \hat \gamma t = U(t,q) \leq \hat v(q) + c_2 \, . 
\end{equation*}
As $q \mapsto \hat v(q)$ is well defined that has been proved in Appendix~\ref{proof theorem ergodic control problem} and $G(q)$ is bounded by definition, we can conclude that $U(t,q)$ is bounded on $(t,q) \in [0, +\infty) \times \Omega^{Q}$. 

Now let us consider $U(T, q) = \mu(T,q) - \hat \gamma T$ with $T \in [0, +\infty)$. Take the limit $T \rightarrow + \infty$, we have 
\begin{equation*}
\lim_{T \rightarrow + \infty}  \frac{1}{T} \mu (T, q) = \lim_{T \rightarrow + \infty}  \frac{1}{T} \big( U(T, q) + \hat \gamma T \big) \, . 
\end{equation*}
Since $U(T, q)$ is bounded, therefore
\begin{equation*}
\lim_{T \rightarrow + \infty}  \frac{1}{T} \mu (T, q) = \hat \gamma = \lim_{T \rightarrow + \infty}  \frac{1}{T} v_0(T-T, q; T) = \lim_{T \rightarrow + \infty}  \frac{1}{T} v_0(0, q; T) \, , 
\end{equation*}
where $v_0(0, q; T)$ satisfies the HJB equation (\ref{finite-time hjb}) with $r=0$ 

So far we've proved that, there exists a constant $\hat \gamma \in \mathbb{R}$ such that $\lim_{r \rightarrow 0} rv_r(q) = \hat \gamma = \lim_{T \rightarrow + \infty} \frac{1}{T} v_0(0, q; T)$ for any $q \in \Omega^Q$,
which addresses one of the challenges in the ergodic control problem \cite{arisawa1998ergodic}.
The next step is to prove that $\hat \gamma = \gamma$, where $\gamma$ is the ergodic constant defined in the ergodic control problem~\eqref{reduced ergodic reward functional}. 

By definition, we have
\begin{equation*}
\begin{split}
\gamma & = \sup_{\delta \in \mathcal{A}} \lim_{T \rightarrow + \infty}  \frac{1}{T} \mathbb{E}_{q} \Big[ \int_0^T f(Q^{\delta^{\pm}}_t; \delta^{\pm}) \, \mathrm{d} t \Big] \\
& \leq \lim_{T \rightarrow + \infty} \sup_{\delta \in \mathcal{A}} \frac{1}{T} \mathbb{E}_{q} \Big[ \int_0^T f(Q^{\delta^{\pm}}_t; \delta^{\pm}) \, \mathrm{d} t \Big] \\
& = \lim_{T \rightarrow + \infty}  \frac{1}{T} v_0(0, q; T) = \hat \gamma \, . 
\end{split}
\end{equation*}
By Theorem~\ref{solve ergodic control problem} and Proposition~\ref{Existence and Uniqueness for Ergodic Optimal Control}, we know that there actually exists an optimal Markov control $\psi^\pm \in \mathcal{A}$ such that 
\begin{equation*}
\begin{split}
\hat \gamma & = \lim_{T \rightarrow + \infty}  \frac{1}{T} \mathbb{E}_{q} \Big[ \int_0^T f(Q^{\psi^\pm}_t; \psi^\pm) \, \mathrm{d} t \Big] \\
& \leq \sup_{\delta \in \mathcal{A}} \lim_{T \rightarrow + \infty}  \frac{1}{T} \mathbb{E}_{q} \Big[ \int_0^T f(Q^{\delta^{\pm}}_t; \delta^{\pm}) \, \mathrm{d} t \Big] = \gamma \, , 
\end{split}
\end{equation*}
hence $\hat \gamma = \gamma$, and all $\hat \gamma$ will be substituted by $\gamma$ in the later context. 
\end{proof}

\subsection{Proof of Theorem \ref{solve ergodic constant}} \label{proof theorem solve ergodic constant}
\begin{proof}
By Theorem \ref{solution of finite-time-horizon model}, the solution to $\boldsymbol v_0(0; T)$ can be given by 
$\boldsymbol v_0(0;T) = \frac{1}{\kappa} \ln (e^{T \boldsymbol{A}} \cdot \boldsymbol z)$. 
As the subdiagonal and the superdiagonal elements of $\boldsymbol A$ are $\lambda^{-} e^{-1}$, $\lambda^{+} e^{-1} > 0$, we can find a real and symmetric tridiagonal matrix $\boldsymbol{J}$ whose entries are given by
\begin{align} \label{matrix J to prove gamma}
\boldsymbol J_{ij} = \begin{cases}
\boldsymbol{A}_{ij}, & \text{if $i = j$},\\
\sqrt{\lambda^{+} \lambda^{-}} e^{-1}, & \text{if i = j-1 or j+1}, \\
0, & \text{otherwise} \, 
\end{cases}
\end{align}
which is similar to $\boldsymbol A$. Hence, by \cite{parlett1998symmetric}, $\boldsymbol{J}$ (and $\boldsymbol{A}$) can be diagonalised with distinct eigenvalues. Let $n = \bar q - \underline q +1$, $\lambda_1, \lambda_2, ... ,\lambda_n$ be $n$ real eigenvalues of $\boldsymbol{A}$ with $\lambda_1 > \lambda_2 > ... > \lambda_n$ and $\boldsymbol{\Lambda}$ be the diagonal matrix of $\boldsymbol{A}$ such that $\boldsymbol{A} = \boldsymbol{P}^{-1} \boldsymbol{\Lambda} \boldsymbol{P}$, where $\boldsymbol{P}$'s columns are the corresponding eigenvectors. 

By Theorem \ref{solve gamma}, we have 
\begin{equation*}
\lim_{T \rightarrow + \infty} \frac{1}{T} v_0(0, q; T) = \gamma, \quad \forall q \in \Omega^{Q} \, . 
\end{equation*}
Therefore, by considering the vector form of $\boldsymbol{v}_0(0; T)$ and using Theorem~\ref{solution of finite-time-horizon model}, we obtain
\begin{equation*}
\begin{split}
\lim_{T \rightarrow \infty } \frac{\boldsymbol{v}_0(t=0;T)}{T} & = \frac{1}{\kappa} \lim_{T \rightarrow \infty } \frac{1}{T} \ln (e^{T \boldsymbol{A}} \cdot \boldsymbol z) = \frac{1}{\kappa} \lim_{T \rightarrow \infty } \frac{1}{T} \ln ( \sum_{i=0}^{\infty} \frac{(T\boldsymbol{A})^n}{n!}\cdot \boldsymbol z)  \\
& = \frac{1}{\kappa} \lim_{T \rightarrow \infty } \frac{1}{T} \ln (\boldsymbol P e^{T \boldsymbol \Lambda } \boldsymbol P^{-1} \cdot \boldsymbol z) \, .
\end{split}
\end{equation*}
Let us denote $\boldsymbol{P} = \begin{bmatrix}
P_{11}&  ... & P_{1n}\\ 
P_{21}&   ... &  P_{2n}\\ 
\vdots &   ...&  \vdots \\
P_{n1}& ... & P_{nn}
\end{bmatrix}_{n \times n}$ and $\boldsymbol{P}^{-1} \cdot \boldsymbol{z} = \begin{bmatrix}
K_{1}\\ 
\vdots \\ 
K_{n}
\end{bmatrix}_{n \times 1}$, then
\begin{equation*}
\begin{split}
\lim_{T \rightarrow \infty } \frac{\boldsymbol{v}_0(t=0;T)}{T} & = \frac{1}{\kappa} \lim_{T \rightarrow \infty } \frac{1}{T} \ln \begin{bmatrix}
\sum_{i=1}^n K_1 P_{1i} e^{\lambda_i T}\\ 
\sum_{i=1}^n K_2 P_{2i} e^{\lambda_i T}\\
\vdots \\ 
\sum_{i=1}^n K_n P_{ni} e^{\lambda_i T}
\end{bmatrix}_{n \times 1} \\
& = \frac{1}{\kappa} [\lambda_1, \lambda_1, ..., \lambda_1]^\top \, , 
\end{split}
\end{equation*}
hence the result. 

\end{proof}

\subsection{Proof of Lemma \ref{gamma kappa Lipschitz differentiable}} \label{continuity ergodic constant}
\begin{proof}
As discussed in Appendix~\ref{proof theorem solve ergodic constant}, we can find a real and symmetric tridiagonal matrix $\boldsymbol{J}$ that is similar to $\boldsymbol{A}$, whose eigenvalues are simple, i.e. the algebraic multiplicity is $1$. Moreover, it is obvious that $\kappa \mapsto \boldsymbol{J}(\kappa)$ given by~\eqref{matrix J to prove gamma} is $C^{\infty}([\underline K, \bar K])$. 
By \cite{horn2012matrix, kriegl2003differentiable, xu2016derivatives}, $\kappa \mapsto \lambda_{max}(\kappa)$ can be parameterised smoothly on $\kappa \in [\underline K, \bar K]$, i.e. $\lambda_{max}(\kappa)$ is $C^{\infty}([\underline K, \bar K])$. 
By Theorem \ref{solve ergodic constant}, we know that $\gamma(\kappa) = \frac{\lambda_{max}(\kappa)}{\kappa}$. 
Therefore, we can conclude that $\gamma(\kappa)$ is $C^2([\underline K, \bar K])$ and $\frac{\mathrm{d}^2}{\mathrm{d} \kappa^2}_{\kappa} \gamma(\kappa)$ is bounded on the compact set $\kappa \in [\underline K, \bar K]$. 
\end{proof}

\subsection{Proof of Proposition \ref{uniqueness of ergodic control}} \label{proof prop uniquess of ergodic control}
\begin{proof}
To prove Proposition \ref{uniqueness of ergodic control}, we recommend to follow the idea in \cite[Proposition 4.7]{gueant2020optimal} and use the properties of the Hamiltonian function $H$ in Lemma~\ref{property of hamiltonian}.  
\end{proof}

\subsection{Proof of Proposition \ref{Existence and Uniqueness for Ergodic Optimal Control}} \label{proof Existence and Uniqueness for Ergodic Optimal Control}
\begin{proof}
Notice that the right hand side of
\begin{equation} \label{solve optimal ergodic control}
\psi^\pm(q) \in \argmax_{\delta^{\pm}} H \Big(q, (\hat v(q') - \hat v(q))_{q' \in \{q-1, q+1\}}\Big), \quad \forall q \in \Omega^{Q}, 
\end{equation}
is invariant under constant shifts in the solution $\hat v$ and hence the optimal control $\psi^\pm$ for the ergodic control problem is uniquely given by expression~\eqref{feedback control} by simply solving the right hand side of~\eqref{solve optimal ergodic control} with the convention that $\psi^\pm (q) = + \, \infty$ for $q = \underline q, \bar q$, respectively. 
Moreover from~\eqref{hamiltonian function} it is easy to see that $\psi^\pm$ is single-valued and given by the result. 
\end{proof}

\subsection{Proof of Theorem \ref{solve ergodic control problem}} \label{proof theorem solve ergodic control problem}
\begin{proof}
By Proposition~\ref{Existence and Uniqueness for Ergodic Optimal Control}, the ergodic HJB equation (\ref{ergodic hjb equation}) can be rewritten as 
\begin{equation*} 
\begin{split}
0  =  - \phi q^2 - \gamma & + \frac{\lambda^{+}}{\kappa} \exp{\big(-1 - \kappa \hat v(q) + \kappa \hat v(q-1)\big)} \mathds{1}_{q > \underline{q}} \\ & + \frac{\lambda^{-}}{\kappa} \exp{\big(-1 - \kappa \hat v(q) + \kappa \hat v(q+1)\big)} \mathds{1}_{q < \bar{q}} \, . 
\end{split}
\end{equation*}

By using Theorem~\ref{solve ergodic constant}, we get an explicit solution for $\gamma = \frac{\lambda_{max}}{\kappa}$. Therefore, to solve the ergodic HJB equation, the next step is to solve $\hat{v}$.
Let $\hat v(q) = \frac{1}{\kappa} \ln \hat \omega(q)$, then
\begin{equation} \label{appendix equation}
- \kappa (\phi q^2 + \gamma) + \lambda^{+} e^{-1} \frac{\hat \omega(q-1)}{\hat \omega (q)} \mathds{1}_{q > \underline{q}} + \lambda^{-} e^{-1}\frac{\hat \omega(q+1)}{\hat \omega (q)} \mathds{1}_{q < \bar{q}} = 0 \, . 
\end{equation}

Let $n = (\bar{q} - \underline{q} + 1)$, $\boldsymbol{\hat \omega} = [\hat{\omega}(\bar{q}), \hat{\omega}(\bar{q}-1), ... , \hat{\omega}(\underline{q})]^\top$ be an $n$-dim vector and $\boldsymbol C$ be an $n$ - square matrix given by 
\scriptsize
\begin{align*}
\boldsymbol C = \begin{bmatrix}
- \kappa (\phi \bar{q}^2 + \gamma) & \lambda^{+} e^{-1} & 0 &  &  & ... & \\ 
\lambda^{-} e^{-1} & - \kappa \big(\phi (\bar{q}-1)^2 + \gamma \big)& \lambda^{+} e^{-1 } & & & ... & \\ 
& & &  ... & \\
& ... & &  & \lambda^{-} e^{-1} & - \kappa \big(\phi (\underline{q}+1)^2 +\gamma \big)& \lambda^{+} e^{-1}\\ 
& ... & & & 0 & \lambda^{-} e^{-1} & - \kappa (\phi \underline{q}^2 + \gamma)
\end{bmatrix}\,.
\end{align*}
\normalsize
Therefore, the equation (\ref{appendix equation}) can be written in a matrix form as
\begin{equation} \label{homogeneous}
\boldsymbol{C} \boldsymbol{\hat \omega} = 0 \, .
\end{equation}

Due to the fact of $\gamma = \frac{\lambda_{max}}{\kappa}$, we observe that $\boldsymbol C = \boldsymbol{A} - \lambda_{max} \boldsymbol{I}$, where the matrix $\boldsymbol{A}$ is given in Theorem \ref{solution of finite-time-horizon model}, $\lambda_{max}$ is the largest eigenvalue of $\boldsymbol{A}$ and  $\boldsymbol{I}$ is the identity matrix. As $\boldsymbol{A}$ has $n$ distinct eigenvalues as proved in Appendix~\ref{proof theorem solve ergodic constant}, hence $\text{rank} (\boldsymbol{C}) = n - 1$.  
Therefore, the null space of the matrix $\boldsymbol{C}$ has dimension $1$ by rank-nullity theorem, implying that the solution $\boldsymbol{\hat{\omega}}$ to the homogeneous equation~\eqref{homogeneous} is unique (up to multiplicative factors). 
Indeed, $\boldsymbol{\hat{\omega}}$ is the eigenvector corresponding the dominant eigenvalue of matrix $A$, which is a Metzler matrix (non-negative off-diagonal entries). By Perron–Frobenius theorem, the eigenvector to the dominant eigenvalue is positive, which completes the proof.
\end{proof}

\subsection{Proof of Lemma~\ref{connection to alternative regret}}
\label{equivalent regret}
\begin{proof}
The equation~\eqref{regret gap bounded} is a step in the proof of Theorem~\ref{performance gap} in a simpler case. 
We know that $\psi^{\kappa^\ast}$ is the optimal control for the ergodic market making model under $\kappa^\ast$ satisfying~\eqref{feedback control}. Hence, from the ergodic HJB equation~\eqref{ergodic hjb equation}, we have
\begin{equation*}
\begin{split}
0 = - \gamma(\kappa^\ast) - \phi q^2  +  & \lambda^{+} e^{-\kappa^\ast \psi^{\kappa^\ast, +}(q)} \Big( \hat{v}^{\kappa^\ast}(q-1) - \hat{v}^{\kappa^\ast}(q) + \psi^{\kappa^\ast, +}(q)\Big)  \mathds{1}_{q > \underline q}\\
+ & \lambda^{-} e^{-\kappa^\ast \psi^{\kappa^\ast, -}(q)} \Big( \hat{v}^{\kappa^\ast}(q+1) - \hat{v}^{\kappa^\ast}(q) + \psi^{\kappa^\ast, -}(q)\Big) \mathds{1}_{q < \bar q} \, . 
\end{split}
\end{equation*}
Therefore,
\begin{equation*}
\begin{split}
\gamma(\kappa^\ast) =  \lambda^{+} \psi^{\kappa^\ast, +}(q) e^{-\kappa^\ast \psi^{\kappa^\ast, +}(q)} & + \lambda^{-} \psi^{\kappa^\ast, -}(q) e^{-\kappa^\ast \psi^{\kappa^\ast, -}(q)}  - \phi q^2  \\ 
& + \lambda^{+} e^{-\kappa^\ast \psi^{\kappa^\ast, +}(q)} \big( \hat{v}^{\kappa^\ast}(q-1) - \hat{v}^{\kappa^\ast}(q) \big) \mathds{1}_{q > \underline q} \\ 
& + \lambda^{-} e^{-\kappa^\ast \psi^{\kappa^\ast, -}(q)} \big( \hat{v}^{\kappa^\ast}(q+1) - \hat{v}^{\kappa^\ast}(q) \big) \mathds{1}_{q < \bar q} \, , 
\end{split}
\end{equation*}
where we ignore the indicator functions in the first line 
and use $\psi^{\kappa^\ast, \pm}(q) e^{-\kappa^\ast \psi^{\kappa^\ast, \pm}(q)} = 0$ for $q = \bar q, \underline q$ respectively. Moreover, we notice that the first line satisfies~\eqref{reward function}, therefore
\begin{equation*}
\begin{split}
\bigg| \gamma(\kappa^\ast) T -  \mathbb{E}_q \Big[ & \int_0^T f(Q_t^{\psi^{\kappa^\ast}; \kappa^\ast}, \psi^{\kappa^\ast}; \kappa^\ast) \, \mathrm{d} t \Big] \bigg|  \\ 
& = \bigg| \mathbb{E}_q \Big[\int_0^T \big(\gamma(\kappa^\ast)  - f(Q_t^{\psi^{\kappa^\ast}; \kappa^\ast}, \psi^{\kappa^\ast}; \kappa^\ast) \big)\, \mathrm{d} t \Big] \bigg| \\ 
& = \bigg| \mathbb{E}_q \Big[\int_0^T \Big( \lambda^{+} e^{-\kappa^\ast \psi^{\kappa^\ast, +}(q)} \big( \hat{v}^{\kappa^\ast}(q-1) - \hat{v}^{\kappa^\ast}(q) \big) \mathds{1}_{q > \underline q} \\
& \qquad \qquad \qquad \quad + \lambda^{-} e^{-\kappa^\ast \psi^{\kappa^\ast, -}(q)} \big( \hat{v}^{\kappa^\ast}(q+1) - \hat{v}^{\kappa^\ast}(q) \big) \mathds{1}_{q < \bar q} \Big) \, \mathrm{d} t \Big] \bigg| \, .  
\end{split}
\end{equation*}
Let $\pi^{\kappa^\ast}$ denote the equilibrium of the optimal ergodic market making model under parameter $\kappa^\ast$ and function $h$ be
\begin{equation*}
\begin{split}
h(\kappa^\ast, q) = \lambda^{+} e^{-\kappa^\ast \psi^{\kappa^\ast, +}(q)} \big( \hat{v}^{\kappa^\ast}(q-1) & - \hat{v}^{\kappa^\ast}(q) \big) \mathds{1}_{q > \underline q} \\
&  + \lambda^{-} e^{-\kappa^\ast \psi^{\kappa^\ast, -}(q)} \big( \hat{v}^{\kappa^\ast}(q+1) - \hat{v}^{\kappa^\ast}(q) \big) \mathds{1}_{q < \bar q} \,. 
\end{split}
\end{equation*}
By Lemma~\ref{boundness of rv_r(q)} (2), $h$ is bounded by $\bar{h} \in \mathbb R^{+}$. 
From a simpler version of Proposition~\ref{expectation under equi is 0} by substituting $\psi^\kappa$ to $\psi^{\kappa^\ast}$, Lemma~\ref{convergence theorem} and Lemma~\ref{boundness of rv_r(q)} (2), we have 
\begin{equation*}
\begin{split}
\bigg| \gamma(\kappa^\ast) T -  \mathbb{E}_q \Big[ & \int_0^T f(Q_t^{\psi^{\kappa^\ast}; \kappa^\ast}, \psi^{\kappa^\ast}; \kappa^\ast) \, \mathrm{d} t \Big] \bigg| \\
& = \bigg| \int_0^T \int_{\Omega^Q} \Big( \lambda^{+} e^{-\kappa^\ast \psi^{\kappa^\ast, +}(q)} \big( \hat{v}^{\kappa^\ast}(q-1) - \hat{v}^{\kappa^\ast}(q) \big) \mathds{1}_{q > \underline q} \\
& \qquad \qquad \qquad \quad + \lambda^{-} e^{-\kappa^\ast \psi^{\kappa^\ast, -}(q)} \big( \hat{v}^{\kappa^\ast}(q+1) - \hat{v}^{\kappa^\ast}(q) \big) \mathds{1}_{q < \bar q} \Big) \, \mathrm{d} t \, \mathrm{d} \pi^{\kappa^\ast}_t \bigg| \\
& \leq \bigg| \bar h \int_0^T \int_{\Omega^Q} \frac{h(\kappa^\ast,q)}{\bar h}  (\mathrm{d} \pi^{\kappa^\ast}_t  - \mathrm{d} \pi^{\kappa^\ast}  )\, \mathrm{d} t\bigg| \\
& \leq \bar h \int_0^T  \left\| \pi^{\kappa^\ast}_t  - \pi^{\kappa^\ast} \right\|_{TV} \, \mathrm{d} t \leq \frac{\bar h C}{ - \ln \alpha}  \,  ,
\end{split}
\end{equation*}
with constants $C > 0$ and $0 < \alpha <1$ independent of $T$, hence the result. 
\end{proof}

\subsection{Proof of Lemma~\ref{lemma: linear ode for discounted model}} \label{proof lemma: linear ode for discounted model}
\begin{proof}
Let $w(t,q)$ satisfy the linear ODE~\eqref{linear ode for discounted model} subject to the terminal condition~\eqref{terminal condition}, i.e. 
\begin{equation*}
0 = \partial_t w(t,q) - r w(t,q) + H \Big(q, \psi^{\kappa},  (w(t,q') - w(t,q))_{q' \in \{q-1, q+1\}} ; \kappa^\ast \Big),\, \forall q \in \Omega^{Q},
\end{equation*}
and $w(T,q) = G(q)$. 
Clearly, the equation~\eqref{linear ode for discounted model} is a linear ODE, hence there exists $\boldsymbol{w} \in C^1([0,T]; \mathbb R^{n})$, which is a solution to~\eqref{linear ode for discounted model}. 

Let us consider the following stochastic process, and we omit the superscript for $Q_t^{\psi^{\kappa}; \kappa^\ast}$ for notational simplicity. 
$$
X(s) = e^{-r(s- t)} w(s, Q_s) + \int_t^{s} e^{-r (u -t)}f(Q_u, \psi^\kappa_u; \kappa^\ast) \, \mathrm{d} u 
$$
We know that $Q_t^{\psi^{\kappa}; \kappa^\ast}$ follows the SDE \eqref{inventory} with market parameter $\kappa^\ast$ and control $\psi^{\kappa}$, i.e. 
\begin{equation*}
\begin{split}
dQ_t^{\psi^{\kappa}; \kappa^\ast} & =  dN_t^{\psi^{\kappa}, -} - dN_t^{\psi^{\kappa}, +} \\
& = \big( \lambda^{+} e^{-\kappa^\ast \psi^{\kappa, -}} - \lambda^{-} e^{-\kappa^\ast \psi^{\kappa, + },} \big)dt  + d\tilde{N}_{t}^{\psi^{\kappa}, -} - d\tilde{N}_{t}^{\psi^{\kappa}, +}, 
\end{split}
\end{equation*}
where $\tilde{N}_{t}^{\psi^{\kappa}, \pm}$ are compensated Poisson processes.

By It\^{o}'s formula, 
\small
\begin{equation*}
\begin{aligned}
& d X(s) =  \, e^{-r(s- t)} \Big\{ \partial_s w(s, Q_s) - r w(s, Q_s) + \lambda^{+} e^{-\kappa^\ast \psi^{\kappa, +}} \left( w(s, Q_s-1) - w(s, Q_s) \right) \mathds{1}_{Q_s > \underline q} \\
& + \lambda^{-} e^{-\kappa^\ast \psi^{\kappa, -}} \left( w(s,Q_s+1) - w(s, Q_s) \right) \mathds{1}_{Q_s < \bar q} + f(Q_s, \psi^\kappa; \kappa^\ast) \Big\} ds \\
& + e^{-r(s- t)} \Big\{ \left( w(s, Q_s-1) - w(s, Q_s) \right) \mathds{1}_{Q_s > \underline q} d\tilde{N}_{s}^{+} + \left( w(s,Q_s+1) - w(s, Q_s) \right) \mathds{1}_{Q_s < \bar q} d\tilde{N}_{s}^{-} \Big\} \\
& =  \, e^{-r(s- t)} \Big\{ \partial_s w(s, Q_s) - r w(s, Q_s) + H \big(Q_s, \psi^\kappa, (w(s, Q_s') - w(s, Q_s))_{Q_s' \in \{Q_s-1, Q_s+1\}} \big) \Big\} ds \\
& + e^{-r(s- t)} \Big\{ \left( w(s, Q_s-1) - w(s, Q_s) \right) \mathds{1}_{Q_s > \underline q} d\tilde{N}_{s}^{+} + \left( w(s,Q_s+1) - w(s, Q_s) \right) \mathds{1}_{Q_s < \bar q} d\tilde{N}_{s}^{-} \Big\}.
\end{aligned}
\end{equation*}
\normalsize
where the last equality comes from the definition of $H(\cdot; \kappa^\ast)$~\eqref{L function} and $f(\cdot; \kappa^\ast)$ \eqref{reward function}. Take the integral and expectation for $X(s)$, we have
\small
\begin{equation*}
\begin{split}
\mathbb{E} \big[ & X(T) \big|  Q_t  = q \big]  = \mathbb{E} \big[ X(t) \big|  Q_t = q \big] \\ &
+ \int_t^T e^{-r(s - t)} \Big( \partial_s w - r w + H \big(Q_s,\psi^{\kappa}, (w(s, Q'_s) - w(s, Q_s))_{Q'_s \in \{Q_s-1, Q_s+1\}} \big) \Big) \, \mathrm{d} s \\
\end{split}
\end{equation*}
\normalsize
As $w(t,q)$ satisfies~\eqref{linear ode for discounted model} and the terminal condition~\eqref{terminal condition}, therefore, 
\begin{equation*}
\begin{split}
w(t,q) = \mathbb{E} \big[ X(t) \big|  Q_t = q \big] & = \mathbb{E} \big[ X(T) \big| Q_t = q \big] \\
& = \mathbb{E}_{q} \Big[ \int_t^T e^{-r(u-t)} f(Q_u, \psi^{\kappa}; \kappa^\ast) \, \mathrm{d} u + e^{-r(T-t)} G(Q_T) \Big].
\end{split}
\end{equation*}
Hence $w(t,q) = v_r^{\psi^\kappa}(t, q; T; \kappa^\ast)$ by~\eqref{discounted finite-time-horizon misspcefication model}.
\end{proof}

\subsection{Proof of Proposition~\ref{linear PDE for gamma}} \label{proof Proposition linear PDE for gamma}
\begin{proof}
First, let $v_r^{\psi^{\kappa}}(q; \kappa^\ast)$ be the expected reward in the discounted infinite-time-horizon setting, where the true price sensitivity parameter is $\kappa^\ast$ but the market maker uses the strategy $\psi^{\kappa}$ given by~\eqref{feedback control} with parameter $\kappa$, i.e.
\begin{equation} \label{discounted infinite horizon model with kappa misspecific}
v_r^{\psi^{\kappa}}(q; \kappa^\ast) = \mathbb{E}_{q} \Big[ \int_0^{+ \infty} e^{-rt} f(Q_t^{\psi^{\kappa}; \kappa^\ast}, \psi^{\kappa}; \kappa^\ast) \, \mathrm{d} t \Big],
\end{equation}
where $f(\cdot; \kappa^\ast)$ is given by~\eqref{reward function}. Then we claim that $v_r^{\psi^{\kappa}}(q; \kappa^\ast)$ satisfies the following linear system 
\begin{equation} \label{linear system discounted infinite}
0 = - r v_r^{\psi^{\kappa}}(q;\kappa^\ast) + H \Big(q, \psi^{\kappa}, (v_r^{\psi^{\kappa}}(q';\kappa^\ast) - v_r^{\psi^{\kappa}}(q; \kappa^\ast))_{q' \in \{q-1, q+1\}} ; \kappa^\ast \Big), \quad \forall q \in \Omega^{Q} \, . 
\end{equation}
We would like to provide a sketch of proof for the claim. First, the existence of $v_r^{\psi^\kappa}(q; \kappa^\ast)$ defined by~\eqref{discounted infinite horizon model with kappa misspecific} can follow the proof in Section~\ref{proof theorem exist infinite} for Theorem~\ref{existence of discounted infinite} but substituting the Hamiltonian function $H$~\eqref{hamiltonian function} for $H(\cdot; \kappa^\ast)$~\eqref{L function}. Let $w(q)$ be a solution for the linear system~\eqref{linear system discounted infinite}. Consider 
\begin{equation*}
X(s) = e^{-rs} w(Q_s) + \int_0^s e^{-rt} f(Q_t, \psi^\kappa; \kappa^\ast) \, \mathrm{d} t \, . 
\end{equation*}
By It\^{o}'s formula\
\small
\begin{equation*}
\begin{aligned}
d X(s) = & \, e^{-rs} \Big\{ - r w(Q_s) + \lambda^{+} e^{-\kappa^\ast \psi^{\kappa, +}_s} \left( w(Q_s-1) - w(Q_s) \right) \mathds{1}_{Q_s > \underline q} \\
& \quad + \lambda^{-} e^{-\kappa^\ast \psi^{\kappa, -}_s} \left( w(Q_s+1) - w(Q_s) \right) \mathds{1}_{Q_s < \bar q} + f(Q_s, \psi^\kappa; \kappa^\ast) \Big\} ds \\
& + e^{-rs} \Big\{ \left( w(Q_s-1) - w(Q_s) \right) \mathds{1}_{Q_s > \underline q} d\tilde{N}_{s}^{+} + \left( w(Q_s+1) - w(Q_s) \right) \mathds{1}_{Q_s < \bar q} d\tilde{N}_{s}^{-} \Big\} \\
= & \, e^{-rs} \Big\{ - r w(Q_s) + H \left( Q_s, \psi^\kappa, (w(Q'_s) - w(Q_s))_{Q'_s \in \{Q_s-1, Q_s+1\}} \right) \Big\} ds \\
& + e^{-rs} \Big\{ \left( w(Q_s-1) - w(Q_s) \right) \mathds{1}_{Q_s > \underline q} d\tilde{N}_{s}^{+} + \left( w(Q_s+1) - w(Q_s) \right) \mathds{1}_{Q_s < \bar q} d\tilde{N}_{s}^{-} \Big\}.
\end{aligned}
\end{equation*}
\normalsize
where the last equality comes from the definition of $H(\cdot; \kappa^\ast)$~\eqref{L function} and $f(\cdot; \kappa^\ast)$ \eqref{reward function}. Take the integral and expectation for $X(s)$, we have
\small
\begin{equation*}
\begin{split}
\mathbb{E} \big[ X(T)  \big| Q_0 = q \big] & = \mathbb{E} \big[ X(0) \big|  Q_0 = q \big] \\ &
+ \int_0^T e^{-rt} \Big( - r w(Q_t) + H \big(Q_t ,\psi^{\kappa}, (w(Q_t') - w(Q_t))_{Q_t' \in \{Q_t-1, Q_t+1\}} \big) \Big) \, \mathrm{d} t \\
\end{split}
\end{equation*}
\normalsize
As $w(q)$ satisfies~\eqref{linear system discounted infinite}, therefore, 
\begin{equation*}
\begin{split}
w(q) = \mathbb{E} \big[ X(0) \big|  Q_0 = q \big] & = \mathbb{E} \big[ X(T) \big| Q_0 = q \big] \\
& = \mathbb{E}_{q} \Big[ \int_0^T e^{-rt} f(Q_t, \psi^{\kappa}; \kappa^\ast) \, \mathrm{d} t + e^{-rT} w(Q_T) \Big].
\end{split}
\end{equation*}
Take limit $T \rightarrow +\infty$ on both sides, with the fact that the limit exists, i.e. $w(q) \in \mathbb R, \forall q \in \Omega^{Q}$
\begin{equation*}
w(q) = \mathbb{E}_{q} \Big[ \int_0^{+\infty} e^{-rt} f(Q_t, \psi^{\kappa}; \kappa^\ast) \, \mathrm{d} t \Big]. 
\end{equation*}
Hence $w(q) = v_r^{\psi^\kappa}(q; \kappa)$ defined by~\eqref{discounted infinite horizon model with kappa misspecific}.

Now, we would like to to show that given $\gamma(\kappa; \kappa^\ast)$ defined by~\eqref{definition gamma(kappa;kappastar)}, it holds that
\begin{equation} \label{gamma(kappa; kappastar) by discounted infinite}
\lim_{r \rightarrow 0} r v_r^{\psi^{\kappa}}(q;\kappa^\ast)  = \gamma(\kappa; \kappa^\ast), \quad \forall q \in \Omega^{Q} \, ,
\end{equation}
where $v_r^{\psi^{\kappa}}(q;\kappa^\ast)$ is defined by~\eqref{discounted infinite horizon model with kappa misspecific}.

Let us start with the following lemma. 
\begin{lemma} \label{bounded for rv_r under model mispecification}
\begin{enumerate} [(1)]
\item $\exists \, C_1 \in \mathbb{R}^{+}$ such that $\left|r v_r^{\psi^\kappa}(q; \kappa^\ast) \right| \leq C_1$ for $\forall q \in \Omega^{Q}$ and $r \in \mathbb{R}^{+}$.

\item $\exists \, C_2 \in \mathbb{R}^{+}$ such that $\left| v_r^{\psi^\kappa}(\hat{q}; \kappa^\ast) - v_r^{\psi^\kappa}(q; \kappa^\ast) \right| \leq C_2 |\hat q - q |$ for $\forall q, \hat{q} \in \Omega^{Q}$ and $r \in \mathbb{R}^{+}$.
\end{enumerate}
\end{lemma}

As $\psi^\kappa \subset \mathcal{A}$, i.e. the collection of all Markov controls optimal for the ergodic control problem is a subset of the admissible control, the running reward function $f(\cdot; \kappa^\ast)$ is bounded. By~\eqref{reward function}, we have
\begin{equation} \label{upper bound for f}
\begin{split}
f(Q_t, \psi^\kappa; \kappa^\ast) & = \lambda^{+} \psi^{\kappa, +} e^{- \kappa^\ast \psi^{\kappa, +}} + \lambda^{-} \psi^{\kappa, -} e^{- \kappa^\ast \psi^{\kappa, -}} - \phi (Q_t)^2 \\
& \leq \sup_{\psi^\kappa \in \overline{\mathbb{R}}^2} \big( \lambda^{+} \psi^{\kappa, +} e^{- \kappa^\ast \psi^{\kappa, +}} + \lambda^{-} \psi^{\kappa, -} e^{- \kappa^\ast \psi^{\kappa, -}} \big) \\
& = \overline C \, 
\end{split}
\end{equation}
and by the boundedness from below of $\mathcal{A}$, 
\begin{equation} \label{lower bound for f}
\begin{split}
f(Q_t, \psi^\kappa; \kappa^\ast) & = \lambda^{+} \psi^{\kappa, +} e^{- \kappa^\ast \psi^{\kappa, +}} + \lambda^{-} \psi^{\kappa, -} e^{- \kappa^\ast \psi^{\kappa, -}} - \phi (Q_t)^2 \\
& \geq \inf_{\psi^\kappa} \big( \lambda^{+} \psi^{\kappa, +} e^{- \kappa^\ast \psi^{\kappa, +}} + \lambda^{-} \psi^{\kappa, -} e^{- \kappa^\ast \psi^{\kappa, -}} \big) - \phi \max(\bar q, \underline q)^2 \\
& = \underline C \, .
\end{split}
\end{equation}
To prove Lemma~\ref{bounded for rv_r under model mispecification} (1), we first consider that
\begin{equation*}
\begin{split}
v_r^{\psi^\kappa}(q; \kappa^\ast) & = \mathbb{E}_{q} \Big[ \int_0^{+ \infty} e^{-rt} f(Q_t^{\psi^{\kappa}; \kappa^\ast}, \psi^{\kappa}; \kappa^\ast) \, \mathrm{d} t \Big]  \geq \mathbb{E}_{q} \Big[
\int_0^{+\infty}e^{-rt} \underline C \, \mathrm{d} t \Big] = \frac{\underline C}{r} \, . 
\end{split}
\end{equation*}
On the other hand, 
\begin{equation*}
\begin{split}
v_r^{\psi^\kappa}(q; \kappa^\ast)  \leq \mathbb{E}_{q} \Big[
\int_0^{+\infty}e^{-rt} \overline C \, \mathrm{d} t  \Big]  = \frac{\bar C}{r} \, .
\end{split}
\end{equation*}
Hence 
\begin{equation*}
\left| r v_r(q) \right| \leq C_1, \, \forall q \in \Omega^{Q}, 
\end{equation*}
with $C_1 = \max(\left| \underline C \right|, \left| \overline C \right| )$. 

To prove Lemma~\ref{bounded for rv_r under model mispecification} (2), we define the stopping time $\tau(q, \hat q)$ for the process 
$Q_t^{\psi^\kappa;\kappa^\ast}$ with initial condition $Q_0 = q$ as 
\begin{equation*}
\tau := \inf\{ t \left. \right\vert Q_t^{\psi^\kappa;\kappa^\ast} = \hat q\} \, . 
\end{equation*}
Since the dynamics of $Q_t^{\psi^\kappa;\kappa^\ast}$ can be equivalently represented by a continuous-time Markov chain that is irreducible and recurrent (see the detailed discussion in Section~\ref{proof lemma ergodic control stability}), it follows that $\mathbb{E}[\tau] < +\infty$.

By~\eqref{upper bound for f} and~\eqref{lower bound for f}, there exists $\overline C, \underline C$ such that $0 \leq f(q; \psi^\kappa; \kappa^\ast) - \underline C \leq \overline C - \underline C$ for any $q \in \Omega^{Q}$. Therefore,
\begin{equation*}
\begin{split}
v_r^{\psi^\kappa}(q;\kappa^\ast) - \frac{\underline C}{r} & = \mathbb{E}_q \Big[
\int_0^{\tau} e^{-rt} \big( f(Q^{\psi^\kappa; \kappa^\ast}_t, \psi^\kappa; \kappa^\ast)  - \underline C \big)\, \mathrm{d} t  + e^{-r \tau} \big( v_r^{\psi^\kappa}(\hat q; \kappa^\ast) - \frac{\underline C}{r} \big)\Big] \\
& \leq \mathbb{E} \Big[ \int_0^{\tau} e^{-rt}  \big( \overline C  - \underline C \big) \, \mathrm{d} t \Big] +  \mathbb{E} \big[ e^{-r \tau } \big] \big( v^{\psi^\kappa; \kappa^\ast}_r(\hat q) - \frac{\underline C}{r} \big) \\
& \leq (\bar C  - \underline C ) \mathbb E[\tau] + v^{\psi^\kappa; \kappa^\ast}_r(\hat q) - \frac{\underline C}{r} \, . 
\end{split}
\end{equation*}
Therefore, 
\begin{equation*}
\begin{split}
\frac{v^{\psi^\kappa; \kappa^\ast}_r(q) - v^{\psi^\kappa; \kappa^\ast}_r(\hat q)}{|q - \hat q|} & \leq  \frac{1}{|q - \hat q|}\Big( (\bar C  - \underline C ) \mathbb E[\tau] \Big) \leq (\bar C  - \underline C ) \mathbb E[\tau], \, \forall q, \hat{q} \in \Omega^{Q}, q \neq \hat q. 
\end{split}
\end{equation*}
Since $\mathbb{E} [\tau] < + \infty$, we conclude that $\frac{v^{\psi^\kappa; \kappa^\ast}_r(q) - v^{\psi^\kappa; \kappa^\ast}_r(\hat q)}{|q - \hat q|}$ is bounded from above. By simply changing the order of $q$ and $\hat q$, we conclude the lower boundedness. Hence, we can find $C_2 \in \mathbb R^{+}$ such that 
$$\left| v^{\psi^\kappa; \kappa^\ast}_r(q) - v^{\psi^\kappa; \kappa^\ast}_r(\hat q) \right| \leq C_2 |\hat q - q|, \, \forall q, \hat{q} \in \Omega^{Q} , r \in \mathbb{R}^{+} \, .$$. 

With the fact that $\Big|r v_r^{\psi^\kappa}(q; \kappa^\ast) \Big|$ and $\Big| v_r^{\psi^\kappa}(\hat{q}; \kappa^\ast) - v_r^{\psi^\kappa}(q; \kappa^\ast) \Big|$ are bounded, we can follow the discussion in Appendix~\ref{proof theorem ergodic control problem}, i.e. consider a sequence $(r_n)_{n \in \mathbb N}$ converging towards $0$ such that the sequences $\big(r_n v_{r_n}^{\psi^\kappa; \kappa^\ast}(q) \big)_{n \in \mathbb N}$ and $\big(v_{r_n}^{\psi^\kappa; \kappa^\ast}(q) - v_{r_n}^{\psi^\kappa; \kappa^\ast}(0) \big)_{n \in \mathbb N}$ are convergent for $\forall q \in \Omega^{Q}$, then show that there exists $\gamma(\kappa; \kappa^\ast) \in \mathbb{R}$ such that $\gamma(\kappa; \kappa^\ast) = \lim_{r \rightarrow 0} r v_r^{\psi^{\kappa}}(q;\kappa^\ast)$. 
Again by substituting $H$~\eqref{hamiltonian function} for $H(\cdot; \kappa^\ast)$ in the proof of Theorem \ref{solve gamma} in Appendix \ref{proof theorem solve gamma}, we can easily conclude that $\gamma(\kappa; \kappa^\ast)$ given by \eqref{gamma(kappa; kappastar) by discounted infinite} also satisfies
\begin{equation*}
\gamma(\kappa; \kappa^\ast) = \lim_{T \rightarrow + \infty}\frac{1}{T} v_{0}^{\psi^{\kappa}}(t= 0, q; T;  \kappa^\ast) \, .
\end{equation*}

Let us define $\hat{v}^{\psi^\kappa}(q;  \kappa^\ast) = \lim_{r \rightarrow 0} \big( v_r^{\psi^{\kappa}}(q; \kappa^\ast) - v_r^{\psi^{\kappa}} (0;  \kappa^\ast)\big)$ for $q \in \Omega^{Q}$. By the convergent of $\big(v_{r_n}^{\psi^\kappa; \kappa^\ast}(q) - v_{r_n}^{\psi^\kappa; \kappa^\ast}(0) \big)_{n \in \mathbb N}$  under the sequence $(r_n)_{n \in \mathbb N}$ converging towards $0$, we know that $\hat{v}^{\psi^\kappa}(q;  \kappa^\ast)$ is well defined. 

By passing the limit~\eqref{gamma(kappa; kappastar) by discounted infinite} to the equation~\eqref{linear system discounted infinite}, we obtain that
\begin{equation*}
0 = - \gamma(\kappa; \kappa^\ast) + H \Big(q, \psi^{\kappa}, (\hat{v}^{\psi^\kappa, \kappa^\ast}(q') - \hat{v}^{\psi^\kappa, \kappa^\ast}(q))_{q' \in \{q-1, q+1\}} ; \kappa^\ast \Big) \, , \, \forall q \in \Omega^{Q}. 
\end{equation*}
\end{proof}

\subsection{Proof of Proposition~\ref{analytical solution for gamma(kappastar)}} \label{proof Proposition analytical solution for gamma(kappastar)}
\begin{proof}
By~\eqref{linear ode for discounted model}, the linear ODE for $t \mapsto v_r^{\psi^\kappa}(t,q;T; \kappa^\ast)$ can be written in a matrix form. Let $\boldsymbol{v}_r(t) = [v_r^{\psi^\kappa}(t, \bar q; T;\kappa^\ast), ..., v_r^{\psi^\kappa}(t, \underline q; T; \kappa^\ast)]$ be an $n-dim$ vector, where $n = \bar q - \underline q +1$. 
Now, let $\boldsymbol{\tilde{A}}_r$ denote an $n$-square matrix whose rows are labelled from $\bar q$ to $\underline q$ and entries are given by
\begin{align} \label{matrix tilde A}
\boldsymbol{\tilde{A}}_{r}(i,q) = \begin{cases}
- (r + \lambda^{+} e^{- \kappa^\ast \psi^{\kappa, +}(q)} \mathds{1}_{q > \underline q} + \lambda^{-} e^{- \kappa^\ast \psi^{\kappa, -}(q)}\mathds{1}_{q < \bar q} ), & \text{if $i = q$}, \\
\lambda^{+} e^{- \kappa^\ast \psi^{\kappa, +}(q)}, & \text{if $i = q+1$}, \\
\lambda^{-} e^{- \kappa^\ast \psi^{\kappa, -}(q)}, & \text{if $i = q-1$}, \\
0, & \text{otherwise. }
\end{cases}
\end{align}
Let $\boldsymbol{\tilde b}$ be an $n-dim$ vector where each component is 
\begin{equation} \label{vector tilde b}
b_i = \lambda^{+} \psi^{\kappa, +}(i) e^{- \kappa^\ast \psi^{\kappa, +}(i)} \mathds{1}_{i > \underline q} + \lambda^{-} \psi^{\kappa, -}(i) e^{- \kappa^\ast \psi^{\kappa, -}(i)}\mathds{1}_{i < \bar q} - \phi i^2,
\end{equation}
for $i = [\bar q, \bar q -1 , ... ,\underline q]$. 
Then $t \mapsto \boldsymbol{v}_r(t)$ satisfies
\begin{equation} \label{linear ode matrix form}
0 = \partial_t \boldsymbol{v}_r(t) + \boldsymbol{\tilde{A}}_r(\kappa) \boldsymbol{v}_r(t) + \boldsymbol{\tilde b}( \kappa), 
\end{equation}
with the terminal condition $\boldsymbol{v}_r(T; q) = [G(\bar q), \ldots,  G(\underline q)]^\top$, with $G$ given by~\eqref{terminal condition}. Let  $\boldsymbol{G}$ denote the vector $[G(\bar q), \ldots,  G(\underline q)]^\top$. We know that there exists a solution to the linear ODE~\eqref{linear ode matrix form} with the terminal condition on $t \in (- \infty, T]$. 

Now, let us consider the case of $r = 0$. We use $\boldsymbol{\tilde{A}}_{0}(, i)$ to denote the $i-$th column of the coefficient matrix $\boldsymbol{\tilde{A}}_{0}$ with $i \in \{ 1, 2, \dots, n \}$. With the entries given by~\eqref{matrix tilde A} under $r=0$, we can observe that 
\begin{equation*}
\sum_{i=1}^{n-1} \boldsymbol{\tilde{A}}_{0}(, i) = - \boldsymbol{\tilde{A}}_{0}(, n), 
\end{equation*}
which means $\boldsymbol{\tilde{A}}_{0}$ is singular. Hence the solution to~\eqref{linear ode matrix form} under $r=0$ can be given by 
\begin{equation}
\boldsymbol{v}_{0}(t) = e^{(T-t)\boldsymbol{\tilde{A}}_0(\kappa)} \boldsymbol{G} + \int_{t}^T e^{(s-t)\boldsymbol{\tilde{A}}_0(\kappa)} \boldsymbol{\tilde b}( \kappa) \, \mathrm{d} s. 
\end{equation}
By~\eqref{matrix tilde A}, $\boldsymbol{\tilde{A}}_{0}$ is a real tridiagonal matrix with all positive off-diagonal entries. Clearly the eigenvalues of $\boldsymbol{\tilde{A}}_{0}$ are simple, i.e. the algebraic multiplicity is 1. Furthermore, $\boldsymbol{\tilde{A}}_{0}$ is diagonally dominant matrix with all negative diagonal entries, i.e.  
\begin{equation*}
- \boldsymbol{\tilde{A}}_{0}(i, i) = \boldsymbol{\tilde{A}}_{0}(i, i-1) + \boldsymbol{\tilde{A}}_{0}(i, i-1) > 0, 
\end{equation*}
then $\boldsymbol{\tilde{A}}_{0}(i, i-1)$ is negative semi-definite. Let $\lambda_i, i = \{1, \dots, n\}$ be the eigenvalues of $\boldsymbol{\tilde{A}}_{0}$ with $\lambda_n < \lambda_{n-1} < \dots < \lambda_1 = 0$, $\boldsymbol U$ be the matrix whose columns are the corresponding eigenvectors and $\boldsymbol \Lambda$ be the diagonal matrix. Then
\small
\begin{equation*}
\begin{split}
\lim_{T \rightarrow +\infty} \frac{1}{T} \boldsymbol{v}_{0}(0, T) & =  \lim_{T \rightarrow +\infty} \frac{1}{T} e^{T \boldsymbol{\tilde{A}}_0(\kappa)} \boldsymbol{G} + \lim_{T \rightarrow +\infty} \frac{1}{T} \int_{0}^T e^{t\boldsymbol{\tilde{A}}_0(\kappa)} \boldsymbol{\tilde b}( \kappa) \, \mathrm{d} t  \\ 
& =  \lim_{T \rightarrow +\infty} \frac{1}{T} \boldsymbol{U} e^{T \boldsymbol{\Lambda}} \boldsymbol{U}^{-1} \boldsymbol{G} + \lim_{T \rightarrow +\infty} \frac{1}{T} \boldsymbol{U} \int_{0}^T  e^{t\boldsymbol{\Lambda}} \, \mathrm{d} t \, \boldsymbol{U}^{-1} \boldsymbol{\tilde b} \\
& = \lim_{T \rightarrow +\infty} \frac{1}{T} \boldsymbol{U} \begin{bmatrix}
e^{\lambda_1 T} &  &  &  & \\ 
& e^{\lambda_2 T} &  & & &  \\ 
& &  &  \dots & \\
& &  & & & e^{\lambda_n T} 
\end{bmatrix} \boldsymbol{U}^{-1} \boldsymbol{G}  \\ & \quad + \lim_{T \rightarrow +\infty} \frac{1}{T} \boldsymbol{U} \begin{bmatrix}
\int_0^T e^{\lambda_1 t} \, \mathrm{d} t&  &  &  & \\ 
& \int_0^T e^{\lambda_2 t} \, \mathrm{d} t &  & & &  \\ 
& &  &  \dots & \\
& &  & & & \int _0^T e^{\lambda_n t} \, \mathrm{d} t
\end{bmatrix} \boldsymbol{U}^{-1} \boldsymbol{\tilde b} \\ 
& = \lim_{T \rightarrow +\infty} \frac{1}{T} \boldsymbol{U} \begin{bmatrix}
T &  &  &  & \\ 
& \frac{1}{\lambda_2}(e^{\lambda_2 T} - 1) &  & & &  \\ 
& &  &  \dots & \\
& &  & & & \frac{1}{\lambda_n} (e^{\lambda_n T} - 1)
\end{bmatrix} \boldsymbol{U}^{-1} \boldsymbol{\tilde b} \\ 
& = \boldsymbol{U} \boldsymbol{W} \boldsymbol{U}^{-1} \boldsymbol{\tilde b} \, , 
\end{split}
\end{equation*}
\normalsize
where $\boldsymbol{W}$ is the $n-$square matrix with only $1$ on the first diagonal element and $0$ otherwise. By~\eqref{definition gamma(kappa;kappastar)}, we know $\gamma(\kappa; \kappa^\ast) \mathds{1} = \boldsymbol{U} \boldsymbol{W} \boldsymbol{U}^{-1} \boldsymbol{\tilde b}$ with $\mathds{1}$ the $n-$dim vector with all entries $1$.
\end{proof}

\subsection{Proof of Lemma~\ref{second derivative of gamma is bounded}} \label{proof lemma second derivative of gamma is bounded}
\begin{proof}
First, we would like to prove that $\kappa \mapsto \psi^{\kappa}$ with $\psi^\kappa$ given by~\eqref{feedback control} is $C^{\infty}([\underline K, \bar K])$. By Proposition~\ref{Existence and Uniqueness for Ergodic Optimal Control} and Theorem~\ref{solve ergodic control problem}, it is equivalent to show that $\kappa \mapsto \boldsymbol{\hat \omega}(\kappa)$ by~\eqref{solve hat omega} is $C^{\infty}([\underline K, \bar K])$. Let us consider a matrix $\boldsymbol{D}$ as 
\begin{equation} \label{matrix D}
\boldsymbol{D} = Diag(\, d_{\bar q}, d_{\bar q - 1}, \dots, d_{\underline q} \,), 
\end{equation}
with $d_q = \prod_{\bar q - 1, \dots, q} \sqrt{\frac{\lambda^{-}}{\lambda^{+}}} $ for $q \in \{\bar q-1, \bar q-2, \dots, \underline q \}$ and $d_{\bar q} = 1$. Then $\boldsymbol{C}$ given in Theorem~\ref{solve ergodic control problem} can be transformed into a real and symmetric tridiagonal matrix $\boldsymbol{\tilde{C}}$ by $\boldsymbol{\tilde{C}} = \boldsymbol{D}^{-1} \boldsymbol{C} \boldsymbol{D}$ with entries 
\begin{align} \label{matrix tilde C}
\boldsymbol{\tilde{C}}(i, q) = \begin{cases}
- \kappa \phi q^2 - \kappa \gamma(\kappa), & \text{if $i = q$},\\
\sqrt{\lambda^{+} \lambda^{-}} e^{-1}, & \text{if $i = q-1$ or $q+1$}, \\
0, & \text{otherwise.}
\end{cases}
\end{align}

As there exists an $\boldsymbol{\hat \omega}$ solves~\eqref{solve hat omega}, there must be an eigenvalue $\lambda_1 = 0$ of $\boldsymbol{C}$, or $\boldsymbol{\tilde{C}}$ as they are similar, with the corresponding eigenvector $\boldsymbol{\hat \omega}$ of $\boldsymbol{C}$ and $\boldsymbol{\tilde \omega}$ of $\boldsymbol{\tilde{C}}$. Therefore, 
\begin{equation}
0 = \lambda_1 \boldsymbol{\tilde \omega} = \boldsymbol{\tilde{C}} \boldsymbol{\tilde \omega} = \boldsymbol{D}^{-1} \boldsymbol{C} \boldsymbol{D} \boldsymbol{\tilde \omega}.
\end{equation}
As $\boldsymbol{D}$ is non-singular, we have $\boldsymbol{\hat \omega} = \boldsymbol{D} \boldsymbol{\tilde \omega}$. 
As shown in Appendix~\ref{continuity ergodic constant}, $\gamma(\kappa)$ is $C^{\infty}([\underline K, \bar K])$, therefore $\kappa \mapsto \boldsymbol{\tilde{C}}(\kappa)$ given by~\eqref{matrix tilde C} is $C^{\infty}([\underline K, \bar K])$. 
Hence, by \cite[Result 7.2; Theorem 7.6]{alekseevsky1998choosing}, the eigenvector $\boldsymbol{\tilde \omega}$ can be parameterised smoothly on $\kappa \in [\underline K, \bar K]$. Obviously, $\boldsymbol{D}$ is independent of $\kappa$. Therefore, we can conclude that $\boldsymbol{\hat \omega}(\kappa) = \boldsymbol{D} \boldsymbol{\tilde \omega}(\kappa)$ is $C^{\infty}([\underline K, \bar K])$.

Proposition~\ref{analytical solution for gamma(kappastar)} gives an analytical solution to $\gamma(\kappa;\kappa^\ast)$ but it is not enough to show $\kappa \mapsto \gamma(\kappa;\kappa^\ast)$ is twice differentiable, as $\boldsymbol{U}$'s columns are the eigenvectors of $\boldsymbol{\tilde{A}}_0$ whereas $\boldsymbol{\tilde{A}}_0$ is not a self-adjoint or normal operator. Let us consider the matrix $\boldsymbol{V}$ as 
\begin{equation} \label{matrix V}
\boldsymbol{V} = Diag(\, v_{\bar q}, v_{\bar q - 1}, \dots, v_{\underline q} \,), 
\end{equation}
with $v_q = \frac{\prod_{i=\bar q - 1}^{q} \sqrt{\lambda^{-}} e^{- \frac{1}{2}\kappa^\ast \psi^{\kappa, -}(i)}}{\prod_{i=\bar q}^{q+1} \sqrt{\lambda^{+}} e^{- \frac{1}{2}\kappa^\ast \psi^{\kappa, +}(i)} }$ for $q \in \{\bar q-1, \bar q-2, \dots, \underline q \}$ and $v_{\bar q} = 1$. Then $\boldsymbol{J} = \boldsymbol{V}^{-1} \boldsymbol{\tilde{A}}_0 \boldsymbol{V}$ is a real and symmetric tridiagonal matrix with the entries as 
\begin{align} \label{matrix J}
\boldsymbol{J}(i,q) = \begin{cases}
- (\lambda^{+} e^{- \kappa^\ast \psi^{\kappa, +}(q)} \mathds{1}_{q > \underline q} + \lambda^{-} e^{- \kappa^\ast \psi^{\kappa, -}(q)}\mathds{1}_{q < \bar q} ), & \text{if $i = q$}, \\
\sqrt{\lambda^{+} \lambda^{-}} e^{- \frac{\kappa^\ast}{\kappa}}, & \text{if $i = q+1$}, \\
\sqrt{\lambda^{+} \lambda^{-}} e^{- \frac{\kappa^\ast}{\kappa}}, & \text{if $i = q-1$}, \\
0, & \text{otherwise. }
\end{cases}
\end{align}
There exists an orthogonal matrix $\boldsymbol{U'}$ whose columns are the eigenvectors of $\boldsymbol{J}$ such that $\boldsymbol{\Lambda} = \boldsymbol{U'}^{-1} \boldsymbol{J} \boldsymbol{U'}$, where $\boldsymbol{\Lambda}$ is the diagonal matrix with eigenvalues of $\boldsymbol{\tilde{A}}_0$ since $\boldsymbol{J}$ is similar to $\boldsymbol{\tilde{A}}_0$. Moreover, we have
\begin{equation*}
\boldsymbol{\Lambda} = \boldsymbol{U'}^{-1} \boldsymbol{J} \boldsymbol{U'} = \boldsymbol{U'}^{-1} \boldsymbol{V}^{-1} \boldsymbol{\tilde{A}}_0 \boldsymbol{V} \boldsymbol{U'} =\boldsymbol{U}^{-1} \boldsymbol{\tilde{A}}_0 \boldsymbol{U}, 
\end{equation*}
hence $\boldsymbol{U} = \boldsymbol{V} \boldsymbol{U'}$. Therefore, 
\begin{equation*}
\gamma(\kappa; \kappa^\ast) \mathds{1} = \boldsymbol{V} \boldsymbol{U'} \boldsymbol{W} \boldsymbol{U'}^{-1} \boldsymbol{V}^{-1} \boldsymbol{\tilde b}. 
\end{equation*}
Let $\boldsymbol{U'} = [\boldsymbol{u}_1, \boldsymbol{u}_2, \dots, \boldsymbol{u}_n]$ where $\boldsymbol{u}_i$ is the corresponding eigenvector of $\boldsymbol{J}$ with the eigenvalue $\lambda_i$. Then
\begin{equation*}
\boldsymbol{U'} \boldsymbol{W} \boldsymbol{U'}^{-1} = Diag(\boldsymbol{u}_1) \boldsymbol{W'} Diag(\boldsymbol{u}_1 ) ,
\end{equation*}
where $\boldsymbol{W'}$ is the $n-$square matrix with all entries $1$ and $\boldsymbol{u}_1$ is the eigenvector of $\boldsymbol{J}$ with $\lambda_1 = 0$. 
\begin{equation} \label{closed-form gamma kappastar}
\gamma(\kappa; \kappa^\ast) \mathds{1} = \boldsymbol{V} Diag(\boldsymbol{u}_1) \boldsymbol{W'} Diag(\boldsymbol{u}_1 )  \boldsymbol{V}^{-1} \boldsymbol{\tilde b}. 
\end{equation}
Clearly, $\psi^{\kappa} \mapsto e^{-\kappa^\ast \psi^{\kappa}}$ is a smooth function and $\kappa \mapsto \psi^{\kappa}$ is $C^{\infty}([\underline K, \bar K])$. Therefore, $\boldsymbol{V}(\kappa)$~\eqref{matrix V}, $\boldsymbol{V}^{-1}(\kappa)$, $\boldsymbol{\tilde b}(\kappa)$~\eqref{vector tilde b} and $\boldsymbol{J}(\kappa)$ are $C^{\infty}([\underline K, \bar K])$. 
Moreover, $\boldsymbol{J}$ is a real and symmetric tridiagonal matrix with the simple eigenvalue $\lambda_1$. By \cite[Result 7.2; Theorem 7.6]{alekseevsky1998choosing}, we know that $\kappa \mapsto \boldsymbol{u}_1(\kappa)$ can be chosen to be parameterised smoothly in $\kappa$.  Hence, by~\eqref{closed-form gamma kappastar}, $\kappa \mapsto \gamma(\kappa; \kappa^\ast)$ is at least $C^{2}([\underline K, \bar K])$ and $\frac{\mathrm{d}^2}{\mathrm{d} \kappa^2} \gamma(\kappa; \kappa^\ast)$ is bounded on the compact set $\kappa \in [\underline K, \bar K]$.  
\end{proof}

\subsection{Proof of Lemma \ref{ergodic control stability}} \label{proof lemma ergodic control stability}
\begin{proof}
The state process $\big( Q_t^{\psi^{\kappa} ; \kappa^\ast} \big)_{t \geq 0}$~\eqref{inventory under kappastar} is driven by two independent Poisson jump processes with intensities $\lambda^+ e^{- \kappa^\ast \delta^+}$ and $\lambda^- e^{- \kappa^\ast \delta^-}$, respectively. 
The depths $\delta^\pm$ are uniquely and continuously determined by the ergodic optimal control $q \mapsto \psi^\kappa(q)$ given the agent's current position $q = Q_t$ at time $t$, where $\psi^\kappa(q)$ given by~\eqref{feedback control} is the ergodic optimal control under the misspecified parameter $\kappa$. 
As a result, the transition probability at any time $t > 0$ depends only on the current state, implying that the stochastic process $(Q_t)_{t \geq 0}$--where we omit the superscript for notational simplicity--satisfies the Markov property. Furthermore, the state space $\Omega^{Q} = [\underline q, \bar q] \cap \mathbb Z$ is discrete and finite. Hence $(Q_t)_{t \geq 0}$ can be equivalently represented as a continuous-time Markov chain with a finite state space. 

Let $\boldsymbol Q = (\boldsymbol{Q}_{ij})_{i,j \in \Omega^{Q}}$ denote the transition rate matrix, where the indices are labelled from $\bar q$ to $\underline q$. 
Each entry $\boldsymbol{Q}_{ij}$ represents the instantaneous transition rate of the process from state $i$ to state $j$, which can be derived from the infinitesimal generator of $(Q_t)_{t\geq0}$. 
Hence the entries of $\boldsymbol{Q}$ are
\begin{align} \label{transition matrix P} 
\boldsymbol Q_{ij} = \begin{cases}
- \big( \lambda^{+} e^{- \kappa^\ast \psi^{\kappa, +}(i)} \mathds{1}_{i > \underline q} + \lambda^{-} e^{- \kappa^\ast \psi^{\kappa, -}(i)} \mathds{1}_{i < \bar q} \big) , & \text{if $i = j$},\\
\lambda^{-} e^{- \kappa^\ast \psi^{\kappa, -}(i)}, & \text{if $i = j-1$}, \\
\lambda^{+} e^{- \kappa^\ast \psi^{\kappa, +}(i)}, & \text{if $i = j+1$}, \\
0, & \text{otherwise}.
\end{cases}
\end{align}
From~\eqref{transition matrix P}, one may notice that $\big( Q_t^{\psi^{\kappa} ; \kappa^\ast} \big)_{t \geq 0}$ is equivalent to a general Birth-Death process with a finite state space. Hence there exists a unique equilibrium distribution $\pi$, for $\big( Q_t^{\psi^{\kappa} ; \kappa^\ast} \big)_{t \geq 0}$ when $t$ goes to infinity \cite[Theorem 5.5.3]{resnick1992adventures}. Moreover, $\pi$ is uniquely determined by 
\begin{equation*}
\pi \boldsymbol{Q} = 0, \quad \text{subject to } \sum_{q \in \Omega^Q} \pi_q = 1\, .
\end{equation*}
\end{proof}

\subsection{Proof of Proposition~\ref{expectation under equi is 0}} \label{proof of Proposition expectation under equi is 0}
\begin{proof}
By Proposition \ref{linear PDE for gamma} and by~\eqref{L function}, the equation \eqref{linear system for gamma(kappa;kappastar)} can be expressed as
\small
\begin{equation*}
\begin{split}
\gamma(\kappa; \kappa^\ast) = - \phi q^2  +  & \lambda^{+} e^{-\kappa^\ast \psi^{\kappa, +}(q)} \Big( \hat{v}^{\psi^\kappa}(q-1; \kappa^\ast) - \hat{v}^{\psi^\kappa}(q; \kappa^\ast) + \psi^{\kappa, +}(q)\Big)  \mathds{1}_{q > \underline q}\\
+ & \lambda^{-} e^{-\kappa^\ast \psi^{\kappa, -}(q)} \Big( \hat{v}^{\psi^\kappa}(q+1; \kappa^\ast) - \hat{v}^{\psi^\kappa}(q; \kappa^\ast) + \psi^{\kappa, -}(q)\Big) \mathds{1}_{q < \bar q} \, . 
\end{split}
\end{equation*}
\normalsize
Take integral from  $0$ to $T$ and then take the expectation with respect to the probability measure $\pi^{\psi^\kappa; \kappa^\ast}$ under the controlled SDE~\eqref{inventory under kappastar} with $Q_0 \sim \pi^{\psi^\kappa; \kappa^\ast}$, we have
\small
\begin{equation*}
\begin{aligned}
& \int_{\Omega^{Q}} \int_0^T \gamma(\kappa; \kappa^\ast) \, \mathrm{d} t \, \mathrm{d} \pi^{\psi^\kappa; \kappa^\ast} = \int_{\Omega^{Q}} \int_0^T \Big( \lambda^{+} \psi^{\kappa, +}(q) e^{-\kappa^\ast \psi^{\kappa, +}(q)} + \lambda^{-} \psi^{\kappa, -}(q) e^{-\kappa^\ast \psi^{\kappa, -}(q)} \\
& - \phi q^2 + \lambda^{+} e^{-\kappa^\ast \psi^{\kappa, +}} \big( \hat{v}^{\psi^\kappa}(q-1; \kappa^\ast) - \hat{v}^{\psi^\kappa}(q; \kappa^\ast) \big) \mathds{1}_{q > \underline q} \\
& + \lambda^{-} e^{-\kappa^\ast \psi^{\kappa, -}} \big( \hat{v}^{\psi^\kappa}(q+1; \kappa^\ast) - \hat{v}^{\psi^\kappa}(q; \kappa^\ast) \big) \mathds{1}_{q < \bar q} \Big) \, \mathrm{d} t \, \mathrm{d} \pi^{\psi^\kappa; \kappa^\ast} \, ,
\end{aligned}
\end{equation*}
\normalsize
where we omit the indicator function in the first line since $\psi^{\kappa, \pm}(q)  e^{-\kappa^\ast \psi^{\kappa, \pm}(q)} = 0$ when $q = \bar q, \underline q$, respectively. 
As $\gamma(\kappa; \kappa^\ast)$ is independent of $q$ and $t$ by Proposition \ref{linear PDE for gamma}, dividing by $T > 0$ we get
\small
\begin{equation*}
\begin{split}
& \gamma  (\kappa; \kappa^\ast) \\
&  =  \frac1T  \mathbb{E}_{\pi^{\psi^\kappa; \kappa^\ast}}  \Big[ \int_0^T   \lambda^{+} \psi^{\kappa, +} e^{-\kappa^\ast \psi^{\kappa, +}(Q_t^{\psi^{\kappa}; \kappa^\ast})}  + \lambda^{-} \psi^{\kappa, -} e^{-\kappa^\ast \psi^{\kappa, -}(Q_t^{\psi^{\kappa}; \kappa^\ast})} - \phi (Q_t^{\psi^{\kappa}; \kappa^\ast})^2 \, \mathrm{d} t \Big] \\
& \qquad + \frac1T\int_{\Omega^Q} \int_0^T \Big( \lambda^{+} e^{-\kappa^\ast \psi^{\kappa, +}(q)} \big( \hat{v}^{\psi^\kappa}(q-1; \kappa^\ast) - \hat{v}^{\psi^\kappa}(q; \kappa^\ast ) \big)  \mathds{1}_{q > \underline q}  \\
& \qquad \qquad + \lambda^{-} e^{-\kappa^\ast \psi^{\kappa, -}(q)} \big( \hat{v}^{\psi^\kappa}(q+1; \kappa^\ast) - \hat{v}^{\psi^\kappa}(q; \kappa^\ast) \big) \mathds{1}_{q < \bar q} \Big) \, \mathrm{d} t \, \mathrm{d} \pi^{\psi^\kappa; \kappa^\ast}. 
\end{split}
\end{equation*}
\normalsize
Moreover, for the initial distribution $\pi^{\psi^\kappa; \kappa^\ast}$, we have 
\begin{equation*}
\begin{split}
& \gamma(\kappa; \kappa^\ast)  = \lim_{T \rightarrow + \infty}\frac{1}{T} v_{0}^{\psi^{\kappa}}(0, q; T; \kappa^\ast) 
\\
& = \lim_{T \rightarrow + \infty} \frac{1}{T} \mathbb{E}_{\pi^{\psi^\kappa; \kappa^\ast}} \Big[  
\int_0^T \Big(\lambda^{+} \psi^{\kappa, +} e^{- \kappa^{\ast} \psi^{\kappa, +}} +  \lambda^{-} \psi^{\kappa, -} e^{- \kappa^{\ast} \psi^{\kappa, -}} - \phi (Q^{\psi^{\kappa}; \kappa^\ast}_t)^2 \Big) \, \mathrm{d} t \Big] 
\end{split}
\end{equation*}
Hence
\begin{equation*}
\begin{split}
0 = \int_{\Omega^{Q}}  \lambda^{+} e^{-\kappa^{\ast} \psi^{\kappa, +}} \Big( & \hat{v}^{\psi^\kappa}(q-1; \kappa^\ast) - \hat{v}^{\psi^\kappa}(q; \kappa^\ast) \Big)  \mathds{1}_{q > \underline q}  \\ 
& + \lambda^{-} e^{-\kappa^\ast \psi^{\kappa, -}} \Big( \hat{v}^{\psi^\kappa}(q+1; \kappa^\ast) - \hat{v}^{\psi^\kappa}(q; \kappa^\ast) \Big) \mathds{1}_{q < \bar q}  \, \mathrm{d} \pi^{\psi^\kappa; \kappa^\ast},
\end{split}
\end{equation*}
which concludes the proof.
\end{proof}

\subsection{Proof of Theorem~\ref{performance gap}} \label{proof theorem performance gap}
\begin{proof}
By~\eqref{linear system for gamma(kappa;kappastar)} in Proposition \ref{linear PDE for gamma} and~\eqref{L function}, we have, 
\begin{equation*}
\begin{split}
\lambda^{+}  \psi^{\kappa, +} (q) e^{-\kappa^\ast \psi^{\kappa, +}(q)} + \lambda^{-} \psi^{\kappa, -}(q)  & e^{-\kappa^\ast \psi^{\kappa, -} (q)}  - \phi q^2 = \\ 
\gamma(\kappa; \kappa^\ast) & + \lambda^{+} e^{-\kappa^\ast \psi^{\kappa, +}} \big( \hat{v}^{\psi^\kappa}(q-1; \kappa^\ast) - \hat{v}^{\psi^\kappa}(q; \kappa^\ast) \big)  \mathds{1}_{q > \underline q} \\
& + \lambda^{-} e^{-\kappa^\ast \psi^{\kappa, -}} \big( \hat{v}^{\psi^\kappa}(q+1; \kappa^\ast) - \hat{v}^{\psi^\kappa}(q; \kappa^\ast) \big) \mathds{1}_{q < \bar q},
\end{split}
\end{equation*}
where we ignore the indicator functions in the first line since $\psi^{\kappa, \pm}(q)  e^{-\kappa^\ast \psi^{\kappa, \pm}(q)} = 0$ for $q = \bar q, \underline q$, respectively. 
Also, by definition, we know that
\begin{equation*}
f(t, q, \delta^{\pm}; \kappa^\ast) = \lambda^{+} \delta^{+} e^{- \kappa^\ast \delta^{+}} + \lambda^{-} \delta^{-} e^{- \kappa^\ast \delta^{-}} - \phi q^2.
\end{equation*}
Therefore, by Definition~\ref{regret def} of regret,
\small
\begin{equation} \label{regret equation}
\begin{split}
& \mathcal{R}^{\Psi}(T)  = \gamma(\kappa^\ast) T - \mathbb{E}_q \Big[ \int_0^T f(t, Q_t^{\psi^{\kappa_t}; \kappa^\ast}, \psi^{\kappa_t}; \kappa^\ast) \, \mathrm{d} t \Big] \\
& = \mathbb{E} \Big[\int_0^T \big( \gamma(\kappa^\ast; \kappa^\ast) - \gamma(\kappa_t; \kappa^\ast) \big) \, \mathrm{d} t \Big] \\
& \,\, \, \, - \mathbb{E}_q \Big[ \int_0^T  \lambda^{+} e^{-\kappa^\ast \psi^{\kappa_t, +}(Q_t^{\psi^{\kappa_t}; \kappa^\ast})} \Big( \hat{v}^{\psi^{\kappa_t}}(Q_t^{\psi^{\kappa_t}; \kappa^\ast}-1; \kappa^\ast) - \hat{v}^{\psi^{\kappa_t}}(Q_t^{\psi^{\kappa_t}; \kappa^\ast}; \kappa^\ast) \Big)  \mathds{1}_{Q_t^{\psi^{\kappa_t}; \kappa^\ast} > \underline q} \\
& \quad \quad + \lambda^{-} e^{-\kappa^\ast \psi^{\kappa_t, -}(Q_t^{\psi^{\kappa_t}; \kappa^\ast})} \Big( \hat{v}^{\psi^{\kappa_t}}(Q_t^{\psi^{\kappa_t}; \kappa^\ast}+1; \kappa^\ast) - \hat{v}^{\psi^{\kappa_t}}(Q_t^{\psi^{\kappa_t}; \kappa^\ast}; \kappa^\ast) \Big) \mathds{1}_{ Q_t^{\psi^{\kappa_t}; \kappa^\ast} < \bar q} \, \mathrm{d} t \Big] \, . \\
\end{split}
\end{equation}
\normalsize
Let us define
\begin{equation} \label{h function}
\begin{split}
h(\kappa_t, q) = \lambda^{+} e^{-\kappa^\ast \psi^{\kappa_t, +}(q)} & \big( \hat{v}^{\psi^{\kappa_t}}(q-1; \kappa^\ast)  - \hat{v}^{\psi^{\kappa_t}}(q; \kappa^\ast) \big) \mathds{1}_{q > \underline q} \\ & + \lambda^{-} e^{-\kappa^\ast \psi^{\kappa_t, -}(q)} \big(\hat{v}^{\psi^{\kappa_t}}(q+1; \kappa^\ast) - \hat{v}^{\psi^{\kappa_t}}(q; \kappa^\ast) \big) \mathds{1}_{q < \bar q}.
\end{split}
\end{equation}
Then
\begin{equation*}
\begin{split}
\mathcal{R}^{\Psi}(T) &  = \mathbb E \Big[ \int_0^T \big( \gamma(\kappa^\ast; \kappa^\ast) - \gamma(\kappa_t; \kappa^\ast) \big) \, \mathrm{d} t \Big]  - \mathbb{E}_q \Big[ \int_0^T h(\kappa_t, Q_t^{\psi^{\kappa_t}, \kappa^\ast}) \, \mathrm{d} t  \Big] \\
& \leq C \mathbb E \Big[ \int_0^T \left| \kappa_t - \kappa^\ast \right|^2 \, \mathrm{d} t \Big] - \int_0^T \int_{\Omega} h(\kappa_t, q) \, d \pi^{\psi^{\kappa_t}; \kappa^\ast}_t \, \mathrm{d} t \,,
\end{split}
\end{equation*}
where the last inequality comes from Corollary~\ref{gamma bounded by difference of kappa}. Moreover, $\pi^{\psi^{\kappa_t}; \kappa^\ast}_t$ is the probability measure evolves under control $\psi^{\kappa_t}$ and parameter $\kappa^\ast$ with $Q_0 \sim q$. By Lemma~\ref{bounded for rv_r under model mispecification} (2), there exists a constant $\bar h > 0$ such that $ \left| h(\kappa_t, q) \right| \leq \bar h$ for
$\kappa_t \in [\underline K, \bar K]$. Therefore, 
\small
\begin{equation*}
\begin{split}
\mathcal{R}^{\Psi}(T) & \leq C \mathbb E \Big[ \int_0^T \left| \kappa_t - \kappa^\ast \right|^2 \, \mathrm{d} t \Big] + \left| \int_0^T \int_{\Omega} h(\kappa_t, q) \, \mathrm{d} \pi^{\psi^{\kappa_t}; \kappa^\ast}_t \, \mathrm{d} t \right| \\ 
& \leq C \mathbb E \Big[ \int_0^T \left| \kappa_t - \kappa^\ast \right|^2 \, \mathrm{d} t \Big] + \left| \bar h \int_0^T \int_{\Omega} \frac{h(\kappa_t, q)}{\bar h} (\, \mathrm{d} \pi^{\psi^{\kappa_t}; \kappa^\ast}_t - \mathrm{d} \pi^{\psi^{\kappa_t}; \kappa^\ast} )\, \mathrm{d} t \right| \\ 
& \quad + \left| \int_0^T \int_{\Omega} h(\kappa_t, q) \, \mathrm{d} \pi^{\psi^{\kappa_t}; \kappa^\ast} \, \mathrm{d} t \right| \\
& \leq C \mathbb E \Big[ \int_0^T \left| \kappa_t - \kappa^\ast \right|^2 \, \mathrm{d} t \Big] + \left| \bar h \int_0^T \int_{\Omega} \frac{h(\kappa_t, q)}{\bar h} (\, \mathrm{d} \pi^{\psi^{\kappa_t}; \kappa^\ast}_t - \mathrm{d} \pi^{\psi^{\kappa_t}; \kappa^\ast} )\, \mathrm{d} t \right|, 
\end{split}
\end{equation*}
\normalsize
where $\pi^{\psi^{\kappa_t}; \kappa^\ast}$ is the probability measure of equilibrium distribution under control $\psi^{\kappa_t}$ and parameter $\kappa^\ast$, and the last step uses Proposition \ref{expectation under equi is 0}. Then, 
\begin{equation*}
\begin{split}   
\mathcal{R}^{\Psi}(T) & \leq C \mathbb E \Big[ \int_0^T \left| \kappa_t - \kappa^\ast \right|^2 \, \mathrm{d} t \Big] + \bar h \int_0^T \int_{\Omega} \left| \mathrm{d} \pi^{\psi^{\kappa_t}; \kappa^\ast}_t  - \mathrm{d} \pi^{\psi^{\kappa_t}; \kappa^\ast} \right| \, \mathrm{d} t \\
& = C \mathbb E \Big[ \int_0^T \left| \kappa_t - \kappa^\ast \right|^2 \, \mathrm{d} t \Big] + 2 \bar h \int_0^T \left\| \pi^{\psi^{\kappa_t}; \kappa^\ast}_t - \pi^{\psi^{\kappa_t}; \kappa^\ast} \right\|_{TV} \, \mathrm{d} t ,\\
\end{split}
\end{equation*}
where $\| \cdot\|_{TV}$ denotes the total variation.
By Lemma \ref{convergence theorem}, we have
\begin{equation*}
\begin{split}   
\mathcal{R}^{\Psi}(T)
& \leq C \mathbb E \Big[ \int_0^T \left| \kappa_t - \kappa^\ast \right|^2 \, \mathrm{d} t \Big] + 2 \bar h \int_0^T C(\kappa_t) \alpha(\kappa_t)^t \, \mathrm{d} t  \\
& \leq C \mathbb E \Big[ \int_0^T \left| \kappa_t - \kappa^\ast \right|^2 \, \mathrm{d} t \Big] + 2 \bar h \bar C \int_0^T \bar{\alpha}^{t} \, \mathrm{d} t \\ 
& = C \mathbb E \Big[ \int_0^T \left| \kappa_t - \kappa^\ast \right|^2 \, \mathrm{d} t \Big] + 2 \bar h \bar C \frac{ 1 - e^{T \ln \bar \alpha}  }{\ln (\bar \alpha^{-1})} \\
& \leq C \mathbb E \Big[ \int_0^T \left| \kappa_t - \kappa^\ast \right|^2 \, \mathrm{d} t \Big] + \frac{2 \bar h \bar C }{\ln (\alpha^{-1})}\,,
\end{split}
\end{equation*}
where $\bar C = \sup_{\kappa_t} C(\kappa_t) > 0$ and $0 < \bar{\alpha} = \sup_{\kappa_t} \alpha(\kappa_t) < 1$ for $\kappa_t \in [\underline K, \bar K]$. The existence of $\bar C$ and $\bar{\alpha}$ can be concluded by the following discussion. 
As shown in Appendix~\ref{proof lemma ergodic control stability}, the state dynamics can be represented by a continuous-time Markov chain. 
Therefore, the convergence rate $\alpha$ is bounded by the exponential of the second largest eigenvalue of the transition rate matrix $\boldsymbol{Q}$ by the Kolmogorov forward equation~\cite[Chapter 6.6]{anderson2012continuous}. 
Since the transition rate matrix $\boldsymbol{Q}$~\eqref{transition matrix P} is tridiagonal, its eigenvalues are simple, i.e. the multiplicity is 1, implying that the eigenvalues are continuous with respect to $\kappa$. 
Therefore, there exists $\bar \alpha = \sup_{\kappa \in [\underline K, \bar K]} \alpha(\kappa)$. Moreover, $C$ can be given by the norm of the eigenvectors for $\boldsymbol{Q}$. 
By \cite{acker1974absolute}, if all the eigenvalues of the $\boldsymbol{Q}$ are simple, the corresponding eigenvectors can be chosen absolutely continuous on $\kappa \in [\underline K, \bar K]$, hence the existence of $\bar C$. 
Let $C_1 = C > 0, C_2 = 2 \bar h \bar C > 0$ and $0 < \alpha = \bar \alpha < 1$, we conclude the result.  
\end{proof}

\subsection{Proof of Concentration Inequality} \label{proof of concentration inequality}
\subsubsection{Proof of Proposition~\ref{prop: concentration rescaling sub-Gaussian}} \label{proof prop  concentration rescaling sub-Gaussian}
\begin{proof}
Let $Z_n := f(\delta_n) (Y_n - e^{-\kappa^\ast \delta_n})$. 
Since $-\|f\|_\infty \leq Z_n \leq \|f\|_\infty$ and since $\mathbb{E}\big[Z_n\big| (\delta_n)_{n=1}^N\big] = 0$, we get that
$$\mathbb{E} \left[ \exp \left( \lambda Z_n \right)  \big| (\delta_n)_{n=1}^N \right] \leq \exp \left(\frac{1}{2} \lambda^2 \|f\|_\infty^2 \right).$$
Therefore, by the Markov inequality, for any $h > 0$,
\begin{align*}
\mathbb{P}^\ast \left( \sum_{n=1}^N Z_n > h \bigg| (\delta_n)_{n=1}^N \right) \leq \inf_{\lambda \in \mathbb{R} } \exp \left(\frac{N}{2} \lambda^2 \|f\|_\infty^2 -  \lambda h \right) = \exp \left( - \frac{h^2}{2 N \|f\|_\infty^2 }\right).
\end{align*}
In particular,
\begin{align*}
\mathbb{P}^\ast \left( \sum_{n=1}^N Z_n > \|f\|_\infty \sqrt{2N \ln \big( \tfrac{2}{\varepsilon} \big)}  \bigg| (\delta_n)_{n=1}^N \right) \leq \frac{\varepsilon}{2}.
\end{align*}
By applying the same argument to $(-Z_n)_{n=1}^N$, we obtain the reverse inequality and prove the required result conditional on $(\delta_n)_{n=1}^N$. Taking the tower property, we achieve the required claim. 
\end{proof}

\subsubsection{Proof of Proposition~\ref{prop: lower bound Fisher}}
\label{proof prop lower bound Fisher}
\begin{proof}

By \eqref{second derivative of the log-likelihood function}, we have 
\small
\begin{align*}
\frac{\mathrm{d}^2}{\mathrm{d} \kappa^2} \Tilde{\ell}_N (\kappa)  = - & \left( \sum_{n=1}^N  (1-Y_n) \delta_n^2 \frac{e^{- \kappa \delta_n}}{(1 - e^{- \kappa \delta_n})^2}
+ \delta_0^2 \frac{e^{- \kappa \delta_0}}{(1 - e^{- \kappa \delta_0})^2}\right)  \mathds{1}_{\kappa \leq \bar{K}}  \\
& \qquad - \left( \sum_{n=1}^N  (1-Y_n) \delta_n^2 \frac{e^{-\bar{K} \delta_n}}{(1 - e^{- \bar{K} \delta_n})^2} +  \delta_0^2  \frac{e^{-\bar{K} \delta_0}}{(1 - e^{- \bar{K} \delta_0})^2} \right) \mathds{1}_{\kappa > \bar{K}}.
\end{align*}
\normalsize
By observing that $x \mapsto \frac{e^{-x}}{(1-e^{-x})^2}$ is decreasing for any $x \geq 0$ and $(1-Y_n) \geq 0$ for all $n \in \mathbb N$,
\begin{align*}
-\frac{\mathrm{d}^2}{\mathrm{d} \kappa^2} \tilde{\ell}_N(\kappa) 
&\geq  \sum_{n=1}^N   (1-Y_n) \delta_n^2 \left( \frac{e^{- \kappa \delta_n}}{(1 - e^{- \kappa \delta_n})^2} \mathds{1}_{\kappa \leq \bar{K}} + \frac{e^{-\bar{K} \delta_n}}{(1 - e^{- \bar{K} \delta_n})^2}  \mathds{1}_{\kappa > \bar{K}}  \right) \\
& \geq \sum_{n=1}^N (1-Y_n) \delta_n^2 \left( \frac{e^{-\bar{K} \delta_n}}{(1 - e^{- \bar{K} \delta_n})^2}  \right) \, .
\end{align*}
Hence, for any policy $(\delta_n)_{n=1}^\infty$ taking values in $[\underline{\delta}, \bar{\delta}] \cup \{+\infty\}$, it holds that for any $\kappa > 0$,
\begin{align*}
-\frac{\mathrm{d}^2}{\mathrm{d} \kappa^2} \tilde{\ell}_N(\kappa)  \geq \sum_{n=1}^N (1-Y_n) \left(    \frac{\underline{\delta}^2e^{-\bar{K} \bar{\delta}}}{(1 - e^{- \bar{K} \bar{\delta}})^2}  \right).
\end{align*}
By Proposition \ref{prop: concentration rescaling sub-Gaussian}, it holds that with $\mathbb P^\ast$-probability at least $1-\varepsilon$
$$\left| \sum_{n=1}^N (Y_n - e^{-\kappa^\ast \delta_n}) \left(    \frac{\underline{\delta}^2e^{-\bar{K} \bar{\delta}}}{(1 - e^{- \bar{K} \bar{\delta}})^2}  \right) \right| \leq \left(    \frac{\underline{\delta}^2e^{-\bar{K} \bar{\delta}}}{(1 - e^{- \bar{K} \bar{\delta}})^2}  \right)  \sqrt{2N \ln \left( \tfrac{2}{\varepsilon} \right) }.$$
In particular, on this event,
\begin{align*}
& \inf_{\kappa > 0}  \Big( -\frac{\mathrm{d}^2}{\mathrm{d} \kappa^2} \tilde{\ell}_N(\kappa) \Big) \geq \sum_{n=1}^N (1-Y_n) \left(    \frac{\underline{\delta}^2e^{-\bar{K} \bar{\delta}}}{(1 - e^{- \bar{K} \bar{\delta}})^2}  \right) \\
&\qquad \geq \sum_{n=1}^N (1-e^{-\kappa^\ast \delta_n}) \left(    \frac{\underline{\delta}^2e^{-\bar{K} \bar{\delta}}}{(1 - e^{- \bar{K} \bar{\delta}})^2}  \right) - \left| \sum_{n=1}^N (Y_n - e^{-\kappa^\ast \delta_n}) \left(    \frac{\underline{\delta}^2e^{-\bar{K} \bar{\delta}}}{(1 - e^{- \bar{K} \bar{\delta}})^2}  \right) \right| \\
&\qquad \geq  \left(    \frac{\underline{\delta}^2e^{-\bar{K} \bar{\delta}}(1-e^{-\underline K\, \underline{\delta}}) }{(1 - e^{- \bar{K} \bar{\delta}})^2}  \right) N - \left(    \frac{\underline{\delta}^2e^{-\bar{K} \bar{\delta}}}{(1 - e^{- \bar{K} \bar{\delta}})^2}  \right)  \sqrt{2N \ln \left( \tfrac{2}{\varepsilon} \right) }.
\end{align*}
This proves the required statement. 
\end{proof}

\subsubsection{Proof of Proposition~\ref{prop: existence of estimator}}
\begin{proof}

Observe that as $\kappa \to \infty$, $\frac{\mathrm{d}}{\mathrm{d} \kappa} \tilde{\ell}_N(\kappa) \to - \infty$ and as $\kappa \to 0$, $\frac{\mathrm{d}}{\mathrm{d} \kappa} \tilde{\ell}_N(\kappa) \to + \infty$. Hence, the solution exists by continuity. The uniqueness follows from the fact that $ \frac{\mathrm{d}^2}{\mathrm{d} \kappa^2} \tilde{\ell}_N(\kappa) < 0$ for all $\kappa > 0$. 
\end{proof}

\subsubsection{Proof of Proposition~\ref{prop: upper bound score}}
\begin{proof}

Since $\kappa^{\ast} \in [\underline K, \bar K]$, by~\eqref{derivative of the log-likelihood function}, 
\begin{align*}
\frac{\mathrm{d}}{\mathrm{d} \kappa} \Tilde{\ell}_N (\kappa^\ast)  = & \Bigg[\sum_{n=1}^N \Bigg( - \delta_n Y_n + (1-Y_n) \delta_n \frac{e^{-\kappa^\ast \delta_n}}{1 - e^{-\kappa^\ast \delta_n}}\Bigg) + \delta_0 \Big( -1 + \frac{e^{-\kappa^\ast \delta_0}}{1 - e^{-\kappa^\ast \delta_0}} \Big)\Bigg] \\
= & \sum_{n=1}^N \frac{- \delta_n \mathds{1}_{\delta_n < + \infty} }{1 - e^{-\kappa^\ast \delta_n}} \Big( Y_n - e^{-\kappa^\ast \delta_n} \Big) + \delta_0 \Big( -1 + \frac{e^{-\kappa^\ast \delta_0}}{1 - e^{-\kappa^\ast \delta_0}} \Big) \,, 
\end{align*} 
where the last equality comes from the fact that $Y_n = 0$ a.s. when $\delta_n = +\infty$ as discussed in Remark~\ref{rm2}.
Let $f(\delta; \kappa^\ast) = \frac{ - \delta \mathds{1}_{\delta_n < + \infty}}{ 1 - e^{-\kappa^\ast \delta}}$, we know that $f(\delta)$ is bounded given $\delta \in [\underline \delta, \bar \delta] \cup \{ + \infty \}$ and $\sup_{\delta \in [\underline \delta, \bar \delta]} \left| f(\delta; \kappa^\ast) \right| = \left| f(\bar \delta; \kappa^\ast) \right|$. Therefore, by Proposition \ref{prop: concentration rescaling sub-Gaussian}, with $\mathbb P^\ast$-probability at least $1 - \varepsilon$, 
\begin{equation*}
\left| \frac{\mathrm{d}}{\mathrm{d} \kappa} \Tilde{\ell}_N (\kappa^\ast)  \right|  \leq  \left| f(\bar \delta; \kappa^{\ast})\right|  \sqrt{4N\ln(\tfrac{2}{\varepsilon})} + \delta_0 \Big| -1 + \frac{e^{-\kappa^\ast \delta_0}}{1 - e^{-\kappa^\ast \delta_0}} \Big| \,,
\end{equation*}
where $$
C = 2 \left| f(\bar \delta; \kappa^{\ast}) \right| , \quad c = \delta_0 \Big| -1 + \frac{e^{-\kappa^\ast \delta_0}}{1 - e^{-\kappa^\ast \delta_0}} \Big|. 
$$
\end{proof}

\subsubsection{Proof of Theorem~\ref{thm: concentration}} \label{proof theorem thm: concentration}

\begin{proof}
By the mean value theorem, there exists $\lambda \in (0,1)$ such that for $\tilde \kappa = \lambda \kappa_N + (1-\lambda) \kappa^\ast$, 
\begin{align*}
0 = \frac{\mathrm{d}}{\mathrm{d} \kappa} \tilde{\ell}_N(\kappa_N) =\frac{\mathrm{d}}{\mathrm{d} \kappa} \tilde{\ell}_N( \kappa^\ast) + \frac{\mathrm{d}^2}{\mathrm{d} \kappa^2} \tilde{\ell}_N (\tilde \kappa) (\kappa_N - \kappa^\ast)\,.
\end{align*}
Therefore, 
$$ \left| \kappa_N - \kappa^\ast \right| =  \Big( - \frac{\mathrm{d}^2}{\mathrm{d} \kappa^2} \tilde \ell_N(\tilde \kappa)\Big)^{-1} \big| \frac{\mathrm{d}}{\mathrm{d} \kappa} \tilde \ell_N( \kappa^\ast) \big| \leq \Big( \inf_{\kappa > 0} \big(- \frac{\mathrm{d}^2}{\mathrm{d} \kappa^2} \tilde \ell_N(\kappa) \big)\Big)^{-1} \big| \frac{\mathrm{d}}{\mathrm{d} \kappa} \tilde \ell_N( \kappa^\ast) \big|\,.$$
By Proposition \ref{prop: lower bound Fisher} and Proposition \ref{prop: upper bound score}, there exists constants $C,c, C', c' \geq 0$ such that it holds with $\mathbb P^\ast$-probability at least $1-2\varepsilon$ for any $\varepsilon \geq 0$, 
$$ 
\left| \kappa_N - \kappa^\ast \right| \leq  \frac{C'  \sqrt{N \ln \big( \tfrac{2}{\varepsilon} \big)} + c'}{cN - C  \sqrt{N \ln \big( \tfrac{2}{\varepsilon} \big)}}\,.
$$
Let $N_0 = \frac{4C^2}{c^2}$. Then, if $N \geq N_0 \ln(\tfrac{2}{\varepsilon})$, 
\begin{align*}
cN - C  \sqrt{N \ln \big( \tfrac{2}{\varepsilon} \big)} &\geq \frac{cN}{2} + \left( \frac{c}{2} \sqrt{N_0 \ln \big( \tfrac{2}{\varepsilon} \big)} \right) \sqrt{N} - C  \sqrt{N \ln \big( \tfrac{2}{\varepsilon} \big)} = \frac{cN}{2}\,.
\end{align*}
Therefore, it holds with $\mathbb P^\ast$-probability at least $1-2\varepsilon$ such that for any $\varepsilon \geq 0$, 
$$ 
\left| \kappa_N - \kappa^\ast \right| \leq  \frac{2 C'}{c} N^{-1/2} \sqrt{ \ln \big( \tfrac{2}{\varepsilon} \big)} + \frac{2 c'}{c} N^{-1}\,.
$$

\end{proof}

\subsection{Proof of Corollary~\ref{cor}} \label{proof corollary}
\begin{proof}
By choosing ${N_0}$ to be sufficiently large, we can guarantee that if $N \geq N_0 \ln \big( \tfrac{2}{\varepsilon} \big)$, then
$N \geq \tilde{N_0} \ln \big( \tfrac{2 \pi^2 N^2}{3 \varepsilon} \big)$ where $\tilde{N_0}$ is a constant given in Theorem~\ref{thm: concentration}. 

In particular, for such $N$, Theorem~\ref{thm: concentration} holds with $\varepsilon_N = \tfrac{3 \varepsilon}{\pi^2 N^2}$. Let $A_N$ denote the corresponding event. We can see that
$$\mathbb P^\ast \Bigg( \bigcup_{N \geq N_0 \ln \big( \tfrac{2}{\varepsilon} \big)} A_N^c\Bigg) \leq \sum_{N \geq N_0 \ln \big( \tfrac{2}{\varepsilon} \big)} \mathbb P^\ast (A_N^c) \leq \sum_{N \geq N_0 \ln \big( \tfrac{2}{\varepsilon} \big)} 2 \left(\frac{3 \varepsilon}{\pi^2 N^2} \right) \leq \varepsilon.$$ 
This gives the required result. 
\end{proof}

\subsection{Proof of Corollary~\ref{cor2}} \label{proof of cor 2}
\begin{proof}
Let $C, c, N_0 \geq 0$ be the constants from Corollary~\ref{cor} and $N_0' = (\tfrac{c}{C})^2$, then it holds with $\mathbb P^\ast$-probability at least $1-\varepsilon$ such that for any $\varepsilon \geq 0$, 
\begin{equation*}
\begin{split}
\left| \kappa_N - \kappa^\ast \right| & \leq   CN^{-1/2} \sqrt{\ln(\tfrac{2N}{\varepsilon})} + cN^{-1}  \\
& \leq 2 C N^{-1/2} \sqrt{\ln(\tfrac{2N}{\varepsilon})} \qquad \text{for all } N \geq \max\left( N_0 \ln(\tfrac{2}{\varepsilon}), N_0' / \ln(\tfrac{2}{\varepsilon}) \right) \, .
\end{split}
\end{equation*}
For such $N$, we have 
\begin{equation*}
\kappa_N \leq \kappa^\ast + 2 C N^{-1/2} \sqrt{\ln(\tfrac{2N}{\varepsilon})} \,, 
\end{equation*}
and 
\begin{equation*}
\kappa_N \geq \kappa^\ast -  2 C N^{-1/2} \sqrt{\ln(\tfrac{2N}{\varepsilon})} \, .
\end{equation*}
Let $N_1' = \max \left( \bar K - \kappa^\ast, \kappa^\ast - \underline K\right)$ and $N_1 = (\tfrac{N_1'}{2 C})^2$. Then it holds with $\mathbb P^\ast$-probability at least $1-\varepsilon$ such that for any $\varepsilon \geq 0$, 
\begin{equation*}
\kappa_N \in[\underline K, \bar K] \quad \text{ for all } N 
\geq \max \left( \ln(\tfrac{2}{\varepsilon}) / (N_1 -1) , N_0 \ln(\tfrac{2}{\varepsilon}), N_0' / \ln(\tfrac{2}{\varepsilon}) \right) \, ,
\end{equation*}
which completes the proof. 
\end{proof}

\subsection{Proof of Proposition~\ref{X_NT bound in high prob}} \label{proof prop X_NT bound in high prob}
On the event that Corollary~\ref{cor} and Corollary~\ref{cor2} hold, we can see that
\small
\begin{equation*}
\begin{split}
X_{N_T} & \leq  \sum_{n=0}^{\lf  N' \rf} (\tau_{n+1} - \tau_n)  |\underline K - \bar K|^2  \\
& \qquad + \sum_{n = \lceil  N' \rceil}^{N_T} (\tau_{n+1} - \tau_n) \left( C n^{-1/2} \sqrt{ \ln \big( \tfrac{2n}{\varepsilon} \big)} + cn^{-1} \right)^2 \\
& \leq C'  \sum_{n=0}^{\lf  N' \rf} (\tau_{n+1} - \tau_n) + C  \sum_{n = \lceil  N' \rceil}^{N_T} (\tau_{n+1} - \tau_n) \left( \frac{\ln \big( \tfrac{2}{\varepsilon} \big)}{n} + \frac{\ln n}{n} \right) \,, 
\end{split}
\end{equation*}
\normalsize
where $N' = \max \Big(N_0 \ln(\tfrac{2}{\varepsilon}), N_0' / \ln(\tfrac{2}{\varepsilon}) \Big)$. 
By choosing $C$ to be sufficiently large, we have
\small
\begin{equation*}
\begin{split}
X_{N_T} & \leq C'(\tau_{1} - \tau_0) + C'  \sum_{n=1}^{\lf N'\rf} (\tau_{n+1} - \tau_n)  \\ 
& \quad - C  \sum_{n = 1}^{\lf  N' \rf} (\tau_{n+1} - \tau_n) \left( \frac{\ln \big( \tfrac{2}{\varepsilon} \big)}{n} + \frac{\ln n}{n} \right)  + C  \sum_{n=1}^{N_T} (\tau_{n+1} - \tau_n) \left( \frac{\ln \big( \tfrac{2}{\varepsilon} \big)}{n} + \frac{\ln n}{n} \right) \\ 
& \leq C'(\tau_{1} - \tau_0) + \sum_{n = 1}^{\lf N' \rf} (\tau_{n+1} - \tau_n)  \left(C' - C \left(\frac{\ln \big( \tfrac{2}{\varepsilon} \big)}{n} + \frac{\ln n}{n} \right) \right)  \\ 
& \quad + C  \sum_{n = 1}^{N_T} (\tau_{n+1} - \tau_n) \left( \frac{\ln \big( \tfrac{2}{\varepsilon} \big)}{n} + \frac{\ln n}{n} \right) \\
& \leq C(\tau_{1} - \tau_0) + C  \sum_{n = 1}^{N_T} (\tau_{n+1} - \tau_n) \left( \frac{\ln \big( \tfrac{2}{\varepsilon} \big)}{n} + \frac{\ln n}{n} \right) \, ,
\end{split}
\end{equation*}
\normalsize
where the last inequality comes from $C' - C \left(\frac{\ln \big( \tfrac{2}{\varepsilon} \big)}{n} + \frac{\ln n}{n} \right) \leq C' - C \frac{\ln \big( \tfrac{2}{\varepsilon} \big)}{n} \leq 0$ for any $1 \leq n \leq N'$ under a sufficiently large $C$.  

Next we need the following lemma which is proved in \cite[Lemma 3.1]{szpruch2024optimal}. 
\begin{lemma}
\label{lem}
Let $Y_n$ be an IID sub-exponential random variable with mean $0$ and $(\rho_n) \subseteq \sR$. Then there exists a constant $C$ such that for any $\varepsilon > 0$ and $N \in \sN$
$$\sP \left( \left| \sum_{n=1}^N Y_n \rho_n \right| \geq C \ln \big( \tfrac{2}{\varepsilon} \big)   \sqrt{\sum_{n=1}^N \rho_n^2}  \right)   \leq \varepsilon.$$
\end{lemma}

By using the above result applying to $\varepsilon_N \propto \varepsilon/N^2$ and taking the countable union of the above events, we have
$$\sP \left( \left| \sum_{n=1}^N Y_n \rho_n \right| \leq C \ln \big( \tfrac{2N}{\varepsilon} \big) \sqrt{\sum_{n=1}^N \rho_n^2}   \quad \text{for all} \quad N \in \sN \right) \geq 1 - \varepsilon \, .$$

Now, we note that $(\tau_{n+1} - \tau_n - \frac{1}{\lambda^+ + \lambda^-})$ is sub-exponential with mean $0$.
Hence, on the event that Corollary \ref{cor} and Lemma \ref{lem} hold with $Y_n = \tau_{n+1} - \tau_n - \frac{1}{\lambda^+ + \lambda^-}$, we know that with probability at least $1-2\varepsilon$, it holds that
\small
\begin{equation} \label{order of NT}
\begin{split}
(\tau_{1} -  \tau_0) + & \sum_{n = 1}^{N_T} (\tau_{n+1}  - \tau_n) \left( \frac{\ln \big( \tfrac{2}{\varepsilon} \big)}{n} + \frac{\ln n}{n} \right) \\
& = \frac{1}{\lambda^+ + \lambda^-} \left( 1 + \sum_{n = 1}^{N_T} \left( \frac{\ln \big( \tfrac{2}{\varepsilon} \big)}{n}  + \frac{\ln n}{n} \right) \right) \\ 
& \qquad + \sum_{n = 1}^{N_T} (\tau_{n+1} - \tau_n - \frac{1}{\lambda^+ + \lambda^-}) \left( \frac{\ln \big( \tfrac{2}{\varepsilon} \big)}{n} + \frac{\ln n}{n} \right) + (\tau_{1} - \tau_0 - \frac{1}{\lambda^+ + \lambda^-})\\
& \leq  \frac{1}{\lambda^+ + \lambda^-} + \frac{1}{\lambda^+ + \lambda^-} \ln N_T \ln \big( \tfrac{2}{\varepsilon} \big) + \frac{1}{2(\lambda^+ + \lambda^-)} \ln^2 N_T  \\  
& \qquad + C \ln \big( \tfrac{2 (N_T+1)} {\varepsilon} \big) \sqrt{1 + \sum_{n = 1}^{N_T} \left( \frac{\ln \big( \tfrac{2}{\varepsilon} \big)}{n} + \frac{\ln n}{n} \right)^2} \\ 
& \leq  \frac{1}{\lambda^+ + \lambda^-} + \frac{1}{\lambda^+ + \lambda^-} \ln N_T \ln \big( \tfrac{2}{\varepsilon} \big) + \frac{1}{2(\lambda^+ + \lambda^-)} \ln^2 N_T  \\  
& \qquad + C \big(\ln (2 N_T) + \ln \big( \tfrac{2} {\varepsilon} \big) \big) \sqrt{3 + \left(\ln \left(\tfrac{2}{\varepsilon}\right) \right)^2 + 2 \ln \left(\tfrac{2}{\varepsilon}\right)} \\ 
& \leq \left( \frac{1}{2(\lambda^+ + \lambda^-)} + C \sqrt{3} \right) \ln^2 N_T + \left( \frac{1}{\lambda^+ + \lambda^-} + C \right) \ln N_T \ln \big( \tfrac{2}{\varepsilon} \big) \\
& \qquad + C (\sqrt{3} + 1 + \ln 2) \ln^2  \big( \tfrac{2}{\varepsilon} \big) + \left( \frac{1}{\lambda^+ + \lambda^-} + C \sqrt{3} \ln 2 \right) \, ,
\end{split}
\end{equation}
\normalsize
where the last inequality uses the fact that $\ln N_T \leq \ln^2 N_T$ for large $N_T$ and $\ln  \big( \tfrac{2}{\varepsilon} \big)  \leq \ln^2  \big( \tfrac{2}{\varepsilon} \big)$ for small $\varepsilon$. Note that the constant $C$ comes from Lemma~\ref{lem} which is independent of $\varepsilon$, 
hence the result. 

\bibliographystyle{plain}
\bibliography{Bibliography}

\end{document}